\documentclass[12pt]{amsart}
\usepackage{amsfonts,amssymb,amscd,xy,mathtools,tikz-cd}
\usepackage[leqno]{amsmath}
\usepackage{mathrsfs}
\usepackage{amssymb}
\usepackage{amsmath}
\usepackage{amsxtra}
\usepackage{amsthm}
\newcommand{\Aut}{\mathrm{Aut}}
\newcommand{\Hom}{\mathrm{Hom}}
\newcommand{\id}{\mathrm{Id}}

\usepackage[utf8]{inputenc}
\usepackage{amsmath}
\usepackage{calc}

\usepackage{accents}

\newcommand{\rkk}{R_{\e k^{\le\prime}\!/k}}

\usepackage{mathtools}

\def\br{{\rm{Br}}\e}

\usepackage{rotating}
\usepackage{anyfontsize}
\fontsize{1}{2}

\DeclareMathAlphabet{\mathbbmsl}{U}{bbm}{m}{sl}

\newcommand{\car}{{\rm char}\,}

\newcommand{\sks}{k_{\lbe\fl}^{\sim}}

\newcommand{\res}{{\rm res}}

\def\sm{\lbe{\rm sm}}

\def\sab{\lbe{\rm sab}}
\def\G{\mathbb{G}}
\newcommand{\ext}{\mathrm{Ext}_{\le\sks}}

\def\s{\mathscr }

\newcommand\sfs{S_{\fl}^{\le\sim}}

\newcommand{\into}{\hookrightarrow}
\newcommand{\onto}{\twoheadrightarrow}
\newcommand{\isoto}{\overset{\!\sim}{\to}}

\newcommand{\Isoto}{\overset{\!\sim}{\lra}}

\DeclareRobustCommand\longtwoheadrightarrow
     {\relbar\joinrel\twoheadrightarrow}

\newcommand{\kp}{k^{\le\prime}}

\newcommand\ck{\mathcal C_{k}}
\newcommand\ak{\mathcal A_{\le k}}

\usepackage{epsfig,amssymb}
\usepackage{bm} 
\usepackage{verbatim} 

\definecolor{labelkey}{rgb}{1,0,0}

\usepackage{verbatim} 

\makeatletter
\def\bm#1{\mathpalette\bmstyle{#1}}
\def\bmstyle#1#2{\mbox{\boldmath$#1#2$}}

\DeclareMathAlphabet{\mathcalligra}{T1}{calligra}{m}{n}

\numberwithin{equation}{section}

\newcommand{\sh}{\kern -.4em\phantom{a}^{\mathbf{\sim}}}

\newcommand{\lra}{\longrightarrow}

\newcommand{\fl}{\le{\rm fl}}

\newcommand{\et}{{\rm {\acute et}}}

\newcommand{\red}{{\rm red}}

\def\be{\kern -.1em}
\def\le{\kern 0.03em}
\def\lle{\kern 0.015em}
\def\lbe{\kern -.05em}

\newcommand{\C}{\mathbb C}
\newcommand{\Z}{{\mathbb Z}}
\newcommand{\Q}{{\mathbb Q}}
\newcommand{\N}{{\mathbb N}}
\newcommand{\R}{{\mathbb R}}
\newcommand{\T}{{\mathbb T}}

\newcommand{\spec}{\mathrm{ Spec}\,}

\newcommand{\krn}{\mathrm{Ker}\,}
\newcommand{\img}{\mathrm{Im}\e\le}
\newcommand{\coim}{\mathrm{Coim}\e\le}
\newcommand{\cok}{\mathrm{Coker}\,}

\def\e{\kern 0.08em}

\newcommand{\alb}{{\rm alb}}
\newcommand{\Alb}{{\rm Alb}}

\newcommand{\picom}{{\rm Pic}_{\lbe X/k}^{\le \sm,\e 0}}

\usepackage{tabularx}
\usepackage{epsfig,amssymb}
\usepackage{bm} 
\usepackage{verbatim} 
\usepackage{mathrsfs}

\def\pic{{\rm{Pic}}\e}

\usepackage{hyperref}

\definecolor{labelkey}{rgb}{1,0,0}

\makeatletter
\def\bm#1{\mathpalette\bmstyle{#1}}
\def\bmstyle#1#2{\mbox{\boldmath$#1#2$}}

\DeclareMathAlphabet{\mathcalligra}{T1}{calligra}{m}{n}

\numberwithin{equation}{section}

\newcommand{\isomt}{{\mathscr I\! so\le m}_{\e \mathcal{O}_{\le T}}\lbe}

\newcommand{\isomv}{{\mathscr I\! so\le m}_{\e \mathcal{O}_{\le V}}\lbe}

\newcommand{\shom}{{\mathscr H\! o\le m}_{\le S_{\fl}^{\lle\sim}}\lbe}

\newcommand{\shomk}{{\mathscr H\! o\le m}_{\le k_{\fl}^{\lle\sim}}\lbe}

\newcommand{\shoms}{{\mathscr H\! o\le m}_{\le C^{\e[-1,0]}\lbe(\lbe S_{\lbe\fl} \lbe)}}
\newcommand{\shomt}{{\mathscr H\! o\le m}_{\le C^{\e[-1,0]}\lbe(T_{\lbe\fl})}}
\newcommand{\exts}{{\mathscr E\be x\le t}_{S_{\fl}^{\lle\sim}}}

\newcommand{\extk}{{\mathscr E\be x\le t}_{ k_{\lbe\fl}^{\lle\sim}}}

\newcommand{\extc}{{\mathscr E\be x\le t}_{ D^{\e[-1,0]}\lbe( k_{\lbe\fl})}}

\newcommand{\exct}{{\mathrm{Ext}}_{\le D^{\e[-1,0]}\lbe( T_{\lbe\fl})}}
\newcommand{\extcs}{{\mathscr E\be x\le t}_{ D^{\e[-1,0\lle]}\lbe(S_{\lbe\fl})}}

\newcommand{\exttc}{{\mathscr E\be x\le t}_{ D^{\e[-1,0\lle]}\lbe( T_{\lbe\fl})}}

\def\be{\kern -.1em}
\def\le{\kern 0.055em}
\def\lle{\kern 0.015em}
\def\lbe{\kern -.03em}

\newcommand{\gmk}{\G_{m,\le k}}
\newcommand{\gak}{\G_{\lbe a,\le k}}

\usepackage{mathtools}

\newtheorem{lemma}{Lemma}[section]
\newtheorem{theorem}[lemma]{Theorem}
\newtheorem{proposition-definition}[lemma]{Proposition-Definition}
\newtheorem{theorem-definition}[lemma]{Theorem\le-Definition}
\newtheorem{corollary}[lemma]{Corollary}
\newtheorem{proposition}[lemma]{Proposition}
\theoremstyle{definition}
\newtheorem{definition}[lemma]{Definition}

\newtheorem{lemma-definition}[lemma]{Lemma-Definition}
\theoremstyle{remark}
\newtheorem{remark}[lemma]{Remark}
\newtheorem{remarks}[lemma]{Remarks}
\newtheorem{example}[lemma]{Example}
\newtheorem{examples}[lemma]{Examples}

\begin{document}

\input xy     
\xyoption{all} 

\title[Local duality theorems for commutative algebraic groups]{Local duality theorems for commutative algebraic groups}

\author{Cristian D. Gonz\'alez\e-\le Avil\'es}
\address{Departamento de Matem\'aticas, Universidad de La Serena, Cisternas 1200, La Serena 1700000, Chile}
\email{cgonzalez@userena.cl}
\thanks{This research was partially supported by Fondecyt grant 1200118.}


\topmargin -1cm

\smallskip

\begin{abstract} Let $k$ be a non-archimedean local field and let $G$ be a commutative algebraic $k$-group that contains no nontrivial unipotent subgroups if $\car k=0$. Given any presentation of $G$ as an extension of commutative algebraic $k$-groups $\mathcal E\colon 0\to L\to G\to A\to 0$, where $L$ is affine and $A$ is an abelian variety, we establish Pontryagin duality theorems that relate the fppf cohomology groups of $G$ (equipped with the extension structure $\mathcal E\e$) to the hypercohomology groups of the dual of $G$ in the truncated derived category $D^{\e[-1,0\le]}\lbe(k_{\fl})$.  These results have applications (to be discussed elsewhere) to certain generalizations of Lichtenbaum\le-van Hamel duality to a class of non-smooth proper $k$-varieties.
\end{abstract}

\maketitle

\section*{0. Introduction}
In \cite{lich}, Lichtenbaum established the existence of a perfect pairing of abelian groups $\pic X\times \br X\to \br k\simeq\Q/\Z\e$, where $k$ is a $p\e$-adic field and $X$ is a proper, smooth and geometrically connected $k\e$-\le curve. This duality theorem was
extended by van Hamel to proper and smooth $k$-schemes of arbitrary dimension in \cite{vhd}. In both works, a key ingredient is Tate's duality theorem for abelian varieties over $k$ \cite[Corollary I.3.4]{adt}, which is applied to the Picard variety of $X$ over $k$ (the latter is, indeed, an abelian variety by the  smoothness of $X\le$). Now, under the indicated properness and smoothness hypotheses, the arguments in \cite{lich} and \cite{vhd} remain valid if the $p\e$-adic field $k$ is replaced by a local function field, i.e., a finite extension of $\mathbb{F}_{\! p}\lbe(\!(t)\!)$, and Tate duality is replaced by Milne\,--Tate duality \cite[Theorem III.7.8]{adt}. However, extending Lichtenbaum\e--\e van Hamel duality to non-proper or non-smooth $k$-schemes over an arbitrary non\le-\le archimedean local field $k$ requires new methods. In the case that $X$ is proper but not smooth, the Picard variety $\picom$ of $X$ over $k$ (i.e., the largest smooth and connected $k$-subgroup of the Picard scheme of $X$ over $k$) is not an abelian variety in general. Thus, to extend Lichtenbaum\e--\e van Hamel duality to proper but non-smooth $k$-schemes, a generalization of Milne\,--Tate duality is required that can describe the Pontryagin dual $H^{\le 1}\lbe(k_{\le\et},\picom\le)^{*}$ of the discrete and torsion group $H^{\le 1}\lbe(k_{\le\et},\picom\le)$ for {\it any} proper and geometrically integral $k$-scheme $X$ over an arbitrary non\le-\le archimedean local field $k$. By work of Harari and Szamuely \cite[Theorem 2.3]{hsz1}, when $G=\picom$ is a semiabelian $k$-variety, i.e., an extension $0\to T\to G\to A\to 0$, where $T$ is a $k$-torus and $A$ is an abelian $k$-variety\,\footnote{\e This is the case, for example, if $X$ is a geometrically nodal $k$-curve.}\,, the profinite group $H^{\le 1}\lbe(k_{\le\et},G\e)^{*}$ is canonically isomorphic to the completion $H^{\le 0}\lbe(k_{\le\et},M\lle)^{\wedge}$ of the hypercohomology group $H^{\le 0}\lbe(k_{\le\et},M\lle)$ relative to the family of open subgroups of finite index for a suitable topology on $H^{\le 0}\lbe(k_{\le\et},M\lle)$, where $M=[\e T^{\lle D}\!\!\to\!\lbe A^{\vee}\e]$ is the Deligne $1$-motive dual to $G$. Here $T^{\lle D}$ denotes the Cartier dual of $T$ placed in degree $-1$ and $A^{\vee}={\rm Pic}^{\le 0}_{\lbe A/k}$ is the abelian $k$-variety dual to $A$. To go beyond the case of semiabelian varieties, we extend the category of Deligne $1$-motives following \cite{rus13}. For any field $k$, let $\s{M}_{1}$ be the category of triples
\begin{equation}\label{mo}
M=(K^{\lbe D}\be,u\e,\mathcal E(L,\iota, G,\pi, A)\le),
\end{equation}
where $K$ and $L$ are affine and commutative algebraic $k$-groups that contain no nontrivial unipotent $k$-subgroups if $\car k=0\,$\footnote{\e The reason for this restriction is explained in Remark \ref{imp0}\e.}\,, $A$ is an abelian $k$-variety, $u\colon K^{\lbe  D}\to G$ is a morphism in $\sks$ and
\begin{equation}\label{fext}
\mathcal E=\mathcal E(L,\iota, G,\pi, A)\colon 0\to L\to G\to A\to 0
\end{equation}
is an extension of commutative algebraic $k$-groups. Neither connectedness nor smoothness of $G$ is assumed. The morphisms in $\s{M}_{1}$ are described in Definition \ref{1mot} below. Recall that, by \cite[Theorem 2.3(i)]{bri17b}, every commutative algebraic $k$-group $G$ is
an extension of the form \eqref{fext} (in general, in more than one way). Any such extension $\mathcal E$ defines a $1$-motive
\begin{equation}\label{pur}
M(\mathcal E\le)=(0\e,0\e,\mathcal E(L,\iota, G,\pi, A)\le)
\end{equation}
whose dual, by definition, is the $1$-motive
\begin{equation}\label{mvd}
\hskip .8cm M(\mathcal E\le)^{\le\diamond}=(L^{\be D}\!,\partial_{\le\mathcal E},\mathcal E(0,0,A^{\vee},{\rm Id},A^{\vee}\lle)),
\end{equation}
where $\partial_{\le\mathcal E}\colon L^{\be D}\to A^{\vee}$ is defined as follows: if $\chi\in L^{\be D}=\shomk( L,\G_{m,\le k})$ and $\beta_{\lbe A}\colon \extk^{1}\be(A,\G_{m,\le k})\isoto A^{\vee}$ is the Barsotti\e--\e Weil isomorphism for $A$, then $\partial_{\le\mathcal E}(\le\chi\lbe)=\beta_{\lbe A}\lbe(\e G_{\le \chi})\in A^{\vee}$, where $G_{\le \chi}\in \extk^{1}\be(A,\G_{m,\le k})$ is the pushout of $\mathcal E$ along $\chi\colon L\to \G_{m,\le k}$.

\smallskip

Now, consider the functor
\[
\s{M}_{1}\to D^{\e[-1,0\le]}\lbe(k_{\fl}), M\mapsto [\le M\le],
\]
where $[\le M\le]=C(u)=[\le K^{\lbe D}\!\overset{\!u}{\to} G\e\le]$ is the mapping cone of $u$. In contrast to the case of Deligne $1$-motives, the above functor is neither faithful nor full, and this fact creates a number of difficulties. For example, we cannot prove that the assignment $M\mapsto M^{\le\diamond}$, where $M^{\le\diamond}$ is given by \eqref{mdu}, defines a {\it functor} on $\s{M}_{1}$. Thus we will not attempt to make sense of the statement 
\begin{equation}\label{dou}
\textrm{``\e there exists a functorial isomorphism $M\isoto M^{\lle\diamond\le\diamond}\,$"}.
\end{equation}
Nevertheless, writing $D(M\le)$ for the dual of $[\le M\le]$ in $D^{\e[-1,0\le]}\lbe(k_{\fl})$, we {\it can} prove that there exists an isomorphism in $D^{\e[-1,0\le]}\lbe(k_{\fl})$
\[
\theta_{M}\colon D(M\le)\isoto [\lle M^{\le\diamond}\e]
\]
that is {\it uniquely} determined by $M$ if $( L\otimes K\le)^{D}(k)=0$. In particular, since the $1$-motives \eqref{pur} and \eqref{mvd} clearly satisfy the preceding condition, the extension $\mathcal E$ \eqref{fext} uniquely determines an isomorphism $\theta_{\mathcal E}\colon D(G\le)\isoto C(\partial_{\le\mathcal E})$ in $D^{\e[-1,0\le]}\lbe(k_{\fl})$ that, in turn, induces an isomorphism of abelian groups
\begin{equation}\label{btet2}
H^{\le 0}\lbe(k,\theta_{\mathcal E})\colon H^{\le 0}\lbe(k,D(G\e))\isoto H^{\le 0}\lbe(k,C(\partial_{\le\mathcal E})).
\end{equation}
Now, in order to verify the symmetry of our pairings, we establish the reflexivity of the objects $[\le M\le]\in D^{\e[-1,0\le]}\lbe(k_{\fl})$, i.e., we prove that the natural evaluation morphism $[\le M\le]\to DD(M)$ is an isomorphism in $D^{\e[-1,0\le]}\lbe(k_{\fl})$. This weaker assertion than the (\le possibly invalid) statement \eqref{dou} suffices for our purposes. In particular, in a preliminary step, we show in section \ref{abrs} that abelian $S\le$-\le schemes are reflexive objects of $D^{\e[-1,0\le]}\lbe(S_{\fl})$ for any scheme $S$. Here we call attention to the fact that the latter statement, which was asserted without proof in \cite[Proposition 4.7(2)]{bro21}, is {\it not} identical with (or trivially equivalent to) the biduality theorem for abelian schemes. In fact, the stated reflexivity and biduality of abelian schemes have often been conflated in the literature.
For example, the reflexivity of Deligne 1-motives, i.e., the assertion that \eqref{dou} is both meaningful and true for such 1-motives, requires, in part, a key verification on the abelian scheme piece (namely, the commutativity of diagram \eqref{key2a} below) to be valid. See Remark \ref{demot} for more details.

\smallskip

Now, the proof of the Harari\e--\e Szamuely local duality theorems combines, via standard {\it d\'evissages}\e, Milne\,--\,Tate duality for abelian $k$-varieties, Tate\,--\e Nakayama duality for $k$-tori \cite[Corollary I.2.4]{adt} and Poitou\,--Tate duality for finite $k\e$-\le groups over a $p\e$-adic field $k$ \cite[Corollary I.2.3]{adt}. The latter two classical duality theorems have been extended by Rosengarten \cite{ros23} to the class of {\it all} affine commutative algebraic $k$-groups over an arbitrary local field $k$. See Theorem \ref{ros23} for the precise statements. In this paper we combine, via appropriate {\it d\'evissages}, Milne\,--Tate duality and Rosengarten local duality to partly generalize \cite[Theorem 2.3]{hsz1} to the full subcategory of $\s{M}_{1}$ whose objects are of the form \eqref{pur} or \eqref{mvd}. As in \cite[\S2]{hsz1}, a key part of the proof is to suitably topologize the fppf hypercohomology group $H^{\le 0}\lbe(k,D(G\e))\simeq H^{\le 0}\lbe(k,C(\partial_{\le\mathcal E}))$ so that the associated maximal Hausdorff quotient $H^{\le 0}\lbe(k,C(\partial_{\le\mathcal E}))_{\rm Haus}$ is topologically isomorphic to $H^{\le 1}\lbe(k,G\e)^{*}$. Now $H^{\le 0}\lbe(k,C(\partial_{\le\mathcal E}))$ fits into an exact sequence of abelian groups
\[
A^{\le\vee}\lbe\lbe(k)\xrightarrow{f^{\lle(0)}} H^{\le 0}\lbe(k,C(\partial_{\le\mathcal E}))\xrightarrow{g^{\lle(0)}}H^{\le 1}\lbe(k,L^{\be D}\e),
\]
where $A^{\vee}\lbe(k)$ is equipped with the topology induced by that of $k$ and $H^{\le 1}\lbe(k,L^{\be D}\le)$ is equipped with the topology introduced in \cite[\S3.6]{ros23}. To define the required topology on 
$H^{\le 0}\lbe(k,C(\partial_{\le\mathcal E}))$, we proceed as follows. First, we show that the morphism of abelian groups $H^{\le 0}\lbe(k,C(\partial_{\le\mathcal E}))\to\Hom(H^{\le 1}\lbe(k,G\e), \Q/\Z\e)$ induced by the cup product pairing $H^{\le 1}\lbe(k,G\e)\times H^{\le 0}\lbe(k,D(G\e))\to \Q/\Z$ and the isomorphism \eqref{btet2}  factors through a map $\gamma\colon H^{\le 0}\lbe(k,C(\partial_{\le\mathcal E}))\to H^{\le 1}\lbe(k,G\e)^{*}$, where $H^{\le 1}\lbe(k,G\e)^{*}=\Hom_{\e\rm cts.}(H^{\le 1}\lbe(k,G\e), (\Q/\Z\e)_{\rm disc})$. Then we show that $\gamma$ is surjective and fits into an exact and commutative diagram of abelian groups
\begin{equation}\label{hep}
\xymatrix{0\ar[r]&\coim f^{(0)}\ar[r]^(.48){\widetilde{f^{(0)}}}\ar@{->>}[d]&
H^{\le 0}\lbe(k,C(\partial_{\le\mathcal E}))\ar[r]^(.56){\widetilde{g^{(0)}}}\ar@{->>}[d]^(.45){\gamma}&\img g^{(0)}\ar[d]^(.45){\simeq}\ar[r]& 0\\
0\ar[r]&\coim((\pi^{\lle (\lbe 1\lbe)})^{*})\ar[r]^(.52){\widetilde{(\lbe\pi^{(\lbe 1\lbe)}\lbe)^{\lbe*}}}&
H^{\le 1}\lbe(k,G\e)^{*}\ar[r]^(.52){\widetilde{(\lbe\iota^{(\lbe 1\lbe)}\lbe)^{\lbe*}}}&\img((i^{\lle (\lbe 1\lbe)})^{*})\ar[r]&0,
}
\end{equation}
where $\coim f^{\lle(0)}=A^{\vee}\lbe(k)/\e\krn f^{\lle(0)}$ and
$\coim((\pi^{\lle (\lbe 1\lbe)})^{*})=H^{\le 1}\lbe(k,A\le)^{*}/\e\e\krn (\pi^{\lle (\lbe 1\lbe)})^{*}$ are equipped with the corresponding quotient topologies and $\img g^{\lle(0)}\subset H^{\le 1}\lbe(k,L^{\be D})$ and $\img((i^{\lle (\lbe 1\lbe)})^{*})\subset H^{\le 1}\lbe(k,L\le)^{*}$ are equipped with the corresponding subspace topologies. The first vertical map in \eqref{hep} is continuous and strict, the third vertical map is a topological isomorphism and the bottom row is a topological extension. Then we equip $H^{\le 0}\lbe(k,C(\partial_{\le\mathcal E}))$ with the initial topology relative to $\gamma$, i.e., its open sets are exactly the sets $\gamma^{-1}(U)$, where $U$ is open in $H^{\le 1}\lbe(k,G\e)^{*}$. It then follows that \eqref{hep} is an exact and commutative diagram of locally compact abelian topological groups and continuous homomorphisms and the maps $\gamma$ and $\widetilde{g^{\lle(0)}}$ (but possibly not $\widetilde{f^{(0)}}\e$) are strict.

\medskip

We can now state the main result of this paper.

\begin{theorem} \label{th-main} Let $k$ be a non-archimedean local field and let $G$ be a commutative algebraic $k$-group that is an extension $\mathcal E\colon 0\to L\to G\to A\to 0$, where $A$ is an abelian $k$-variety and $L$ is an affine algebraic $k$-group that contains no nontrivial unipotent $k$-subgroups if $\car k=0$. Let $D(G\e)\in D^{\e[-1,0\le]}\lbe(k_{\fl})$ be the truncated dual of $G$ and set $H^{\le -1}_{\lbe\wedge\be}\be\le(k,D(G\e))\overset{{\rm def.}}{=}\krn\be[\e \partial_{\le\mathcal E}\lbe(k)^{\wedge}\colon\! L^{\be D}\be\lle(k)^{\wedge}\!\to\! A^{\vee}(k)^{\wedge}\e]$. Then there exist continuous perfect pairings of Hausdorff and locally compact abelian topological groups
\begin{enumerate}
\item[(i)] $G\lbe(k)^{\wedge}\times H^{\le 1}\lbe(k,D(G\e))\to\Q/\lle\Z\e$,
\item[(ii)] $H^{\le 1}\lbe(k,G\e)\times H^{\le 0}\lbe(k,D(G\e))_{\rm Haus}\to\Q/\lle\Z$ and
\item[(iii)] $H^{\le 2}\lbe(k,G\e)\times H^{\le -1}_{\lbe\wedge\be}\be\le(k,D(G\e))\to\Q/\lle\Z$.
\end{enumerate}
\end{theorem}

In part (ii), both groups are second countable and $H^{\le 1}\lbe(k,G\e)$ is torsion. Further, $H^{\le 0}\lbe(k,D(G\e))_{\rm Haus}$ is a topological extension
\[
0\lra A^{\vee}\lbe(k)/\,\e\overline{\partial_{\lle\mathcal E}\lbe(k)}\lra H^{\le 0}\lbe(k,D(G\e))_{\rm Haus}\lra \krn\be[H^{\e 1}\lbe(k,L^{\! D}\le)\xrightarrow{\partial_{\lle\mathcal E}^{\lle(1)}} H^{\le 1}\lbe(k,A^{\vee}\le)]\lra 0,
\]
where the left-hand group is profinite and the right-hand group is locally profinite and of finite exponent. If $G$ is smooth, then $H^{\le 1}\lbe(k,G\e)$ is discrete and countable and $H^{\le 0}\lbe(k,D(G\e))_{\rm Haus}$ is profinite. 

\smallskip

Theorem \ref{th-main} only partly generalizes \cite[Theorem 2.3]{hsz1}. A full generalization of \cite[Theorem 2.3]{hsz1}, i.e., establishing duality theorems for $H^{\le r}\lle\be(k,M\lle)$ for $r\leq 2$ and arbitrary $1$-motives $M$, requires more work and will be discussed elsewhere. In contrast, the following complete generalization of part (iii) of Theorem \ref{th-main} follows from the {\it proof} of part (i) of the theorem:

\begin{theorem} \label{main2} {\rm(= Theorem \ref{-1bis})} Let $k$ be a non-archimedean local field and let $M=(K^{\be\lle D}\be,u\e,\mathcal E(L,\iota, G,\pi, A)\le)$ be an arbitrary $1$-motive with corresponding dual $1$-motive $M^{\lle\diamond}=(L^{\be D}, u^{\lbe\vee}, \mathcal E(K,\iota^{\lbe\vee}, G^{\e\vee},\pi^{\lbe\vee}, A^{\vee}\le))$. Set $H^{\le 2}\lbe(k,M\le)=H^{\le 2}\lbe(k,[\le M\le]\lle)$ and $H^{\le -1}_{\lbe\wedge\be}(k,M^{\lle\diamond})=\!\krn\be[\e u^{\lbe\vee}\lbe\lbe(k)^{\wedge}\colon\! L^{\be D}\lbe\lbe(k)^{\wedge}\to G^{\e\vee}\lbe(k)^{\wedge}\le]$. Then there exists a continuous perfect pairing of abelian topological groups
\[
H^{\le 2}\lbe(k,M\lle)\times H^{\le -1}_{\lbe\wedge\be}(k,M^{\lle\diamond}) \to\Q/\Z,
\]
where $H^{\le 2}\lbe(k,M\lle)$ is discrete and torsion and $H^{\le -1}_{\lbe\wedge\be}(k,M^{\lle\diamond})$ is profinite.
\end{theorem}

\section*{Use of AI disclosure}

I have used AI in several places, namely in the proofs of Lemmas \ref{ref1}, \ref{compat}, \ref{sbro}, \ref{pud} and \ref{col} and in the proofs of Propositions \ref{xxx} and \ref{ghp}. These proofs are essentially my own ``written in collaboration with AI", except for the proof of Proposition \ref{ghp}\,, which is a machine\e-\e generated proof that I have thoroughly checked for errors and edited for clarity (the ideas contained in this proof are part of ``standard Brauer\le-\le Manin theory" and some of them can be traced to \cite{uet12} and \cite[proof of Proposition 8.2.9]{poo17}.)

\section*{Acknowledgements}

I am very grateful to one of the referees for pointing out an error in an earlier version of this paper. My former student Felipe Rivera\e--\e Mesas made a key suggestion that eventually led me to a new approach that allowed to circumvent this error. I thank him heartily for his help. I also thank Sylvain Brochard for suggesting and sketching a proof of Lemma \ref{sbro}\,, and Michel Brion, Brian Conrad, David Harari, Gabriel Ribeiro, Zev Rosengarten, Ki-Seng Tan, Fabien Trihan and MathOverflow user YCor for additional helpful comments and encouragement. This research was funded by Fondecyt grant 1200118.

\section{Preliminaries}\label{prem}

If $S$ is a scheme, we will write $S_{\fl}$ for the category of {\it all} $S$-schemes equipped with the fppf topology, i.e., $S_{\fl}=(\textrm{Sch}/\lbe S\le)_{\textrm{fppf}}$. The category of abelian sheaves on $S_{\fl}$ will be denoted by $\sfs$.

\smallskip

\begin{lemma}\label{ker-cok} Let $\mathcal A$ be an abelian category and let $f$ and $g$ be morphisms in $\mathcal A$ such that $g\be\circ\!\be f$ is defined. Then there exists a canonical exact sequence in $\mathcal A$ 
\[
0\to\krn f\to\krn\lbe(\e g\be\circ\!\be f\e)\to\krn g\to\cok f\to\cok\be(\e g\be\circ\!\be f\e)\to\cok g\to 0.
\]
\end{lemma}
\begin{proof} See, for example, \cite[1.2]{bey}.
\end{proof}

\begin{lemma} \label{triv} Let $k$ be a field, let $H$ be a proper $k$-group and let $G$ be a smooth, connected and affine algebraic $k$-group. Then $\Hom_{\e\text{$k$-\le\rm gr}}\le(G,H\lle)$ is trivial.
\end{lemma}
\begin{proof} Any $f\in\Hom_{\e\text{$k$-\le\rm gr}}\le(G,H\lle)$ factors as $G\overset{p}\to(\krn f\le)\be\setminus\be G\overset{\!i}{\to} H$, where $p$ is faithfully flat and of finite type and $i$ is a closed immersion \cite[${\rm VI}_{\rm A}$, Proposition 5.4.1]{sga3}. Consequently, $(\krn f\le)\!\setminus\! G$ is proper over $k$. Since $(\krn f\le)\!\setminus\! G$ is also affine over $k$ by
\cite[${\rm VI}_{\rm B}$, Theorem 11.17]{sga3}, $(\krn f\le)\!\setminus\! G$ is, in fact, finite over $k$ \cite[Tag 01WG]{sp}. By, e.g., \cite[Lemma 2.54]{bga}, $(\krn f\le)\!\setminus\! G$ is also smooth and connected. Consequently, $(\krn f\le)\!\setminus\! G$ is trivial by \cite[${\rm IV}_{4}$, Corollary 17.6.2]{ega} and \cite[II,\S5, proposition 1.4, p.~236]{dg}. Thus $f$ is the trivial homomorphism.
\end{proof}

\section{Reflexive objects in the truncated derived category}

Let $S$ be a scheme. For every $F\in\sfs$, the Cartier dual of $F$ is the object of $\sfs$ 
\begin{equation}\label{fdd}
\hskip 2cm F^{\lle D}=\shom( F,\G_{m,\e S}).
\end{equation}
Thus $F^{\lle D}\lbe(\le T\le)=\Hom_{\, T_{\lbe\fl}^{\le\sim}}\lbe( F_{\le T},\G_{m,\e T}\lbe)$ for every $S$-scheme $T$. We will write
\begin{equation}\label{e1s}
\hskip .9cm E^{\lle 1}\lbe(F\le)=\exts^{\e 1}\be(F,\G_{m,\e S})
\end{equation}
for the fppf sheaf associated to the presheaf that maps an $S$-scheme $T$ to the group $\mathrm{Ext}_{\e T_{\be\fl}^{\le\sim}}^{1}\be(F_{\le T},\G_{m,\e T})$ 
of equivalence classes of degree $1$ Yoneda extensions of $F_{\e T}$ by $\G_{m,\e T}$ in $T_{\fl}^{\le\sim}$. By \cite[Tag 06XP]{sp}, \eqref{e1s} agrees with the usual derived-functor definition, whence the local\le-\le to\le-\le global spectral sequence for Exts yields the equality 
\begin{equation}\label{er}
\hskip 2.9cm E^{\le 1}\lbe(F\le)(\lle T\lle)=\mathrm{Ext}_{\e T_{\be\fl}^{\le\sim}}^{1}\be(F_{\le T},\G_{m,\e T})\quad(\textrm{if $F^{D}=0$})
\end{equation}
for every $S$-scheme $T$. See, for example, \cite[Proposition 2.3]{rr25}.

\smallskip

For every $F\in\sfs$, the sheaf $(F\!\times\! F^{\lle D}\le)^{D}$ has a distinguished section, namely the {\it evaluation pairing} ${\rm ev}_{\lbe F}\colon F\times F^{\lle D}\to \G_{m,\e S}, (s,\chi)\mapsto \chi(\le s\lle)$. The {\it evaluation morphism} for $F$ is the morphism in $\sfs$
\begin{equation}\label{evm}
e_{F}\colon F\to F^{\lle D\lle D}
\end{equation}
defined by $e_{F}(s)={\rm ev}_{\lbe F}(s,-)\colon F^{D}\to \G_{m,\e S}, \chi\mapsto\chi(s)$, for every section $s$ of $F$. In other words, for every $S$-scheme $T$ and every section $s\in F(\lle T\le)$, $e_{F}(\le T\le)(s)$ maps 
$\chi\in F^{\lle D}\lbe(\le T^{\le\prime}\le)$ to the section $\chi(\le T^{\le\prime}\le)(s|_{T^{\le\prime}})\in \G_{m,\e S}(\le T^{\le\prime}\le)$ for every $T$-scheme $T^{\le\prime}$. 

\smallskip

The sheaf $F\in\sfs$ is called {\it reflexive}\,(relative to $\G_{m,\e S}$) if $e_{F}$ \eqref{evm} is an isomorphism in $\sfs$. If $F$ is reflexive, then $F^{\lle D}$ is reflexive.

\medskip

Next let $C(\lbe S_{\fl}\lbe)$ denote the category of complexes of objects of $\sfs$ and let $D\lbe(S_{\fl})$ be the corresponding derived category. We will write $C^{\e[-1,0\le]}\lbe(S_{\fl})$ for the category of complexes of objects of $\sfs$ concentrated in degrees $-1$ and $0$ and $D^{\e[-1,0\le]}\lbe(S_{\fl})$ for the corresponding derived category. Thus $D^{\e[-1,0\le]}\lbe(S_{\fl})$ is the  full additive subcategory of $D\lbe(S_{\fl})$ whose objects are represented by complexes $X\in C(S_{\fl})$ such that $H^{\e i}(X\le)=0$ for $i\neq -1, 0$ \cite[Proposition 13.1.12(ii)]{ks06}. Morphisms $X\to Y$ in $C^{\e[-1,0\le]}\lbe(S_{\fl})$ will be described as pairs $(\e f,g\e)$, where $f$ and $g$ are the degree $-1$ and degree $0$ components of $X\to Y$, respectively. If $F\in\sfs$ is a sheaf, the object $F[0]\in D^{\e[-1,0\le]}\lbe(S_{\fl})$ will often be denoted by $F$.

\smallskip

Now recall that, if 
$X\in C(\lbe S_{\fl}\lbe)$ is arbitrary and $n\in\Z$, the {\it $n$-th truncation} of $X$ is the complex 
\[
\tau_{\leq\le n}\e X=\dots \to X^{\le n-2}\to X^{\le n-1}\to\krn\lbe[\e X^{\le n}\to X^{\le n+1}]\to 0\to 0\to \dots.
\]
Clearly, there exists a distinguished triangle in $D\lbe(S_{\fl})$
\begin{equation}\label{0-tr}
\tau_{\leq\le -1}\e X\to \tau_{\leq\le 0}\e X\to H^{\le 0}(X\le)\to (\tau_{\leq\le -1}\e X)[1].
\end{equation}
For any object $C$ of $D\lbe(S_{\fl})$ such that $H^{\le 0}(C\le)=0$, there exists an isomorphism in $D\lbe(S_{\fl})$
\begin{equation}\label{der}
\hskip 1.5cm\tau_{\leq\le 0}(C\lle[-1])\simeq (\tau_{\leq\le 0}\e C\le)[-1] \quad (\text{if $H^{\le 0}(C\le)=0$}).
\end{equation}
We also recall that, if $v\colon X\to Y$ is a morphism in $C(S_{\fl})$, then the {\it mapping cone} of $v$ is the object $C( v\le)=X[1]\oplus Y$ of $C(S_{\fl})$ with differential defined in \cite[Definition 11.1.3]{ks06}. There exists a canonical exact sequence in $C(S_{\fl})$
\begin{equation}\label{ces}
0\to Y\to C\lbe( v\le)\to X\le[1]\to 0.
\end{equation}
Further, if $v$ is a morphism in $\sfs$, then $C( v\lle)$ is an object of $C^{\e[-1,0]}\lbe(S_{\fl})$.

\medskip

Next, if $X$ and $Y$ are arbitrary objects of $C(S_{\fl})$, let ${\mathscr H\! o\le m}_{ S_{\fl}^{\lle\sim}}^{\bullet}\be(X,Y\le)$ be the object of $C(S_{\fl})$ defined in \cite[\S11.7, p.~292]{ks06}. By \cite[proof of Lemma V.1.2, pp.~255\e-\le 256]{h66}, there exists a canonical {\it evaluation morphism} in $C(S_{\fl})$
\begin{equation}\label{evmc}
e_{X}^{(\le Y\lle)}\colon X\to  {\mathscr H\! o\le m}_{\le S_{\fl}^{\lle\sim}}^{\bullet}\be({\mathscr H\! o\le m}_{ S_{\fl}^{\lle\sim}}^{\bullet}\be(X,Y\le),Y\le)
\end{equation}
such that $e_{F}^{(\le \G_{m,S})}=e_{F}$ \eqref{evm}. 

\smallskip

Now, for any object $C$ of $D\lbe(S_{\fl})$, set
\begin{equation}\label{rc}
\hskip .7cm R(C\le)=\mathbb R\le\shom(C,\G_{m,S}[1]).
\end{equation}
There exists a functorial morphism in $D\lbe(S_{\fl})$
\begin{equation}\label{hrdm}
 \epsilon_{\e C}\colon C\to RR(C\le)
\end{equation}
that is represented by the map $e_{X}^{(\lle I\le)}$ \eqref{evmc}, where $X$ is any complex that represents $C$ and $I$ is an injective resolution of $\G_{m,\e S}[\lle 1\lle]$.

\smallskip

Next, if $C$ is any object of $D\lbe(S_{\fl})$, the (truncated) {\it dual} of $C$ in $D\lbe(S_{\fl})$ is the object
\begin{equation}\label{duf}
\hskip -1cm D(C\e)=\tau_{\le\leq\le 0}\le R(C\le).
\end{equation}
Clearly
\begin{equation}\label{ton}
\hskip -1.65cm  H^{\le 0}(D(C\e))=H^{\le 0}(R(C\le)).
\end{equation}
Now let
\begin{equation}\label{jmap}
j_{\le C}\colon D(C\le)\to R(C\le)
\end{equation}
be the canonical truncation morphism. Truncating the composition
\begin{equation}\label{kmp}
C\xrightarrow{\epsilon_{\le C}}RR(C\le)\xrightarrow{R(\e j_{\le C}\le)}R(D(C\e))
\end{equation}
in degrees $\leq 0$, we obtain a morphism in $D\lbe(S_{\fl})$
\begin{equation}\label{kmp2}
\tau_{\le\leq\le 0}(R\le(\e j_{\le C}\le)\circ \epsilon_{\le C})\colon \tau_{\le\leq\le 0}\e C\to DD(C\e).
\end{equation}
Consequently, for every object $C$ of $D^{\e[-1,0\le]}\lbe(S_{\fl})$, there exists a functorial {\it evaluation morphism} in $D^{\e[-1,0\le]}\lbe(S_{\fl})\e$:
\begin{equation}\label{dduf}
e_{\le C}\colon C\to D\lbe D(C\e)
\end{equation}
such that
\begin{equation}\label{compat-ev}
j_{\lle D(C\le)}\circ e_{\le C}=R\le(\e j_{\le C})\circ\epsilon_{\le C},
\end{equation}
where $j_{ D(C\le)}\colon DD(C\le)\to R(D(C\le))$ is an instance of the map \eqref{jmap}.

\medskip

We will say that an object $C$ of $D^{\e[-1,0\le]}\lbe(S_{\fl})$ is {\it reflexive} if the map \eqref{dduf} is an isomorphism in  $D^{\e[-1,0\le]}\lbe(S_{\fl})$.

\smallskip

\smallskip

\begin{lemma}\label{ref1} Let
\[
X\xrightarrow{f}Y\xrightarrow{g}Z\xrightarrow{h}X[1]
\]
be a distinguished triangle in $D(S_{\fl})$ such that
$H^{\e 0}(D(X\e))=0$. Then the given triangle induces a distinguished triangle in $D(S_{\fl})$
\[
D(X\e)[-1]\xrightarrow{D(h)}D(Z\e)
\xrightarrow{D(\le g)}D(\e Y)\xrightarrow{D(\le f\le)}D(X\e).
\]
\end{lemma}

\begin{proof} Applying $R$ \eqref{rc} to the given triangle yields a distinguished triangle
\[
R(X\le)[-1]\xrightarrow{R(h)}R(Z\le)
\xrightarrow{R(g)}R(\le Y)\xrightarrow{R(\le f\le)}R(X\le),
\]
where we used the canonical identification
$R(X[1])=R(X\le)[-1]$. By \eqref{ton}, the long cohomology sequence associated with the above triangle contains the exact sequence
\[
H^{\e 0}(D(Z\e))\xrightarrow{H^{\le 0}(D(\le g))}H^{\e 0}(D(\le Y))
\longrightarrow H^{\e 0}(D(X\e))=0.
\]
Thus $H^{\e 0}(D(\le g))$ is surjective. Consequently, \cite[Lemma 3.10]{bro21} yields a distinguished triangle
\[
\tau_{\leq 0}(R(X\le)[-1])
\xrightarrow{\tau_{\leq 0}\le R(h)}
D(Z\le)\xrightarrow{D(\le g)}D(\le Y)
\xrightarrow{D(\le f\le)}D(X\le).
\]
Finally, since $H^{\le 0}(R(X\le))=H^{\le 0}(D(X\e))=0$, \eqref{der} yields a canonical isomorphism
$\tau_{\leq 0}\le R(X\le)\simeq D(X)[-1]$. Under this isomorphism, the first morphism in the
preceding triangle is $D(h)$, as claimed.
\end{proof}

\begin{lemma}\label{compat}
Let
\[
X\xrightarrow{f}Y\xrightarrow{g}Z\xrightarrow{h}X[1]
\]
be a distinguished triangle in $D(S_{\fl})$, where
$X,Y,Z\in D^{\e[-1,0\le]}\lbe(S_{\fl})$. Assume that
\[
H^{\e 0}(D(X\e))=H^{\e 0}(D(D(X\le)[-1])))=0.
\]
Then there exists a distinguished triangle in $D(S_{\fl})$
\[
DD(X\le)\xrightarrow{DD(\le f\le)}DD(\le Y)\xrightarrow{DD(\le g)}
DD(Z\le)\xrightarrow{\delta}DD(X\le)[1]
\]
such that the following diagram is a morphism of distinguished triangles in $D(S_{\fl})${\rm{:}}
\[
\xymatrix@C=3.5em@R=3em{
X \ar[r]^{f} \ar[d]_{e_{X}} & Y \ar[r]^{g} \ar[d]_{e_{\le Y}}
& Z \ar[r]^{h} \ar[d]_{e_{\lle Z}}& X[1] \ar[d]^{e_{X}[1]} \\
DD(X\le) \ar[r]^{DD(\le f\le )}& DD(\le Y) \ar[r]^{DD(\le g)}
& DD(Z\le) \ar[r]^(.45){\delta}
& DD(X\le)[1].
}
\]
In particular, if any two of $X,Y,Z$ are reflexive,
then so is the third.
\end{lemma}
\begin{proof} Since $H^{\e 0}(D(X\e))=0$, Lemma \ref{ref1} yields a distinguished triangle
\[
D(X\e)[-1]\xrightarrow{D(h)}D(Z\e)
\xrightarrow{D(\le g)}D(\e Y)\xrightarrow{D(\le f\le)}D(X\e)=(D(X\le)[-1])[1].
\]
Now, since $H^{\e 0}(D(D(X\le)[-1])))=0$, we may apply Lemma \ref{ref1} to the above triangle and obtain a distinguished triangle
\begin{equation}\label{compat-tr2}
D(D(X\le)[-1])[-1]\xrightarrow{DD(\le f\le)}DD(Y)
\xrightarrow{DD(\le g)}DD(Z)\xrightarrow{D(D(\le h\le))}D(D(X\le)[-1]).
\end{equation}
Next, via the identification $R(X[1])=R(X)[-1]$, we have
\begin{equation}\label{compat-a}
j_{\le Z}\circ D(h)=R(h)\circ j_{\le X}[-1],
\end{equation}
where $j_{\le Z}\colon D(Z\e)\to R(Z\e)$ and $j_{\le X}\colon D(X\e)\to R(X\e)$ are instances of \eqref{jmap}. Further, since $R(D(X\le)[-1])=R(D(X))[1]$, the equality
$H^{\le 0}(D(D(X\le)[-1]))=0$ implies that $H^{\le 1}(R(D(X)))=0$. Consequently, truncation yields a canonical isomorphism
\begin{equation}\label{compat-beta}
\beta\colon DD(X)[1]\isoto D(D(X\le)[-1])
\end{equation}
such that
\begin{equation}\label{compat-beta2}
j_{\le D(X\le)[-1]}\circ\beta=j_{D(X)}[1].
\end{equation}
Using \eqref{compat-beta} to identify $D(D(X\le)[-1])[-1]$ and $DD(X\le)$, the triangle
\eqref{compat-tr2} is now
\[
DD(X)\xrightarrow{DD(f)}DD(Y)
\xrightarrow{DD(g)}DD(Z)
\xrightarrow{\delta}DD(X)[1],
\]
where
\begin{equation}\label{dta}
\delta=\beta^{\le -1}\circ D(\lbe D(h\le)).
\end{equation}
Finally, in the statement of the lemma, the first two squares commute by the functoriality of the evaluation morphisms. It remains only to check that
\begin{equation}\label{compat-last}
\delta\circ e_{\le Z}=e_{\le X}[1]\circ h.
\end{equation}
By the naturality of the truncation morphisms, we have
\begin{equation}\label{ext}
j_{D(X\le)[-1]}\circ D(D(h\le))=R(D(h\le))\circ j_{\le D(Z)}.
\end{equation}
Thus, using the naturality of $\epsilon_{\le (-)}$ \eqref{hrdm}, we obtain
\[
\begin{array}{rcl}
j_{D(X\le)[-1]}\circ\beta\circ\delta\circ e_{\le Z}
&\overset{\eqref{dta}}{=}&j_{D(X\le)[-1]}\circ D(D(h\le))\circ e_{\le Z}\\
&\overset{\eqref{ext}}{=}&R(D(h\le))\circ j_{D(Z\le)}\circ e_{\le Z}\overset{\eqref{compat-ev}}{=}R(D(h\le))\circ R(\e j_{\le Z})\circ\epsilon_{\le Z}\\
&=&R(\e j_{Z}\circ D(h\le))\circ\epsilon_{\le Z}\overset{\eqref{compat-a}}{=}R(R(h)\circ j_{\le X}[-1])\circ\epsilon_{\le Z}\\
&=&R(\e j_{\le X}[-1])\circ R(R(h\le))\circ\epsilon_{\le Z}=R(\e j_{X})[1]\circ\epsilon_{X}[1]\circ h\\
&\overset{\eqref{compat-ev}}{=}&j_{D(X)}[1]\circ e_{X}[1]\circ h\overset{\eqref{compat-beta2}}{=}j_{D(X\le)[-1]}\circ\beta\circ e_{X}[1]\circ h.
\end{array}
\]
Now $D(D(X)[-1])=\tau_{\le\leq\le 0}\e R(D(X)[-1])$ and
$Z\in D^{\le\leq\le 0}(S_{\fl})$. Therefore the canonical truncation morphism $j_{D(X)[-1]}$ induces an isomorphism
\[
\Hom_{\le D(S_{\fl})}(Z,D(D(X)[-1]))
\xrightarrow{\sim}
\Hom_{\le D(S_{\fl})}(Z,R(D(X)[-1])).
\]
It follows that $\beta\circ\delta\circ e_Z
=\beta\circ\e e_{X}[1]\circ h$. Since $\beta$ is an isomorphism, we obtain the equaility \eqref{compat-last}. The final assertion of the lemma follows from \cite[Proposition 10.1.15]{ks06} and appropriate rotations (if necessary).
\end{proof}

\medskip

Let $F\in\sfs$ be a sheaf. Then $D(F\le)$ is the object of $D^{\e[-1,0\le]}\lbe(S_{\fl})$ such that
\begin{equation}\label{e1}
\hskip -1.5cm H^{\le -1}(D(F\le)\lle)\!=F^{\le D}
\end{equation}
and 
\begin{equation}\label{e2}
\hskip -.6cm H^{\le 0}(D(F\le))=E^{\le 1}\lbe(F\le).
\end{equation}
By \eqref{e1} and \eqref{e2}, \eqref{0-tr} is a distinguished triangle in $D\lbe(S_{\fl})$
\begin{equation}\label{ftr}
F^{\lle D}[1]\lra D(F\e)\to E^{\le 1}\lbe(F\le)\lra F^{\lle D}[2].
\end{equation}
Henceforth, we will make the following identifications in $D^{\e[-1,0\le]}\lbe(S_{\fl})\le$:
\begin{equation}\label{obv1}
\hskip .9cm D(F\e)=F^{\le D}[1]\hskip .5cm\qquad\textrm{(\le if $E^{\le 1}\lbe(F\le)=0$)}
\end{equation}
(via the first map in \eqref{ftr}) and 
\begin{equation}\label{obv2}
\hskip .35cm D(F\e)=E^{\le 1}\lbe(F\le)
\hskip .5cm\qquad\textrm{(\le if $F^{\lle D}=0$)}
\end{equation}
(via the second map in \eqref{ftr}).

\medskip

Next, for any sheaf $F\in\sfs$, we have
\begin{equation}\label{sad}
D(\lbe F^{\le D}[1])=\tau_{\leq\e  0}\e\mathbb R\le\shom(F^{\lle D}[1],\G_{m,\e S}[\lle 1\lle])=F^{DD}.
\end{equation}
Consequently, the dual of the first map in \eqref{ftr} yields the horizontal map in the following diagram in $D^{\e[-1,0\le]}\lbe(S_{\fl})$: 
\begin{equation}\label{trg}
\xymatrix{F^{DD}&\ar[l]D\lbe D(F\le)\\
F\ar[u]^{e_{F}}\ar[ur]_{e_{\lle F\lbe[\lle 0\lle]}}&.
}
\end{equation}
Thus \eqref{obv1}, \eqref{sad} and \eqref{trg} induce identifications
\begin{equation}\label{nid}
\hskip 3.5cm D\lbe D(F\e)=F^{\lle DD}\qquad\textrm{(\le if $E^{\le 1}\lbe(F\le)=0$)}
\end{equation}
and
\begin{equation}\label{nid2}
\hskip 4.1cm e_{\lle F[\le 0\le]}=e_{F}\hskip .45cm\qquad\textrm{(\le if $E^{\le 1}\lbe(F\le)=0$).}
\end{equation}
The preceding identifications yield the following

\begin{proposition}\label{res} If $F\in\sfs$ is a reflexive sheaf and $E^{\le 1}\lbe(F\le)=0$, then $F\lle[\lle 0\lle]$ is a reflexive object of $D^{\e[-1,0\le]}\lbe(S_{\fl})$.
\end{proposition}

\smallskip

\begin{lemma}\label{pai} Let $v\colon F\to A$ be a morphism in $\sfs$, where $F$ and $A$ are such that $E^{\le 1}\lbe(F\le)=A^{D}=0$, and set
\[
G(v\le)\overset{{\rm def.}}{=}\extcs^{\e 1}\lbe(\e C\lbe( v\le),\G_{m,S}).
\]
Then the sheaf $G(v\le)$ is an extension in $\sfs$
\begin{equation}\label{qap}
0\to F^{\lle D}\to G( v\le)\to E^{\le 1}\lbe\lbe(\lbe A\lle)\to 0
\end{equation}
and there exists a natural isomorphism in $D^{\e[-1,0\le]}\lbe(S_{\fl})$
\begin{equation}\label{iso0}
\theta_{v}\colon D(C\lbe( v\le))\isoto G( v\le).
\end{equation}
\end{lemma}
\begin{proof} Since $\extcs^{\e i}(F\lle[1],\G_{m,\le S})=\exts^{\e i-1}(F,\G_{m,\le S})$ for all $i\geq 0$, the exact sequence $0\to A\to C\lbe(v\le)\to F\le[1]\to 0$ \eqref{ces} yields the extension \eqref{qap} and the equality
\[
H^{-1}(D(C\lbe(v\le)))=\shoms(C( v\le),\G_{m,\le S})=0.
\]
Thus the projection $D(C\lbe( v\le))\isoto H^{\le 0}(D(C\lbe( v\le)))$ is an isomorphism in $D^{\e[-1,0\le]}\lbe(S_{\fl})$. On the other hand, by \eqref{ton}, there exists a canonical isomorphism in $\sfs$
\[
H^{\le 0}(D(C\lbe( v\le)))\isoto \extcs^{\e 1}\lbe(\e C\lbe( v\le),\G_{m,S})=G(v\le).
\]
The composition of the preceding isomorphisms is the map \eqref{iso0}.
\end{proof}

\begin{remarks}\label{pai2}\indent
\begin{enumerate}
\item[(a)] For every $S$-scheme $T$, \eqref{iso0} induces
an isomorphism of abelian groups
\begin{equation}\label{iso1}
H^{\le 0}(\theta_{v})(\le T\le)\colon H^{\le 0}(D(C\lbe( v\le)))(\le T\le)\isoto G(v)(\le T\le).
\end{equation}
Now we claim that there exists a natural isomorphism of abelian groups
\begin{equation}\label{iso2}
G(v)(\le T\le)\isoto \exct^{1}(C(v)_{\le T},\G_{m,\le T}).
\end{equation}
Indeed, by \cite[Tag 08J7]{sp}, the derived internal Hom functor commutes with restrictions. Thus the argument in the proof of the lemma applies over $T$ and yields $\shomt(C( v\le)_{T},\G_{\le m,\le T})=0$. Now, since $C(v)_{T}\in D^{\le\leq 0}(\le T_{\fl})$ and
$\G_{m,T}[-1]\in D^{\le\geq 1}(\le T_{\fl})$, we have $\exttc^{\e i}(C(v)_{T},\G_{\le m,\le T})=0$
for all $i\leq 0$. Consequently, the local\le-\le to\le -\le global spectral sequence for derived Exts \cite[Tags 0BKM and 0B6E]{sp}
\[
H^{\le p}\bigl(T_{\fl},
\exttc^{\le q}(C(v)_{T},\G_{\le m,\le T})\bigr)\Longrightarrow \exct^{\le p+q}(C(v)_{T},\G_{\le m,\le T})
\]
yields the natural isomorphism \eqref{iso2}
\[
\hskip 1.5cm G(v)(\le T\le)=H^{\le 0}(T_{\fl},
\exttc^{\le 1}(C(v)_{T},\G_{\le m,\le T}))\simeq \exct^{1}(C(v)_{T},\G_{\le m,\le T}).
\]

We will now describe the composition
\begin{equation}\label{komp}
H^{\le 0}(D(C\lbe( v\le)))(\le T\le)\xrightarrow[\sim]{\eqref{iso1}} G(v\le)(\le T\le)\xrightarrow[\sim]{\eqref{iso2}} \exct^{1}(C(v)_{T},\G_{m,\le T}).
\end{equation}
Let $f\in H^{\le 0}(D(C\lbe( v\le)))(\le T\le)=\Hom_{D(\le T_{\fl})}(C(v\le)_{\le T},\G_{\le m,\le T}[1])$. Then there exist a complex $C^{\e\prime}\in C^{\e[-1,0\le]}\lbe(T_{\fl})$ that is quasi\le-\le isomorphic to $C(v\le)_{\le T}$ and a short exact sequence of complexes
\begin{equation}\label{sese}
0\longrightarrow \G_{\le m,\le T}\longrightarrow E
\longrightarrow C^{\e\prime}\longrightarrow 0
\end{equation}
such that, under the identification $C^{\e\prime}= C(v\le)_{\le T}$ in $D(T_{\fl})$, the connecting morphism $C^{\e\prime}\longrightarrow \G_{\le m,\le T}[1]$ in $D(T_{\fl})$ associated with the above sequence corresponds to $f\colon C(v\le)_{\le T}\to \G_{\le m,\le T}[1]$. The image of $f$ under the composition \eqref{komp} is then the extension class in $\exct^{1}(C(v)_{T},\G_{m,\le T})$ represented by the sequence \eqref{sese}.

\smallskip

\item[(b)] In the setting of the lemma, if $F=0$, then $C(v)=A$, $G(v)=E^{\le 1}\lbe\lbe(\lbe A\lle)$ and the isomorphism $\theta_{\le 0}\colon D(C\lbe( v\le))\isoto G( v\le)$ \eqref{iso0} is the identification $D(A)=E^{\le 1}\lbe\lbe(\lbe A\lle)$ \eqref{obv2}.
\end{enumerate}
\end{remarks}

\bigskip

\section{The reflexivity of abelian schemes}\label{abrs}

In this section, $A$ is an abelian $S$-scheme with unit section $\varepsilon\colon S\to A$. Let ${\rm Pic}_{ A/S}$ be the relative (fppf) Picard functor of $A$ over $S$. By \cite[p.~204]{blr}, for any $S$-scheme $T$ there exists a functorial isomorphism of abelian groups ${\rm Pic}_{\lbe A/S}(\le T\le)\simeq \krn \varepsilon_{ T}^{*}$, where $\varepsilon_{T}^{*}={\rm Pic}\le(\varepsilon_{ T})\colon {\rm Pic}\le(\lbe A_{\le T})\to {\rm Pic}\le(\e T\e)$ is the canonical map. Let ${\rm Pic}^{0}\lbe(\lbe A_{\le T}/\e T\e)$ be the subgroup of ${\rm Pic}\le(\lbe A_{\le T})$ that consists of isomorphism classes of invertible sheaves on $A_{\le T}$ that are {\it fiberwise} algebraically equivalent to $0$ relative to $A_{\le T}\to T\e$. Now, set ${\rm Pic}^{0}_{\e\varepsilon}\le(\lbe A_{\le T}/\e T\e)={\rm Pic}^{0}\le(\lbe A_{\le T}/\e T\e)\cap\krn \varepsilon_{ T}^{*}$, where the intersection takes place inside ${\rm Pic}\le(\lbe A_{\le T})$. Every class in ${\rm Pic}^{0}_{\e\varepsilon}\le(\lbe A_{\le T}/\e T\e)\subset {\rm Pic}\le(\lbe A_{\le T})$ is represented by an invertible sheaf $\mathcal L$ on $A_{\le T}$ such that there exists an  isomorphism of $\mathcal O_{\le T}\le$-\le modules $\iota_{\lbe\mathcal L}\colon \varepsilon_{ T}^{*}\e\mathcal L\isoto\mathcal O_{\le T}$. Now recall that the abelian $S$-scheme $A^{\vee}$ dual to $A$ represents the functor ${\rm Pic}^{0}_{\e\varepsilon}\le(\lbe A/\lle S\e)\colon T\mapsto {\rm Pic}^{0}_{\e\varepsilon}\le(\lbe A_{\le T}/\e T\e)$. More precisely, by the Yoneda argument of \cite[$\text{I}_{\le\text{new}}$, (1.1.8)]{ega}, there exists an isomorphism in $\sfs$
\begin{equation}\label{fia}
\Phi_{\be A}\colon A^{\vee}\isoto {\rm Pic}^{\lle 0}_{\e\varepsilon}\lle(\lbe A/\lle S\e)
\end{equation}
such that, for every $S$-scheme $T$,
\[
\hskip .4cm\Phi_{\be A}(\le T\le)\colon A^{\vee}(\le T\le)\isoto {\rm Pic}^{0}_{\e\varepsilon}\le(\lbe A_{\le T}/\e T\e)
\]
maps an $S$-morphism $g\colon T\to A^{\vee}$ to $({\rm Id}_{A}\lbe\times\lbe g\e)^{*}\le\mathcal P_{\! A}\in {\rm Pic}^{0}_{\e\varepsilon}\le(\lbe A_{\le T}/\e T\e)$, where $\mathcal P_{\! A}\overset{{\rm def.}}{=}\Phi_{\be A}(\lbe A^{\vee}\le)({\rm Id}_{A^{\vee}}\lbe)\in 
{\rm Pic}^{0}_{\e\varepsilon}\le(\lbe A\times_{S}\be A^{\vee}/ A^{\vee}\e)\subseteq {\rm Pic}\le(\lbe A\times_{S}A^{\vee}\e)$ is the {\it Poincar\'e class of $A$}.

\smallskip

Next let $\overline{\varepsilon}\colon S\to A^{\vee}$ be the unit section of $A^{\vee}$ and let $\sigma_{\!A^{\vee}\be,\e A}\colon A^{\vee}\times_{S}A\to A\times_{S}A^{\vee}$ be the morphism that switches the factors. Then $\sigma_{\!A^{\vee}\lbe,\le A}^{*}\le\mathcal P_{\! A}\in {\rm Pic}^{0}_{\,\overline{\varepsilon}}\le(\lbe A^{\vee}\times_{S}A/A\e)$ whence, by the preceding discussion applied to $A^{\vee}$, there exists a unique morphism of $S$-schemes
\begin{equation}\label{bdi}
\hskip -.5cm\kappa_{\lbe A}\colon A\to A^{\vee\vee}
\end{equation}
such that 
\begin{equation}\label{bdi2}
\sigma_{\!A^{\vee}\lbe,\le A}^{*}\mathcal P_{\! A}=({\rm Id}_{A^{\vee}}\times \kappa_{\lbe A}\lbe)^{*}\le\mathcal P_{\! A^{\vee}}.
\end{equation}
By \cite[Theorem 27.222]{gw23}, the map \eqref{bdi} is an {\it isomorphism} of abelian $S$-schemes.

\smallskip

We now choose an invertible sheaf $\mathcal P$ on $A\times_{S}\be A^{\vee}$ representing $\mathcal P_{\! A}$  and a rigidification $\iota_{\le\mathcal P}\colon
(\varepsilon\times\id_{\lbe A^{\lbe\vee}})^{*}\le\mathcal P
\xrightarrow{\sim}\mathcal O_{\be A^{\lbe\vee}}$. Now, since 
$\overline{\varepsilon}$ is the zero element of $A^{\vee}(S\le)$, $(\id_{A}\times \overline{\varepsilon}\,)^{*}\le\mathcal P_{\! A}=\Phi_{\be A}(S\le)(\overline{\varepsilon})$ is the zero class in ${\rm Pic}\le(\lbe A\e)$, whence there exists a rigidification
$\rho\colon
(\id_{A}\times\overline{\varepsilon}\,)^{*}\mathcal P
\xrightarrow{\sim}\mathcal O_{\be A}$. By \cite[VII, 1.3.2]{sga7}, $\rho$ can be modified so that the two
rigidifications agree at the common origin, while keeping
$\iota_{\le\mathcal P}$ fixed. Thus $\mathcal P_{\! A}$ is {\it birigidified} relatively to the pair $(\varepsilon_{\be A^{\lbe\vee}},\overline{\varepsilon}_{\be A}\le)$. Consequently, by \cite[VII, Example 2.9.5]{sga7}, the class $\mathcal P_{\! A}\in {\rm Pic}^{0}_{\e\varepsilon,\e \overline{\varepsilon}}\le(\lbe A\times_{S}A^{\vee}/\e A^{\vee}\e)$ corresponds to an element (also denoted by $\mathcal P_{\! A}$) of the abelian group ${\rm Biext}^{1}(A, A^{\vee}\e;\le\G_{m,\e S})$ of isomorphism classes of biextensions of $(A,A^{\vee})$ by $\G_{m,\e S}$. The class $\mathcal P_{\! A}\in {\rm Biext}^{1}(A, A^{\vee}\e;\le\G_{m,\e S})$ is called the {\it Poincar\'e biextension} class {\it of A}.

Next, by \cite[VII, Corollary 3.6.5]{sga7}, there exists a canonical isomorphism of abelian groups
\[
{\rm Biext}^{1}(A, A^{\vee}\e;\le\G_{m,\e S})\simeq
\Hom_{\le D\lle(S_{\fl})}(A\otimes^{\le\mathbf{L}}\! A^{\vee}, \G_{m,\e S}[\lle 1\lle]).
\]
Since $\Hom_{\le D\lle(S_{\fl})}(A\otimes^{\le\mathbf{L}}\lbe A^{\vee}, \G_{m,\e S}[\lle 1\lle])\simeq \Hom_{\le D^{\e[-1,0\le]}\lbe(S_{\fl})}(\tau_{\,\geq\e -1}(A\otimes^{\le\mathbf{L}}\lbe A^{\vee}), \G_{m,\e S}[\lle 1\lle])$, we conclude that the Poincar\'e biextension class $\mathcal P_{\! A}\in {\rm Biext}^{1}(A, A^{\vee}\e;\le\G_{m,\e S})$ corresponds to a morphism in $D^{\e[-1,0\le]}\lbe(S_{\fl})$
\begin{equation}\label{pmor}
\tau_{\,\geq\e -1}\e(A\otimes^{\le\mathbf{L}}\!\be A^{\vee})\to \G_{m,\e S}[\lle 1\lle].
\end{equation}
It is possible (but not very enlightening) to describe \eqref{pmor} entirely within $D^{\e[-1,0\le]}\lbe(S_{\fl})$. However, as suggested by Grothendieck in \cite{gv66} and developed by Deligne in \cite[XVIII, 1.4]{sga4}, it is (much) better to describe \eqref{pmor} in terms of {\it commutative group stacks}. To this end, we will review parts of \cite[XVIII, 1.4]{sga4} and \cite{bro21}.

\smallskip

Let $({\rm CGS})$ denote the $2$-category of commutative group stacks over $S_{\fl}$ \cite[XVIII, 1.4.5]{sga4}\,\e\footnote{\e In \cite[XVIII, 1.4]{sga4}, these objects are called {\it champs de Picard strictement commutatifs}.}\e. Every $G\in ({\rm CGS})$ is a stack in groupoids over $S_{\fl}$ in which every $2$-morphism is an isomorphism (thus $({\rm CGS})$ is, in fact, a $(2,1)$-category \cite[Tag 003S]{sp}). For every $G\in ({\rm CGS})$, we will write $H^{-1}(G\le)$ (respectively, $H^{\le 0}\lbe(G\e)$) for the fppf sheaf of automorphisms of the unit section of $G$ (respectively, the fppf sheaf associated to the presheaf that maps an $S$-scheme $T$ to the group of isomorphism classes of objects of $G(\le T\e)$). 
The functors $H^{\le i}(-)\colon ({\rm CGS})\to\sfs$ transform every commutative diagram in $({\rm CGS})$ (i.e., commutative up to a $2\e$-\le isomorphism) into a {\it strictly} commutative diagram in $\sfs$. If $G_{1},G_{2}\in({\rm CGS})$, the {\it stack of homomorphisms} ${\rm HOM}(G_{1},G_{2})={\rm HOM}_{\le{\rm CGS}}(G_{1},G_{2})\in({\rm CGS})$ is the stack whose fiber over an $S$-scheme $T$ is the groupoid of homomorphisms of commutative $T$-group stacks $G_{1}\times_{S}T\to G_{2}\times_{S}T$. If $G\in ({\rm CGS})$, then the biadditive functors from $G_{1}\times_{S} G_{2}$ to $G$ form a commutative group stack that is denoted by ${\rm HOM}(G_{1},G_{2}\e;\e G\e)$. See \cite[XVIII, 1.4.7 and 1.4.8]{sga4}. 

\smallskip

Every $F\in\sfs$ defines an object of $({\rm CGS})$ that, for clarity, will be denoted by ${\mathbf{s\lle t}}(F\le)$. For every $S$-scheme $T$, the objects of ${\mathbf{s\lle t}}(F\le)(\le T\e)$ are pairs $(a, T\e)$, where $a\in F(\le T\e)$, and its morphisms $(a, T\le)\to (a^{\le\prime},T^{\le\prime})$ are morphisms of $S$-schemes $f\colon T^{\le\prime}\to T$ such that $a^{\le\prime}=F(\lle f\le)(a)$. Every groupoid ${\mathbf{s\lle t}}(F\le)(\le T\le)$ is discrete, whence $H^{\le 0}({\mathbf{s\lle t}}(F\le)\le)=F\e$.

\smallskip

Next, for every $G\in ({\rm CGS})$, there exists a functorial {\it projection homomorphism} in $({\rm CGS})$
\begin{equation}\label{phom}
{\bm{\pi}}_{G}\colon G\to {\mathbf{s\lle t}}(H^{\le 0}\lbe(G\e))
\end{equation}
such that $H^{\le 0}({\bm{\pi}}_{G})=\id_{H^{\le 0}\lbe(G\e)}$.
If $H^{-1}(G\le)=0$ (i.e., if the identity is the only automorphism of the unit section of $G\e$), then \eqref{phom} is an {\it equivalence} in $({\rm CGS})$. See, e.g., \cite[Example 2.13]{bro21}.

\smallskip

Now, by definition \cite[Definition 3.1]{bro21}, the {\it dual} in $({\rm CGS})$ of a commutative group stack $G\in({\rm CGS})$ is the stack of homomorphisms
\[
D(\lbe G\e)\overset{{\rm def.}}{=}{\rm HOM}(G,B\e\G_{m,\e S}),
\]
where $B\e\G_{m,\e S}$ is the classifying stack of $\G_{m,\e S}\e$, i.e., the stack whose fiber over an $S$-scheme $T$ is the groupoid of (sheaf) $\G_{m,\e T}\e$-\e torsors over $T$.

\smallskip

Let ${\textrm{PIC}}_{S}\in ({\rm CGS})$ be the {\it absolute Picard stack} over $S$, i.e., for every $S$-scheme $T$, ${\textrm{PIC}}_{S}(\le T\le)$ is the groupoid whose objects are the invertible $\mathcal O_{\le T}\e$-\e modules and whose morphisms are the $\mathcal O_{\le T}\e$-\e module isomorphisms. For every $T$ as above, there exists an {\it equivalence} of groupoids $B\e\G_{m,\e S}\lle(\le T\le)\isoto {\textrm{PIC}}_{S}(\le T\le), F\mapsto F\lbe\wedge^{\G_{m,T}}\mathcal O_{\le T}$, where  $F\lbe\wedge^{\G_{m,T}}\mathcal O_{\le T}$ is the contracted product of sheaf $\G_{m,T}\e$-\e torsors \cite[III.1.3]{gir71}. A quasi-inverse to the above equivalence is the functor ${\textrm{PIC}}_{S}(\le T\le)\isoto 
B\e\G_{m,\e S}\lle(\le T\le), \mathcal L\mapsto \isomt(\mathcal{O}_{\e T}, \mathcal{L}\le)$. Thus we obtain an equivalence in $({\rm CGS})$
\begin{equation}\label{bpic0}
B\e\G_{m,\e S}\isoto{\textrm{PIC}}_{S}.
\end{equation}
Henceforth, via \eqref{bpic0}, for every $G\in ({\rm CGS})$ we will make the identifications
\[
\hskip .5cm D(G\e)={\rm HOM}(G,{\textrm{PIC}}_{S}\e)
\]
and 
\[
D\le(G\lbe\times_{S}\lbe D(\lbe G\e))={\rm HOM}(G,D(\lbe G\e)\e;\e {\textrm{PIC}}_{S}\e).
\]
The stack $D\le(G\lbe\times_{S}\lbe D(\lbe G\e))$ contains a distinguished object, namely the {\it evaluation pairing}
\begin{equation}\label{evp}
\textrm{\bf{ev}}_{\lle G}\colon G\times_{S} D(\lbe G\e)\to {\textrm{PIC}}_{S}, (\le g,\varphi)\mapsto \varphi\lle(\le g\e).
\end{equation}
The {\it evaluation morphism} of $G$ is the morphism in $({\rm CGS})$
\begin{equation}\label{std}
\bm{e}_{\lle G}\colon G\to DD(\lbe G\e)
\end{equation}
defined by $\bm{e}_{\le G}(\le g\e)=\textrm{\bf{ev}}_{\lle G}(\le g,-)\colon D(\lbe G\e)\to {\textrm{PIC}}_{S}, \varphi\mapsto \varphi\lle(\le g\e)$. We will say that $G\in ({\rm CGS})$ is {\it reflexive} if the map \eqref{std} is an equivalence in $({\rm CGS})$.

\smallskip

Next let $\overline{({\rm CGS})}$ denote the category whose objects are those of $({\rm CGS})$ and whose morphisms are the $2$-isomorphism classes of homomorphisms in $({\rm CGS})$. Note that equivalences in $({\rm CGS})$ correspond to isomorphisms in $\overline{({\rm CGS})}$. By \cite[XVIII, 1.4.11]{sga4}, the functor $C^{\e[-1,0\le]}\lbe(S_{\fl})\to ({\rm CGS})$
that maps a complex $[\le X^{-1}\be\to\be X^{\le 0}\e]$ to the quotient stack $[ X^{\le 0}\be/X^{-1}\le]$ induces an {\it equivalence of categories}
\begin{equation}\label{beq1}
{\mathbf{s\lle t}}\colon D^{\e[-1,0\le]}\lbe(S_{\fl})\isoto \overline{({\rm CGS})}
\end{equation}
that is compatible with duality in the following sense: there exists a {\it natural} isomorphism of contravariant functors from $D^{\e[-1,0\le]}\lbe(S_{\fl})$ to $\overline{({\rm CGS})}\e$
\begin{equation}\label{amor}
\bm{\alpha}\colon  D\circ {\mathbf{s\lle t}}\isoto {\mathbf{s\lle t}}\circ D.
\end{equation}
See \cite[XVIII, 1.4.18]{sga4}\,\footnote{\e We will describe the specific value of \eqref{amor} that is relevant to this paper in \eqref{dial} below.}\,. In particular, $C\in D^{\e[-1,0\le]}\lbe(S_{\fl})$ is a reflexive object of $D^{\e[-1,0\le]}\lbe(S_{\fl})$ if, and only if, the associated stack $\mathbf{s\lle t}(C\e)$ is a reflexive object of $\overline{({\rm CGS})}$. 

\smallskip

Next, the functor \eqref{beq1} induces isomorphisms in $\overline{({\rm CGS})}$
\begin{equation}\label{beq2}
\hskip 1.2cm {\mathbf{s\lle t}}(\tau_{\,\geq\e -1}\e A\otimes^{\le\mathbf{L}}\!\lbe A^{\vee}\le)\simeq \mathbf{s\lle t}(A)\lbe\otimes\lbe \mathbf{s\lle t}(A^{\vee}\le)
\end{equation}
and
\begin{equation}\label{beq3}
\hskip 1cm {\mathbf{s\lle t}}(\G_{m,\e S}[1])\simeq B\e\G_{m,\e S}\le,
\end{equation}
where $\otimes$ is the tensor product in $({\rm CGS})$ constructed in \cite[XVIII, 1.4.20]{sga4}. Via the functor $\otimes\colon \mathbf{s\lle t}(A)\times_{S}\e \mathbf{s\lle t}(A^{\vee})\to\mathbf{s\lle t}(A)\otimes\e \mathbf{s\lle t}(A^{\vee})$ \cite[XVIII, 1.4.8 and 1.4.20]{sga4} and the isomorphisms \eqref{bpic0}, \eqref{beq2} and \eqref{beq3}, the morphism $\tau_{\,\geq\e -1}\e A\otimes^{\mathbf{L}}\!\lbe A^{\vee}\to \G_{m,\e S}[\lle 1\lle]$ \eqref{pmor} corresponds, under \eqref{beq1}, to a pairing in $\overline{({\rm CGS})}$.
\begin{equation}\label{stpo2}
\langle-,-\rangle_{\le A}\colon \mathbf{s\lle t}(A)\times_{S} \mathbf{s\lle t}(A^{\vee})\to {\textrm{PIC}}_{S}\e.
\end{equation}

We will now describe the above pairing.

\smallskip

To this end, let ${\rm PIC}_{\e\varepsilon}\lle(\lbe A\lbe/\lbe S\le)$ be the stack whose fiber over an $S$-scheme $T$ is the groupoid whose objects are triples $(\mathcal L,\iota_{\lbe\mathcal L}, T\le)$, where $\mathcal L$ is an invertible $\mathcal O_{\be A_{\lle T}}\e$-\e module and $\iota_{\lbe\mathcal L}\colon  \varepsilon_{T}^{*}\le\mathcal L\isoto\mathcal O_{\le T}$ is an isomorphism of $\mathcal O_{\le T}\le$-\le modules, and whose morphisms $(\mathcal L,\iota_{\lbe\mathcal L}, T\le)\to(\mathcal L^{\le\prime},\iota_{\lbe\mathcal L^{\prime}}, T\le)$ are isomorphisms of $\mathcal O_{\be A_{\lle T}}$-\le modules $\varphi\colon\mathcal L\isoto \mathcal L^{\le\prime}$ such that $\iota_{\lbe\mathcal L^{\prime}}\le\circ\le \varepsilon_{\lle T}^{*}(\varphi)=\iota_{\lbe\mathcal L}$.
We will write ${\rm PIC}^{\e 0}_{\varepsilon}\lle( A\lbe/\lbe S\le)$ for the full subcategory of ${\rm PIC}_{\varepsilon}\lle( A\lbe/\lbe S\le)$ whose objects are triples $(\mathcal L,\iota_{\lbe\mathcal L}, T\le)$ as above such that $\mathcal L$ is fiberwise algebraically equivalent to $0$ (relative to $A_{T}\to T\e$). Clearly $H^{\le 0}(\le{\rm PIC}_{\e\varepsilon}\lle( A\lbe/\lbe S\le)\le)={\rm Pic}_{\e\varepsilon}\le(\lbe A/\lbe S\e)$ and $H^{\le 0}(\le{\rm PIC}^{\e 0}_{\e\varepsilon}\lle( A\lbe/\lbe S\le)\le)={\rm Pic}^{\e 0}_{\e\varepsilon}\le(\lbe A/\lbe S\e)$. The following statement is ``well-known", but we were unable to find a suitable reference.

\begin{lemma}\label{noa} For every $S$-scheme $T$ and every $(\mathcal L,\iota_{\lbe\mathcal L}, T\e)\in {\rm PIC}_{\varepsilon}\lle( A\lbe/\lbe S\le)(\le T\le)$, we have $\Aut(\mathcal L,\iota_{\lbe\mathcal L}, T\le)=\{\id_{\le(\mathcal L\le,\e\iota_{\lbe\mathcal L}, T\le)}\}$.
\end{lemma}
\begin{proof} Let $\zeta_{\le T}\colon \mathcal O_{\le T}\isoto (q_{\le T})_{*}\le\mathcal O_{\be A_{\le T}}$ be the canonical isomorphism, i.e., the unit of the adjunction $(q_{\le T}^{\e *},(q_{\le T})_{*})$ (see, e.g., \cite[Answer 3.11. p.~64]{kl05}). Then $\zeta_{\le T}$ induces a bijection $\Gamma(T,\mathcal O_{\le T})^\times\isoto\Gamma(A_{\le T},\mathcal O_{\be A_{\le T}})^\times,\, u\mapsto q_{\le T}^{\e *}(u)$. Now, since 
$\mathcal L$ is an invertible sheaf on $A_{\le T}$, any $\mathcal O_{\be A_T}\le$-\le linear automorphism $\psi$ of $\mathcal L$ is multiplication by $\zeta_{\le T}(u)=q_{\le T}^{\e *}(u)$ for some (unique) $u\in \Gamma(T,\mathcal O_T)^{\times}$. Consequently, on $\varepsilon_{\le T}^{*}\le\mathcal L\simeq \mathcal O_{\le T}$, the morphism $\varepsilon_{\le T}^*(\psi)$ is multiplication by
\[
\varepsilon_{\le T}^*(q_{\le T}^{\e *}(u))=(q_{\le T}\circ\varepsilon_{\le T})^{*}(u)=(\id_{\le T})^{*}u=u.
\]
Finally, since 
$\iota_{\lbe\mathcal L}\colon\varepsilon_{\le T}^{*}\mathcal L\isoto\mathcal O_{\le T}$ is $\mathcal O_{\le T}\le$-\le linear, we have
\[
\iota_{\lbe\mathcal L}=\iota_{\lbe\mathcal L}\circ\varepsilon_T^*(\psi)=(\text{multiplication\ by }u\text{ on }\mathcal O_{\le T})\circ\iota_{\lbe\mathcal L},
\]
whence $u=1$, i.e., $\psi=\id_{\le\mathcal L}$.
\end{proof}

By the lemma, $H^{\le -1}(\le{\rm PIC}_{\e\varepsilon}\lle( A\lbe/\lbe S\le)\le)=0$, whence ${\bm{\pi}}_{\le{\rm PIC}_{\lle\varepsilon}\lle(\lbe A\lbe/\lbe S\le)}$ \eqref{phom} 
\begin{equation}\label{dst3a}
{\rm PIC}_{\lle\varepsilon}\lle(\lbe A\lbe/\lbe S\le)\isoto\mathbf{s\lle t}(\le{\rm Pic}_{\e\varepsilon}\le(\lbe A/\lle S\e)) 
\end{equation}
is an equivalence in $({\rm CGS})$. Similarly, the projection
${\bm{\pi}}_{\le{\rm PIC}_{\lle\varepsilon}^{\le 0}\lle(\lbe A\lbe/\lbe S\le)}$ is an equivalence in $({\rm CGS})$:
\begin{equation}\label{dst3b}
{\rm PIC}^{\e 0}_{\lle\varepsilon}\lle(\lbe A\lbe/\lbe S\le)\isoto\mathbf{s\lle t}(\le{\rm Pic}^{0}_{\e\varepsilon}\le(\lbe A/\lle S\e)).
\end{equation}
Next, there exists an isomorphism in $\overline{({\rm CGS})}$
\begin{equation}\label{gmo1}
\bm{\Phi}_{\be A} \colon {\mathbf{s\lle t}}(A^{\vee})\isoto {\rm PIC}^{\e 0}_{\varepsilon}\lle( A\lbe/\lbe S\le),
\end{equation}
such that the following diagram in $\overline{({\rm CGS})}$ commutes:
\begin{equation}\label{brev4}
\xymatrix{
&&{\mathbf{s\lle t}}(\le{\rm Pic}^{\lle 0}_{\e\varepsilon}\le(\lbe A/\lbe S\e))\\
{\mathbf{s\lle t}}(A^{\vee})\ar[urr]^(.45){{\mathbf{s\lle t}}(\Phi_{\be A})}_{\simeq}
\ar[rr]^(.45){\bm{\Phi}_{\be A}}_(.45){\simeq}&&{\rm PIC}^{\e 0}_{\varepsilon}\lle( A\lbe/\lbe S\le)\le,\ar[u]_(.43){\eqref{dst3b}}^(.45){\simeq}
}
\end{equation}
where $\Phi_{\be A}$ is the isomorphism \eqref{fia}. 
In particular, $\Phi_{\be A}=H^{\le 0}\lbe(\bm{\Phi}_{\be A})$. 

\smallskip

The map \eqref{gmo1} is constructed as follows. Since ${\rm Pic}^{0}_{\e\varepsilon}\le(\lbe A/\lle S\e)=H^{\le 0}(\le{\rm PIC}^{\e 0}_{\varepsilon}\lle( A\lbe/\lbe S\le)\le)$, the Poincar\'e class $\mathcal P_{\! A}\in {\rm Pic}^{0}_{\e\varepsilon}\le(\lbe A/\lle S\e)(A^{\vee})$ is represented by a triple $(\mathcal P,\iota_{\le\mathcal P}, A^{\vee}\le)\in {\rm PIC}^{\e 0}_{\varepsilon}\lle(\lbe A\lbe/\lbe S\le)(A^{\vee})$, where 
$\mathcal P$ is an invertible sheaf on $A\times_{S}A^{\vee}$ and $\iota_{\le\mathcal P}\colon  \varepsilon_{\be A^{\lbe\vee}}^{*}\mathcal P\isoto\mathcal O_{\be A^{\lbe\vee}}$ is an isomorphism of $\mathcal O_{\be A^{\vee}}\le$-\le modules. Let
$\bm{\Phi}_{(\mathcal P,\e\iota_{\le\mathcal P},\le A^{\vee}\le)}\colon \mathbf{s\lle t}(\lbe A^{\vee})\to{\rm PIC}^{\e 0}_{\varepsilon}\lle( A\lbe/\lbe S\le)$ be the morphism in $({\rm CGS})$ such that, for every $S$-scheme $T$ and every $S$-morphism $g\colon T\to A^{\vee}$, 
\[
\bm{\Phi}_{(\mathcal P,\e\iota_{\le\mathcal P}, \le A^{\vee})}(\le g,T\e)=((\le\id_{\lle A}\be\times g\e)^{*}\lle\mathcal P,\e\iota_{\le(\id_{\lle A}\times g\e)^{\lbe *}\lbe\mathcal P},\lle T\e),
\]
where $\iota_{\le(\id_{\lle A}\times g\e)^{\lbe *}\lbe\mathcal P}\colon \varepsilon_{T}^{*}((\lle\id_{\lbe A}\be\times\be g\e)^{*}\le\mathcal P\le)\isoto \mathcal O_{\le T}$ is the composition
\[
\varepsilon_{T}^{*}((\id_{A}\be\times\be g\e)^{*}\mathcal P\le)\simeq
g^{*}\lbe(\varepsilon_{\be A^{\vee}}^{*}\mathcal P\le)
\underset{\sim}{\xrightarrow{\,g^{\lbe *}\lbe(\lbe\iota_{\lle\mathcal P}\lbe)}}g^{\lbe *}\lbe(\mathcal O_{\be A^{\vee}}\lbe)\simeq
\mathcal O_{\le T}.
\]
Then $\bm{\Phi}_{(\mathcal P,\e\iota_{\le\mathcal P},\le A^{\vee})}$ is an equivalence in $({\rm CGS})$. If $(\mathcal P^{\le\prime},\e\iota_{\le\mathcal P^{\le\prime}},\le A^{\vee}\lle)\simeq (\mathcal P,\iota_{\le\mathcal P}, A^{\vee}\lle)$ is another representative of $\mathcal P_{\! A}$, then there exists a $2$-isomorphism $\bm{\Phi}_{(\mathcal P^{\le\prime},\e\iota_{\le\mathcal P^{\le\prime}},\le A^{\vee})}\!\implies\!\bm{\Phi}_{(\mathcal P,\e\iota_{\le\mathcal P},\le A^{\vee})}$ in $({\rm CGS})$. Thus, setting $\bm{\Phi}_{\be A}=\bm{\Phi}_{(\le\mathcal P,\,\iota_{\lle\mathcal P},\e A^{\vee})}$ for any representative $(\mathcal P,\e\iota_{\le\mathcal P},\le A^{\vee})$ of $\mathcal P_{\! A}$, we obtain a well\le-\le defined isomorphism \eqref{gmo1} in $\overline{({\rm CGS})}$ such that \eqref{brev4} commutes.

Next, for every $S$-scheme $T$ and every 
$S$-morphism $f\colon T\to A$, let $(\le f,\id_{\le T})_{\le T} \colon T \to A_{\le T}$ be the canonical morphism induced by $f$. By restriction of domain, the functor $(\le f,\id_{\le T})_{\le T}^{*}\colon {\textrm{PIC}}_{S}(A_{T})\to {\textrm{PIC}}_{S}(\le T\le)$ induces a functor $(\le f,\id_{\le T})_{\le T}^{*}\colon {\rm PIC}^{\e 0}_{\varepsilon}\lle( A\lbe/\lbe S\le)(\le T\le)\to {\textrm{PIC}}_{S}(\le T\le)$. Now let $g\colon T\to A^{\vee}$ be an $S$-morphism. Since $(\le{\rm Id}_{A}\lbe\times\lbe g\e)\circ (\le f,\id_{\le T})_{\le T}=(\le f,\e g\e)_{\le T} \colon T \to A\times_{S}A^{\vee}$, we have the following equality in ${\textrm{PIC}}_{S}(\le T\le)$
\[
\hskip -1cm(\le f,\id_{\le T})_{\le T}^{*}
\le\bm{\Phi}_{\be A}(\le g,T\e)=(\le f,\le g\e)_{\le T}^{*}\le\mathcal P,
\]
where $\mathcal P$ is any fixed Poincar\'e sheaf on $A\times_{S}A^{\vee}$. We now define the pairing \eqref{stpo2} as the $2\e$-\e isomorphism class of the pairing 
\[
\langle-,-\rangle_{\le \mathcal P\lbe,\, A}\colon \mathbf{s\lle t}(A)\times_{S} \mathbf{s\lle t}(A^{\vee})\to {\textrm{PIC}}_{S}
\]
given by the formula
\begin{equation}\label{stpo3b} 
\langle\le f,\le g\le\rangle_{\le \mathcal P\lbe,\, A}(\le T\e)\overset{{\rm def.}}{=}(\le f,\le g\e)_{\le T}^{*}\le\mathcal P
\end{equation}
for every $T,f$ and $g$ as above.

Henceforth, we will describe the pairing \eqref{stpo2} more succintly by the formula
\begin{equation}\label{stpo3a} 
\hskip .5cm\langle\le f,\le g\le\rangle_{\lbe A}(\le T\e)=(\le f,\le g\e)_{\le T}^{*}\le\mathcal P_{\! A},
\end{equation}
where $\mathcal P_{\! A}\in {\rm Pic}\le(\lbe A\times_{S}A^{\vee}\e)$ is the Poincar\'e class of $A$.

\smallskip

\begin{remarks}\indent
\begin{enumerate}
\item[(a)] The pairing \eqref{stpo2}, given by the formula \eqref{stpo3a}, is the ``universal pairing" associated to $A$, in the sense that \eqref{stpo2} underlies the construction of {\it every} classical pairing defined on $A\times_{S}A^{\vee}$, e.g., the Weil pairing \cite[VIII, (2.2.5)]{sga7} and the N\'eron-Tate height pairings \cite{mt83}.
\item[(b)] If $S=\spec k$, where $k$ is a field, an explicit  \'etale Poincar\'e sheaf $\mathcal Q_{A}$ on the smooth site over $A$ is exhibited in \cite[\S4.3]{wit08}. Thus, setting $\mathcal P=\mathcal Q_{A}$ in \eqref{stpo3b}, there exists an {\it explicit} pairing
$\mathbf{s\lle t}(A)\times_{k} \mathbf{s\lle t}(A^{\vee})\to \textrm{PIC}_{\le k}$ in 
$\overline{(\textrm{CGS})}_{\mathrm{sm}}$ whose $\mathrm{fppf}$-stackification yields the pairing \eqref{stpo2}.

\item[(c)] In the setting of (b), a concrete description of the biduality isomorphism $\kappa_{A}\colon A\to A^{\vee\vee}$ \eqref{bdi} is given in \cite[Corollary 4.3.3]{wit08} in terms of the sheaf $\mathcal Q_{A}$.
\end{enumerate}
\end{remarks}

\smallskip

Next, we will relate the pairing $\langle-,-\rangle_{ A}$ \eqref{stpo2} to the evaluation pairing $\textrm{\bf{ev}}_{\lbe \mathbf{s\lle t}(A)}$ \eqref{evp}. 

\smallskip

Let
$\bm{\varepsilon}\colon S\to\mathbf{s\lle t}(A)$ and $\bm{\epsilon}\colon S\to {\textrm{PIC}}_{S}$ be the unit sections of the indicated stacks and let $\textrm{MOR}_{\,\textrm{S\le t}/S,\e \bullet}(\mathbf{s\lle t}(A),{\textrm{PIC}}_{S})\in
({\rm CGS})$ be the stack whose fiber over an $S$-scheme $T$ is the groupoid whose objects are the pairs $(\varphi,\iota\e)$, where
$\varphi\colon \mathbf{s\lle t}(A_{T})\to {\textrm{PIC}}_{\le T}$ is a morphism of $T$-stacks and $\iota\colon \varphi\le\circ\le\bm{\varepsilon}_{\le T}\Longrightarrow\bm{\epsilon}_{\le T}$ is a $2\e$-\le isomorphism, and whose morphisms $(\varphi,\iota\e)\to(\varphi^{\le\prime},\iota^{\le\prime}\e)$ are the $2\e$-\le isomorphisms $\varphi \Longrightarrow \varphi^{\le\prime}$ that are {\it strictly} compatible with $\iota$ and $\iota^{\le\prime}$. The group structure of $\textrm{MOR}_{\,\textrm{S\le t}/S,\e \bullet}(\mathbf{s\lle t}(A),{\textrm{PIC}}_{S})$ is induced by that of ${\textrm{PIC}}_{S}$. Now consider the forgetful functor in $({\rm CGS})\e$:
\begin{equation}\label{dst1}
D(\lbe \mathbf{s\lle t}(A)\le)={\rm HOM}(\mathbf{s\lle t}(A),{\textrm{PIC}}_{S})\into \textrm{MOR}_{\,\textrm{S\le t}/S,\e \bullet}(\mathbf{s\lle t}(A),{\textrm{PIC}}_{S}).
\end{equation}
By the $2\e$-Yoneda lemma for stacks over $T$ \cite[Lemma 3.4.21]{alp}, the functor
\begin{equation}\label{2Yo}
\textrm{MOR}_{\,\textrm{S\le t}/\le T}(\mathbf{s\lle t}(A_{\le T}),{\textrm{PIC}}_{\le T})\isoto {\textrm{PIC}}_{\le T}(A_{\le T}), \varphi\mapsto\mathcal L_{\le\varphi}\overset{{\rm def.}}{=}\varphi\lle(\e\id_{ A_{\lle T}},A_{\le T}\lbe),
\end{equation}
is an equivalence of groupoids. Under the preceding functor, a $2$-isomorphism $\iota\colon \varphi\circ\bm{\varepsilon}_{\le T}\Longrightarrow\bm{\epsilon}_{\le T}$ corresponds to an isomorphism of $\mathcal{O}_{\le T}\e$-\e modules $   \iota_{\varphi} \colon \varepsilon_{\lbe T}^{\le *}(\mathcal L_{\le\varphi})\isoto\mathcal{O}_{\le T}$. Further, 
a $2\e$-\le isomorphism $\alpha\colon\varphi \Longrightarrow \varphi^{\e\prime}$ corresponds to an isomorphism of  $\mathcal O_{\be A_{\lle T}}\e$-\le modules $\psi_{\alpha} \colon \mathcal L_{\le\varphi}\isoto \mathcal{L}_{\le\varphi^{\lle\prime}}$ and the strict compatibility mentioned above corresponds to an {\it equality} $\iota_{\le\varphi^{\le\prime}} \circ \varepsilon_{\lbe T}^{\le *}(\psi_{\alpha})=\iota_{\le\varphi}$.
Thus \eqref{2Yo} induces an equivalence in $({\rm CGS})$ 
\begin{equation}\label{dst2}
\textrm{MOR}_{\,\textrm{S\le t}/S,\e \bullet}(\mathbf{s\lle t}(A),{\textrm{PIC}}_{S})\isoto {\rm PIC}_{\e\varepsilon}\lle(\lbe A\lbe/\lbe S\le).
\end{equation}
Composing the functors \eqref{dst3a}, \eqref{dst1} and \eqref{dst2}, we obtain a monomorphism in $({\rm CGS})\le$:
\[
D(\lbe \mathbf{s\lle t}(A)\le)\into \mathbf{s\lle t}(\le{\rm Pic}_{\e\varepsilon}\lbe(\lbe A/\lle S\e)).
\]
By \cite[proof of Proposition 3.23]{bro21}, the preceding monomorphism factors through ${\mathbf{s\lle t}}({\rm Pic}^{0}_{\e\varepsilon}\le(\lbe A/\lle S\e))$, whence we obtain a monomorphism
\begin{equation}\label{dst4b}
D(\lbe {\mathbf{s\lle t}}(A)\le)\into {\mathbf{s\lle t}}({\rm Pic}^{0}_{\e\varepsilon}\le(\lbe A/\lle S\e)).
\end{equation}
Now let
\begin{equation}\label{omap2}
\bm{\omega}_{\be A}\colon D(\lbe {\mathbf{s\lle t}}(\lbe A)\le)\into {\mathbf{s\lle t}}(\lbe A^{\vee}\le)
\end{equation}
be defined by the commutativity of the following diagram in $({\rm CGS})$
\[
\xymatrix{D(\lbe {\mathbf{s\lle t}}(A)\le)\ar@{^{(}->}[d]_(.47){\bm{\omega}_{\be A}}\ar@{^{(}->}[rr]^(.45){\eqref{dst4b}}&& {\mathbf{s\lle t}}({\rm Pic}^{0}_{\e\varepsilon}\le(\lbe A/\lle S\e)),\\
{\mathbf{s\lle t}}(\lbe A^{\vee}\le)\ar[urr]^(.4){\mathbf{s\lle t}(\Phi_{\be A})}_(.4){\simeq}&& 
}
\]
where $\Phi_{\be A}$ is the map \eqref{fia}. Thus, if   $\varphi\colon \mathbf{s\lle t}(A_{\le T})\to{\textrm{PIC}}_{\le T}$ is a $T$-homomorphism, then $\bm{\omega}_{\be A}(\le T\le)(\varphi)=(\le g,T\le)$, where $g\colon T\to A^{\vee}$ is the unique $S$-morphism such that
$\varphi\le(\e\id_{ A_{\lle T}},A_{\le T}\lbe)=({\rm Id}_{A}\lbe\times\lbe g\e)^{*}\le\mathcal P_{\! A}$ in ${\rm Pic}^{0}_{\e\varepsilon}\le(\lbe A_{\le T}/\e T\e)$. In standard set\e-\e theoretic shorthand, we have 
 \begin{equation}\label{tit}
\varphi\le(\e\id_{\le A},A\lbe)=({\rm Id}_{A}\lbe\times\lbe \bm{\omega}_{\be A}(\varphi)\e)^{*}\le\mathcal P_{\! A}.
\end{equation}

\begin{lemma}\label{sbro} The following diagram in $\overline{({\rm CGS})}$ commutes
\[
\xymatrix{\mathbf{s\lle t}(A)\times_{S}\be D(\lbe \mathbf{s\lle t}(A)\le)
\ar[d]_{\textrm{\bf{ev}}_{\lbe \mathbf{s\lle t}(A)}}\ar@{^{(}->}[rr]^(.53){\id_{\le\mathbf{s\lle t}(\lbe A)}\times\le\bm{\omega}_{\be A}}&&\mathbf{s\lle t}(A)\times_{S} \mathbf{s\lle t}(A^{\vee})\ar[d]^(.5){\langle-,-\rangle_{ A}}_{\eqref{stpo3a}}\\
{\mathrm{PIC}}_{S}\ar@{=}[rr]&&{\mathrm{PIC}}_{S},
}
\]
where $\bm{\omega}_{\be A}$ is the map \eqref{omap2}. Thus, set\le-\le theoretically, we have 
\begin{equation}\label{nit}
\textrm{\bf{ev}}_{\lbe \mathbf{s\lle t}(A)}(\le f,\le\varphi\le)=(\le f,\bm{\omega}_{\be A}(\varphi))^{*}\le\mathcal P_{\! A}
\end{equation}
for sections $f\in A$ and $\varphi\in D(\lbe \mathbf{s\lle t}(A)\le)$.
\end{lemma}
\begin{proof} Since the statement is compatible with arbitrary base changes, we will work with $S$-sections throughout. Let $f\colon S\to A$ be an $S$-section and let $\varphi\colon \mathbf{s\lle t}(A)\to{\textrm{PIC}}_{\le S}$ be an $S$-homomorphism. By the pseudonaturality of $\varphi$, the following diagram of fiber groupoids commutes up to a $2$-isomorphism
\[
\xymatrix{\mathbf{s\lle t}(A)(A)\ar[d]\ar[rr]^{\varphi\lle(\lbe A\lbe)}&&{\mathrm{PIC}}_{S}(A)\ar[d]^{(\le f,\e\id_{ S})_{S}^{*}}\\
\mathbf{s\lle t}(A)(\le S\le)\ar[rr]^{\varphi\lle(\le S\le)}&&{\mathrm{PIC}}_{S}(\le S\le).
}
\]
The left-hand vertical map above sends an object $(\le h,A)\in \mathbf{s\lle t}(A)(A)$ to the object $(\le h\circ (\le f,\id_{S})_{S},S\le)\in \mathbf{s\lle t}(A)(\le S\le)$, where $h\in{\rm End}_{\le S}\lle(A)$.
Evaluating the component of the $2$-isomorphism alluded to above at $h=\id_{A}$, we obtain a functorial isomorphism of invertible sheaves on $S$
\[
\varphi(\le S\le)((\le f,\id_{\le S})_{S}, S\e)\simeq(f,\id_{\le S})_{S}^{*}(\le \varphi(A)(\e\id_{\le A}, A)). 
\]
The formula \eqref{nit} now follows from the above, the definition of $\textrm{\bf{ev}}_{\lbe \mathbf{s\lle t}(A)}$ \eqref{evp} and the formula \eqref{tit}.
\end{proof}

\smallskip

Next, for any $F\in\sfs$, let ${\rm EXT}(F,\G_{m,\e S}\lbe)\in({\rm CGS})$ be the {\it stack of extensions} of $F$ by $\G_{m,\e S}$, i.e., for every $S$-scheme $T$, ${\rm EXT}(F,\G_{m,\e S}\lbe)(\lle T\le)$ is the groupoid whose objects are the degree $1$ Yoneda extensions of $F_{T}$ by $\G_{m,\e T}$ in $T_{\fl}^{\le\sim}$ and whose morphisms are the equivalences between such extensions \cite[Tag 06XP]{sp}. We have $H^{-1}\lbe({\rm EXT}(F,\G_{m,\e S}\lbe))=F^{\le D}$ and $H^{\le 0}({\rm EXT}(F,\G_{m,\e S}\lbe))=E^{\le 1}\lbe(F\le)$.

Next, by \eqref{bpic0}, \eqref{beq3} and \cite[XVIII, 1.4.23]{sga4}, there exists an  equivalence in $({\rm CGS})$
\begin{equation}\label{nef}
\bm{\eta}_{F}\colon {\rm EXT}(F,\G_{m,\e S}\lbe)\isoto D(\lbe \mathbf{s\lle t}(F)\le)
\end{equation}
which is defined as follows: for every $S$-scheme $T$, $\bm{\eta}_{F}(\le T\le)$ maps a degree $1$ extension $\mathcal E\colon 0\to \G_{m,\le T}\to E\to F_{\le T}\to 0$ in $T_{\be\fl}^{\le\sim}$ to the $T$-homomorphism $\varphi_{\le\mathcal E}\colon \mathbf{s\lle t}(F_{\le T})\to {\textrm{PIC}}_{\le T}$ such that $\varphi_{\mathcal E}(\le T^{\le\prime}\le)(a)=
a^{*}\lbe(E\le)\wedge^{\mathbb{G}_{m,T^{\lle\prime}}}\! \mathcal{O}_{\e T^{\lle\prime}}$ for every $a\in F_{T}(\le T^{\le\prime}\le)=\Hom_{\, T_{\fl}^{\le\sim}}\lbe( T^{\le\prime},F_{\le T})$, where $a^{*}\lbe(\lbe E\le)$ denotes the $\mathbb{G}_{m,T^{\le\prime}}\e$-\e torsor over $T^{\le\prime}$ associated to the extension $a^{*}\lbe\mathcal E$ \cite[VII, \S1.2.2]{sga7}.

\smallskip

Now assume that $F^{\le D}=H^{-1}\lbe({\rm EXT}(F,\G_{m,\e S}\lbe))=0$, so that the canonical projection
${\bm{\pi}}_{{\rm EXT}}\colon {\rm EXT}(F,\G_{m,\e S}\lbe)\to \mathbf{s\lle t}(E^{\le 1}(F\e))$ \eqref{phom} is an isomorphism in $\overline{({\rm CGS})}$. Via the identification $D(F\le)=E^{\le 1}\lbe(F\le)$ \eqref{obv2}, the isomorphism of functors \eqref{amor} evaluated at $F$
\begin{equation}\label{dial}
{\bm{\alpha}}_{F}\colon D(\mathbf{s\lle t}(F\e))\isoto \mathbf{s\lle t}(E^{\le 1}\lbe(F\le))\qquad\textrm{(\le if $F^{\lle D}=0$)}
\end{equation}
is the following composition of isomorphisms in $\overline{({\rm CGS})}$:
\[
D(\mathbf{s\lle t}(F\e))\underset{\bm{\eta}_{F}^{-1}}{\Isoto}{\rm EXT}(F,\G_{m,\e S}\lbe)\underset{{\bm{\pi}}_{{\rm EXT}}}{\Isoto}\mathbf{s\lle t}(E^{\le 1}\lbe(F\le)).
\]
We now apply the above discussion to the abelian $S$-scheme $F=A$. By definition, the {\it Barsotti\e--\le Weil morphism} for $A$ 
\begin{equation}\label{bwm}
{\bm{\beta}}_{\be A}\colon \mathbf{s\lle t}(E^{\le 1}\lbe(\lbe A))\to \mathbf{s\lle t}(A^{\vee}\le)
\end{equation}
is the following composition of morphisms in $\overline{({\rm CGS})}\e$:
\[
\mathbf{s\lle t}(E^{\le 1}\lbe(\lbe A))\underset{{\bm{\alpha}}_{\lbe A}^{-1}}{\Isoto}
D(\lbe \mathbf{s\lle t}(A)\le)\underset{\bm{\omega}_{\lbe\lbe A}}{\into}\mathbf{s\lle t}(A^{\vee}),
\]
where $\bm{\alpha}_{A}$ and $\bm{\omega}_{\be A}$ are the  morphisms \eqref{dial} (with $F=A$) and \eqref{omap2}, respectively. By \cite[VII, 1.3.8.2]{sga7} (see also \cite[Corollary 3.6]{rr25}), the map \eqref{bwm} is an {\it isomorphism} in $\overline{({\rm CGS})}$. Consequently,
\begin{equation}\label{omap3}
\bm{\omega}_{\be A}\colon D(\lbe {\mathbf{s\lle t}}(\lbe A)\le)\isoto {\mathbf{s\lle t}}(\lbe A^{\vee}\le)
\end{equation}
is an isomorphism as well. We will write
\begin{equation}\label{bwf}
\beta_{A}\overset{{\rm def.}}{=}H^{\le 0}(\le {\bm{\beta}}_{\be A})\colon E^{\le 1}\lbe(\lbe A)\isoto A^{\vee}
\end{equation}
for the isomorphism in $\sfs$ induced by \eqref{bwm} and call it the {\it Barsotti\e--\e Weil isomorphism} for $A$.
Note that ${\bm{\beta}}_{\be A}=\mathbf{s\lle t}(\le \beta_{A})$.

\begin{proposition}\label{xxx} The following diagram in $\overline{({\rm CGS})}$ commutes:
\[
\xymatrix{\mathbf{s\lle t}(A)\ar[d]_(.47){\mathbf{s\lle t}(\kappa_{\lbe A})}^(.45){\simeq}\ar[rr]^(.4){\bm{e}_{\le\mathbf{s\lle t}(A)}}&& DD(\lbe \mathbf{s\lle t}(A)\lle)\\
\mathbf{s\lle t}(A^{\vee\vee})\ar[rr]^(.45){\bm{\omega}_{\!\lbe A^{\lbe\lbe\vee}}^{\le -1}}_(.42){\simeq}&& D\lbe(\lbe \mathbf{s\lle t}(A^{\vee})\lle)\ar[u]_{D(\bm{\omega}_{\be A})}^{\simeq},
}
\]
where $\kappa_{\lbe A}$ is the biduality isomorphism \eqref{bdi}, $\bm{e}_{\le\mathbf{s\lle t}(A)}$ is the evaluation pairing \eqref{std} and $\bm{\omega}_{\be A}$ and $\bm{\omega}_{\be A^{\vee}}$ are instances of \eqref{omap3}. In particular, $\mathbf{s\lle t}(A)$ is a reflexive object of $\overline{({\rm CGS})}$.
\end{proposition} 
\begin{proof} As in the proof of Lemma \ref{sbro}\,, we will work 
with $S$-sections throughout and, for simplicity, omit $S$ from the notation. Let $f\in A$ be any section. By \eqref{nit} and the definition of $\bm{e}_{\le\mathbf{s\lle t}(A)}$ \eqref{std}, for any section $\varphi\in D(\lbe \mathbf{s\lle t}(A)\lle)$ we have 
\[
\bm{e}_{\le\mathbf{s\lle t}(A)}(\le f\le)(\le\varphi\le)=\textrm{\bf{ev}}_{\lbe \mathbf{s\lle t}(A)}(\le f, \varphi)=(\le f,\bm{\omega}_{\be A}(\varphi))^{*}\le\mathcal P_{\! A}.
\]
On the other hand, by \eqref{bdi2} and \eqref{nit} for $A^{\vee}$, we have
\[
\begin{array}{rcl}
(\le f,\bm{\omega}_{\be A}(\varphi))^{*}\le\mathcal P_{\! A}&=&(\le\bm{\omega}_{\be A}(\varphi),\le f\le)^{*}({\rm Id}_{A^{\vee}}\be\times \kappa_{\lbe A}\lbe)^{*}\le\mathcal P_{\! A^{\vee}}=(\le\bm{\omega}_{\be A}(\varphi),\le \kappa_{\lbe A}(f\le))^{*}\le\mathcal P_{\! A^{\vee}}\\
&=& \textrm{\bf{ev}}_{\lbe \mathbf{s\lle t}(A^{\vee})}(\le \bm{\omega}_{\be A}(\varphi),\bm{\omega}_{\be A^{\vee}}^{-1}(\kappa_{\lbe A}(f\le)))=\bm{\omega}_{\be A^{\vee}}^{-1}(\kappa_{\lbe A}(f\le))(\bm{\omega}_{\be A}(\varphi))\\
&=&(\bm{\omega}_{\be A^{\vee}}^{-1}(\kappa_{\lbe A}(f\le))\circ\bm{\omega}_{\be A})(\varphi)=D(\bm{\omega}_{A})(\bm{\omega}_{\be A^{\vee}}^{-1}(\kappa_{\lbe A}(f\le)))(\varphi).
\end{array}
\]
Thus $\bm{e}_{\le\mathbf{s\lle t}(A)}(\le f\le)=D(\bm{\omega}_{A})(\bm{\omega}_{\be A^{\vee}}^{-1}(\kappa_{\lbe A}(f\le)))$ in $\overline{({\rm CGS})}$, whence the proposition follows.
\end{proof}

\begin{corollary}\label{abr} If $A$ is an abelian $S$-scheme, then $A$ is a reflexive object of $D^{\e[-1,0\le]}\lbe(S_{\fl})$.    
\end{corollary}
\begin{proof} This follows from the proposition and the comment right after \eqref{amor}.
\end{proof}

\section{Abelian topological groups}\label{tops}

Let $G$ be an abelian topological group. If $S$ is a subset of $G$, the closure of $S$ in $G$ will be denoted by ${\rm cl}_{\e G}(\be S\lle)$ (or by $\overline{S}$ if the ambient group $G$ is clear from the context). Recall that $G$ is called {\it second countable} if $G$ has a countable basis. Recall also that a morphism of topological groups (i.e., a continuous homomorphism) $u\colon G\to H$ is called {\it strict}\,\,\footnote{\e Note that in \cite{str}, which is one of our references for the theory of locally compact abelian topological groups, strict morphisms are called {\it proper} \cite[Definition 20.14]{str}.} (or {\it continuous and strict}, for emphasis) if the induced map $G/\e\krn u\to\img u$ is an isomorphism of topological groups, where $G/\e\krn u$ is equipped with the quotient topology and $\img u\subset H$ is equipped with the subspace topology. Equivalently, $u$ is strict if the map $G\to \img u$ induced by $u$ is open \cite[\S III.2.8, Proposition 24, p.~236]{bou}\e.

\begin{lemma}\label{lem:stc} Let $u\colon G\to H$ be a morphism of abelian topological groups. 
\begin{enumerate}
\item[(i)] If $G$ is compact and $H$ is Hausdorff, then $u$ is strict.
\item[(ii)] If $G$ is locally compact and second countable, $H$ is locally compact and Hausdorff and $\img u$ is closed in $H$, then $u$ is strict.
\end{enumerate}	
\end{lemma}
\begin{proof} For (i), see \cite[III, \S2.8, Remark 1, p.~237]{bou}. In (ii), $\img u$ is locally compact and Hausdorff by \cite[Lemma 1.20(b)]{str}. Thus, by \cite[Theorem 5.29 and (5.30)]{hr}, the map $G\to \img u$ induced by $u$ is open, which yields (ii).
\end{proof}

An exact sequence of abelian topological groups $\dots\to A_{\le i}\overset{\!f_{\lbe i}}{\to} A_{\le i+1}\overset{\!f_{\lbe i+1}}{\to} A_{\le i+2}\to\dots$ is called {\it strict exact} if every map $f_{\lbe i}$ is strict. A {\it topological extension} is a short strict exact sequence of abelian topological groups. Every such sequence is homeomorphic to an exact sequence of abelian topological groups of the form $0\to H\to G\to G\be/\be H\to 0$, where the second map is the inclusion of a subgroup $H$ of $G$ equipped with the subspace topology and $G\be/\be H$ is equipped with the quotient topology. Recall that $G\be/\be H$ is Hausdorff if, and only if, $H$ is closed in $G$ \cite[Proposition 6.6]{str}.

\smallskip

An abelian topological group $G$ is called {\it profinite} if $G$ is homeomorphic to an inverse limit of finite and discrete groups equipped with the inverse limit topology.  Equivalently, $G$ is Hausdorff, compact and totally disconnected \cite[\S1.1, Proposition 0]{ser3}. If $G$ is an abelian topological group, the {\it profinite completion of $G$} is the profinite group $G^{\le\wedge}=\varprojlim G\lbe/\lbe U$, where $U$ ranges over the family of open subgroups of finite index of $G$. The map $c_{\le G}\colon G\to G^{\le\wedge}$ induced by the projections $G\to G/U$ is continuous with dense image, i.e., ${\rm cl}_{\le G^{\lle\wedge}}\be(\le \img c_{\le G})=G^{\le\wedge}$ \cite[p.~2 and Lemma 1.1.7]{rz}, and has the following universal property: if $u\colon G\to H$ is a morphism of abelian topological groups, where $H$ is profinite, then there exists a unique morphism of profinite abelian topological groups $u^{\wedge}\colon G^{\le\wedge}\to H$ such that $u=u^{\wedge}\circ c_{\le G}$. The functor $G\mapsto G^{\le\wedge}$ is right-exact in the following sense: if $0\to F\to G\to H\to 0$ is a {\it strict} exact sequence of abelian topological groups, then $F^{\le\wedge}\to G^{\le\wedge}\to H^{\lbe\wedge}\to 0$ is a strict exact sequence of profinite abelian topological groups. See Lemma \ref{lem:stc}(i) and \cite[Proposition C.1.4]{ros23}. If $G$ is profinite then, by the universal property, the map $c_{\le G}\colon G\isoto G^{\le\wedge}$ is a topological isomorphism. Henceforth, every profinite group $G$ will be identified with its profinite completion $G^{\le\wedge}$ via $c_{\le G}$.

\begin{lemma} \label{lem:eas} Let $u\colon G\to H$ be a morphism of abelian topological groups and let $c_{\le H}\colon H\to H^{\wedge}$ be the canonical map. Then $\img\be(\lbe u^{\wedge}\lbe)={\rm cl}_{H^{\lbe\wedge}}\lbe( c_{\lle H}(\le\img u)\le)$.
\end{lemma}
\begin{proof} Since $u^{\wedge}\colon G^{\le\wedge}\to H^{\lle\wedge}$ is strict by Lemma \ref{lem:stc}(i), $\img(\lbe u^{\wedge}\lbe)$ is a compact subgroup of $H^{\lle\wedge}$. Therefore, $\img(\lbe u^{\wedge}\lbe)$ is closed in $H^{\lbe\wedge}$ \cite[Lemma 1.20(a) and Theorem 6.7(b)]{str}. Further, since $c_{H}\circ u=u^{\wedge}\circ c_{\le G}$, we have $c_{\lle H}(\le\img u)=u^{\wedge}(\le \img c_{\le G})\subseteq \img\be(\lbe u^{\wedge}\lbe)$. Consequently,
\[
\begin{array}{rcl}
{\rm cl}_{H^{\lbe\wedge}}\lbe( c_{\lle H}(\le\img u)\le)&\subseteq& {\rm cl}_{H^{\lbe\wedge}}\be(\le\img\be(\lbe u^{\wedge}\lbe))=
\img\be(\lbe u^{\wedge}\lbe)=u^{\wedge}\lbe({\rm cl}_{\le G^{\wedge}}\be(\le \img c_{\le G}))\\
&\subseteq& {\rm cl}_{H^{\lbe\wedge}}\be(u^{\wedge}\lbe(\img c_{\le G}))={\rm cl}_{H^{\lbe\wedge}}\be( c_{\lle H}(\le\img u)\le),
\end{array}
\]
whence the lemma follows.
\end{proof}

\begin{lemma}\label{prev} Let
\[
\xymatrix{G\ar@{->>}[d]_(.43){u}\ar@{->>}[r]^(.45){f}& K \ar[d]^(.45){v}\\
H\ar[r]^{g}& L
}
\]
be a commutative diagram in the category of abelian topological groups, where $f$ and $u$ are surjective and $g$ is an open map. If $u$ is strict, then $v$ is strict.
\end{lemma}
\begin{proof} Let $U$ be an open subset of $K$. By the commutativity of the diagram and the surjectivity of $f$, we have $v(U\le)=v(f(f^{-1}(U\le)))=g(u(f^{-1}(U\le)))$. Now, since $u(f^{-1}(U\le))$ is open in $\img u=H$ and $g$ is an open map,  $v(U\le)=g(u(f^{-1}(U\le)))$ is open in $g(\img u)=v(\img f)=\img v$. This completes the proof.
\end{proof}

If $G$ is an abelian topological group, then the abelian group $G_{\rm Haus}\overset{{\rm def.}}{=} G/\,\overline{\{0\}}$ equipped with the quotient topology is the maximal Hausdorff quotient of $G$, i.e., if $q_{\le G}\colon G\to G_{\rm Haus}$ is the canonical projection and 
$u\colon G\to H$ is a morphism of abelian topological groups, then there exists a commutative diagram of abelian topological groups
\begin{equation}\label{diag:hq}
\xymatrix{G\ar@{->>}[r]^(.4){q_{\le G}}\ar[d]_(.45){u}& G_{\rm Haus} \ar[d]^(.5){u_{\lle\rm Haus}}\\
 H\ar@{->>}[r]^(.4){q_{\lle H}}& H_{\rm Haus}
}
\end{equation}
where $u_{\rm Haus}$ is uniquely determined by $u$. 

\begin{remark}\label{bas} It follows without difficulty from the universal properties that the map $q_{\le G}^{\le\wedge}\colon G^{\le\wedge}\to (G_{\rm Haus}\lbe)^{\le\wedge}$ is an isomorphism of profinite abelian topological groups. Thus, if $G_{\rm Haus}$ is profinite, then 
$G_{\rm Haus}$ and $G^{\le\wedge}$ are canonically isomorphic.
\end{remark}

\begin{lemma} \label{nos} Let $A\overset{\! f}{\to} B\overset{\! g}{\to} C\overset{\! h}{\to}D$ be an exact sequence of abelian topological groups, where $A$ and $C$ are Hausdorff. Then
\[
A\,\xrightarrow{f_{\lle \rm Haus}} B_{\le \rm Haus}
\xrightarrow{g_{\le \rm Haus}} C\overset{\! h}{\lra}D
\]	
is an exact sequence of abelian topological groups.
\end{lemma}
\begin{proof} By the commutativity of diagram \eqref{diag:hq} and the identifications $q_{\le A}={\rm Id}_{A}$ and $q_{\le C}={\rm Id}_{\le C}$, we have $q_{\lle B}\circ f=f_{\lle \rm Haus}$ and $g_{\le \rm Haus}\circ q_{B}=g$ . Now an application of Lemma \ref{ker-cok} to the pair of maps $B\overset{\! q_{B}}{\twoheadrightarrow} B_{\le \rm Haus}\overset{\! g_{\le \rm Haus}}{\lra} C$ (whose composition is $g$) shows that $\krn g_{\le \rm Haus}=\krn g/\e\krn q_{\lle B}=q_{\lle B}(\e\img f\e)=\img(\e q_{\lle B}\circ f\e)=\img f_{\lle \rm Haus}$. Finally, $\img g_{\e \rm Haus}=\img(\e g_{\le \rm Haus}\circ q_{\lle B})=\img g=\krn h$.
\end{proof}

\smallskip

\subsubsection{}\label{extp} 
A property $P$ of abelian topological groups is called an {\it extension property} if the following holds: if $G$ is an abelian topological group, $H$ is a subgroup of $G$ and $H$ and $G\be/\be H$ have property $P$, then $G$ has property $P$. In this paper we will need the following extension properties:
\begin{enumerate}
\item[(a)] Hausdorff
\item[(b)] second countable
\item[(c)] profinite
\item[(d)] locally compact
\item[(e)] discrete
\item[(f)] totally disconnected.
\end{enumerate}
For (a), see \url{https://mathoverflow.net/questions/175892}\e. For (b), see \cite[Proposition 2.C.8(2)]{cd} and also \url{https://mathoverflow.net/questions/406916}. For (c)-(f), see \cite[Proposition 6.9(c) and Theorem 6.15]{str}.

\smallskip

Additionally, we will need the following facts.
\begin{enumerate}
\item[(i)] If $G$ is a locally compact abelian topological group and $H$ is a subgroup of $G$, then  $G\be/\be H$ is locally compact.
\item[(ii)] Closed subgroups of locally compact abelian topological groups are locally compact.
\item[(iii)] If $G$ is locally compact and totally disconnected and $H$ is a closed subgroup of $G$, then $G\be/\be H$ is totally disconnected.
\end{enumerate}
See \cite[Lemma 1.20(b), Theorem 6.7(c) and Proposition 6.9(b)]{str}\,.

\smallskip

\subsubsection{}\label{pond} 
If $G$ is a locally compact abelian topological group, we let $G^{\le*}=\Hom_{\e\text{cts}}\lbe(G,\T\e)$ be its Pontryagin dual, where
$\T=\{z\in\C\colon |z|=1\}\simeq \R/\Z$ is equipped with its natural topology. We equip $G^{\le*}$ with the compact-open topology, i.e., the topology generated by subsets of the form $\{\chi\in G^{\le*}\colon \chi(\le C\le)\subseteq U \}$, where $C\subseteq G$ is compact and $U\subseteq \T$ is open \cite[Definition 9.1]{str}. Since $\mathbb T$ is Hausdorff, $G^{\le*}$ is Hausdorff as well (even if $G$ is not) \cite[Lemma 9.2(c)]{str}. Further, by \cite[Theorem 23.8]{hr} and \cite[Theorem 20.5]{str}, the canonical map $q_{\e G}^{*}\colon (G_{\lbe\rm Haus}\lbe)^{\le*}\to G^{\le*}$ is an isomorphism of Hausdorff and locally compact abelian topological groups.
We will use the following facts without further comment: 
\begin{enumerate}
\item[(a)] the Pontryagin dual of a  strict exact sequence of Hausdorff and locally compact abelian topological groups is a sequence of the same type \cite[Theorem 23.7]{str};
\item[(b)] the Pontryagin dual of a profinite (respectively, discrete and torsion) abelian topological group is a discrete and torsion (respectively, profinite) abelian topological group \cite[Theorems 19.9, 22.6 and 23.30]{str};
\item[(c)] the Pontryagin dual of a (Hausdorff and locally compact) second countable abelian topological group is second countable \cite[Theorem 24.14]{hr}\e.
\end{enumerate}

\smallskip

\subsubsection{}\label{du}
The duality theorems to be established in this paper will only involve groups that belong to the class $\s C$ of Hausdorff and locally compact abelian topological groups which are iterated topological extensions of abelian topological groups of one of the following types: 
\begin{enumerate}
\item[(1)] profinite,
\item[(2)] discrete and torsion or
\item[(3)] locally profinite and of finite exponent.
\end{enumerate}
Recall that an abelian topological group $G$ is called {\it locally profinite} if $G$ is Hausdorff, locally compact and totally disconnected. Equivalently, every neighborhood of $0\in G$ contains a profinite open subgroup (see, e.g., \cite[Theorem 1 and its proof, p.~118]{gru}). It follows from  \ref{extp}(iii) above that a Hausdorff quotient of a group of type (3) is a group of the same type. The class $\s C$ has the following properties:
\begin{enumerate}
\item[(i)] If $G\in\s C$, then $G^{\le*}\in \s C$.
\item[(ii)] If $G\in\s C$, then there exists a canonical topological isomorphism $G^{\le*}\simeq\Hom_{\e\rm{cts}}\lbe(G,(\Q/\Z\e)_{\rm disc})$, where the group on the right is equipped with the compact open topology and $(\Q/\Z\e)_{\rm disc}$ denotes $\Q/\Z$ with its discrete topology.
\end{enumerate}
For (i), see \ref{pond}, (a) and (b), and \cite[Proposition 23.29]{str}. Property (ii) is stated (without proof) in \cite[p.~2]{ros23}\,\e\footnote{\e A proof could be obtained by using the fact every $G\in\s C$ has a neighborhood base of $0$ consisting of profinite open subgroups, so that every $\chi\in G^{\le*}$ is a character of an extension of discrete torsion groups.}\e.

\smallskip

Throughout, whenever $\Q/\Z$ occurs as the target of a topological pairing, we equip it with the discrete topology.

\smallskip

If $A, B\in\s C$ and $A\times B$ is equipped with the product topology, then a continuous pairing (i.e., bi-additive map) $A\times B\to \Q/\Z$ is called {\it perfect} if the induced maps $A\to B^{\le*}$ and $B\to A^{*}$ are topological isomorphisms.

\begin{lemma}\label{cpa} A pairing of abelian topological groups  $\langle-,-\rangle\colon A\times B\to \Q/\Z$ is continuous if, and only if, the following conditions hold:
\begin{enumerate}
\item[(i)] $\langle-,-\rangle$ is continuous at $(0,0)$, i.e., there exist open neighborhoods $U\!$ of $\e 0_{\lbe A}$ and $V\!$ of $\e 0_{\lbe B}$ such that $\langle\e U,V\e\rangle=\{0\}$.
\item[(ii)] For every $b\in B$, the homomorphism $\langle\e -,b\e\rangle\colon A\to \Q/\Z, a\mapsto \langle\e a,b\e\rangle,$ is continuous at $0$, i.e., there exists an open neighborhood $U\!$ of $\e 0_{\lbe A}$ such that $\langle\e U,b\e\rangle=\{0\}$. Similarly with the roles of $A$ and $B$ interchanged.
\end{enumerate}
\end{lemma}
\begin{proof} If the given pairing is continuous, then it is continuous at $(a,0_{\lbe B})$ and $(0_{\lbe A},b)$ for every $a\in A,b\in B$ and (i) and (ii) are immediate. Conversely, to establish the continuity of the pairing at $(a,b\le)\in A\times B$, we need to find open neighborhoods $U^{\prime\prime}$ of $\e 0_{\lbe A}$ and $V^{\prime\prime}\!$ of $\e 0_{\lbe B}$ such that $\langle\e a + U^{\prime\prime}, b + V^{\prime\prime}\e\rangle=\{\langle\e a, b\e\rangle\}$. Choose open neighborhoods $U\!$ of $\e 0_{\lbe A}$ and $V\!$ of $\e 0_{\lbe B}$ as in (i). By (ii), there exist open neighborhoods $U^{\prime}\!$ of $\e 0_{\lbe A}$ and $V^{\prime}\!$ of $\e 0_{\lbe B}$ such that $\langle\e a,V^{\prime}\e\rangle=\langle \e U^{\prime},b\e\rangle=\{0\}$. Now, set $U^{\prime\prime}=U\cap U^{\prime}$ and $V^{\prime\prime}=V\cap V^{\prime}$. Then
\[
\{\langle \e a,b\e\rangle\}\subseteq\langle \e a + U^{\prime\prime}, b + V^{\prime\prime}\e\rangle\subseteq \langle \e a,b\e\rangle+\langle \e a,V^{\prime}\e\rangle+\langle \e U^{\lle\prime},b\e\rangle+\langle \e U,V\e\rangle=\{\langle \e a,b\e\rangle\},
\]
which completes the proof.
\end{proof}

\begin{remarks}\label{dct}\indent
\begin{enumerate}
\item[(a)] If one of the groups $A$ or $B$ in the lemma is discrete, then only one of the two conditions in (ii) needs to be checked.
\item[(b)] If the pairing $\langle-,-\rangle\colon A\times B\to \Q/\Z$ is continuous, then the induced pairing $A_{\le\rm Haus}\times B_{\le\rm Haus}\to \Q/\Z$ is also continuous. This is immediate from the lemma using the fact that the canonical projections
$A\to A_{\le\rm Haus}$ and $B\to B_{\le\rm Haus}$ are open maps \cite[Lemma 6.2(a)]{str}
\end{enumerate}
\end{remarks}

\begin{proposition}\label{pt} Let $\langle-,-\rangle\colon A\times B\to \Q/\Z$ be a pairing of locally compact abelian topological groups. Assume that
\begin{enumerate}
\item[(i)] $\langle-,-\rangle$ is continuous at $(0,0)$,
\item[(ii)] $\langle-,b\e\rangle\colon A\to \Q/\Z$ is continuous at $0$ for every $b\in B$, and
\item[(iii)] the map $B\to A^{*}=\Hom_{\e{\rm cts}}\lbe(A,\Q/\Z\e), b\mapsto \langle-,b\rangle$, is continuous at $0$.
\end{enumerate}
Then $\langle-,-\rangle$ is continuous.
\end{proposition}
\begin{proof} By Lemma \ref{cpa}\,, it suffices to check that $\langle a,-\e\rangle\colon B\to \Q/\Z$ is continuous at $0$ for every $a\in A$. Given $a\in A$, choose a compact neighborhood $C\subseteq A$ of $a$. Since $\{\e\chi\in A^{*}\colon \chi(C\e)=\{0\}\}$ is an open neighborhood of $0_{\lbe A^{\lbe*}}$ and (iii) holds, there exists an open neighborhood $V\!$ of $\e 0_{\lbe B}$ such that $\langle\e C,b\e\rangle=\{0\}$ for every $b\in V$. Consequently, $\langle\e a,V\e\rangle\subseteq\langle\e C,V\e\rangle=\{0\}$, i.e., $\langle\e a,V\e\rangle=\{0\}$. 
\end{proof}

We now discuss the topology of cohomology groups.

\smallskip

Let $k$ be a local field. Then $k\simeq\mathbb R$ or $\mathbb C\e$ or $k$ is non-archimedean. In the latter case, either $k$ is a $p\e$-adic field, i.e., a finite extension of $\Q_{\lle p}$, or $k$ is a local function field, i.e., a finite extension of $\mathbb{F}_{\be p}(\!(t)\!)$. By \cite[Chapter I\e]{wei} and \cite[Corollary to Proposition 13, p.~VII.22]{bint}, a local field is a non-discrete, Hausdorff, locally compact and second countable topological field. Moreover, $k^{\times}$ is open in $k$ and the inversion map $k^{\times}\to k^{\times}, x\mapsto x^{-1}$, is continuous \cite[Theorems 16.11 and 16.12]{tf}. If $k$ is a local field and $X$ is an algebraic $k$-scheme, $X(k)$ will be equipped with the topology defined in \cite[\S3]{con}. Further, if $G$ is a commutative algebraic $k$-group and $r\geq 1$, $H^{\le r}\lbe(k, G\e)$ will be endowed with the topology defined in \cite[\S3]{ces}.

\begin{proposition} \label{topg} Let $k$ be a local field and let $G$ be a commutative algebraic $k$-group. Then $G(k)$ is Hausdorff, locally compact and second countable. If $k$ is non-archimedean, then $G(k)$ is totally disconnected. If $G$ is proper, then $G(k)$ is compact. If $k$ is non-archimedean and $G$ is proper, then $G(k)$ is profinite.
\end{proposition}
\begin{proof} Since $k$ is locally compact and $G$ is separated over $k$ \cite[${\rm VI}_{\rm A}$, 0.3]{sga3}, $G(k)$ is Hausdorff and locally compact. Further, if $k$ is non-archimedean, then $G(k)$ is totally disconnected. See \cite[Propositions 3.1 and 5.4]{con}. Now, to check that $G(k)$ is second countable, we may cover $G$ by finitely many open affine subsets, embed each of the latter into some affine space $\mathbb A_{\! k}^{\le n}$ and then apply \cite[Proposition 2.A.12(2)]{cd}. Finally, if $G$ is proper over $k$, then $G(k)$ is compact \cite[Lemma 2.12.1(b)]{ces} and therefore profinite if $k$ is non-archimedean.
\end{proof}

\begin{proposition} \label{fc} Let $k$ be a local field and let $G$ be a commutative algebraic $k$-group. Then $H^{\lle 1}\lbe(k, G\e)$ is torsion, Hausdorff, locally compact, totally disconnected and second countable. If $G$ is affine (respectively, smooth), then $H^{\le 1}\lbe(k, G\e)$ is of finite exponent (respectively, discrete). If $r\geq 2$, then $H^{\lle r}\lbe(k, G\e)$ is discrete and torsion.
\end{proposition}
\begin{proof} The abelian group $H^{\lle r}\lbe(k, G\e)$ is discrete if $r\geq 2$ (respectively, torsion if $r\geq 1$) by \cite[Proposition 3.5(c)]{ces} (respectively, the argument in \cite[proof of Lemma 3.2.1]{ros23}). All remaining assertions, except for the fact that $H^{\lle 1}\lbe(k, G\e)$ is totally disconnected and second countable when  $G$ is not affine, are established in \cite[Propositions 3.5(a), 3.7(c) and 3.9]{ces} and \cite[Proposition 3.3.4 and Lemma 3.4.1]{ros23}. We will now show that $H^{\le 1}\lbe(k, G\e)$ is totally disconnected and second countable. Write $G$ as an extension
$0\to L\to G\to A\to 0$, where $L$ is affine and $A$ is an abelian variety. By Proposition \ref{topg}\,, Lemma \ref{lem:stc}(i) and \cite[Propositions 4.2 and 4.3(c)]{ces}, the given extension induces a strict exact sequence of Hausdorff and locally compact abelian topological groups
\[
A(k)\xrightarrow{\partial} H^{\le 1}\lbe(k, L\le)\xrightarrow{f} H^{\le 1}\lbe(k, G\e)\to H^{\le 1}\lbe(k, A\le).
\]
By \cite[Proposition 3.3.4]{ros23}, $H^{\le 1}\lbe(k, L\le)$ is totally disconnected and second countable, whence the Hausdorff group $H^{\le 1}\lbe(k, L\le)/\partial\le(\lbe A(k))=H^{\le 1}\lbe(k, L\le)/f^{-1}(0)$ has the same properties. Now, since $A$ is smooth, $H^{\le 1}\lbe(k, A\le)$ is discrete and therefore totally disconnected. We claim that $H^{\le 1}\lbe(k, A\le)$ is second countable. This is clear if $k$ is archimedean since in this case $H^{\le 1}\lbe(k, A\le)$ is finite \cite[Remark I.3.7]{adt}. If $k$ is non-archimedean then, by Milne\e--\le Tate duality, $H^{\le 1}\lbe(k, A\le)$ is topologically isomorphic to the Pontryagin dual of $A^{\vee}\lbe(k)$ (see Remark \ref{ltd} below). Thus, since $A^{\vee}\lbe(k)$ is second countable by Proposition \ref{topg}\,, $H^{\le 1}\lbe(k, A\le)$ is second countable as well by \ref{pond}(c). We conclude that $H^{\le 1}\lbe(k, G\e)$ is a topological extension of totally disconnected and second countable abelian topological groups, so it has the same properties by \eqref{extp}, (b) and (f).
\end{proof}

\begin{remark}\label{pon} It follows from the proof that $H^{\le 1}\lbe(k,G\e)$ is a topological extension of a discrete and torsion group by a locally profinite group of finite exponent. Consequently, $H^{\le 1}\lbe(k,G\e)$ belongs to the class $\s C$ described in \ref{du} above, whence $H^{\le 1}\lbe(k,G\e)^{*}=\Hom_{\e\rm{cts}}\lbe(H^{\le 1}\lbe(k,G\e),\Q/\Z\e)$ by \ref{du}(ii). 
\end{remark}

\begin{corollary}\label{got} Let $k$ be a local field and let $0\to G_{\lbe 1}\overset{\!i}{\to} G_{2}\to G_{3}\to 0$ be an exact sequence of commutative algebraic $k$-groups. Then the induced sequence of Hausdorff and locally compact abelian topological groups
\[
0\to G_{\lbe 1}\lbe(k)\to G_{2}\lbe (k)\to G_{3}\lbe (k)\overset{\partial}{\to} H^{\le 1}\lbe(k,G_{\lbe 1}\le)\overset{i^{\lle (\lbe 1\lbe)}}{\to} H^{\le 1}\lbe(k,G_{2}\le)\to H^{\le 1}\lbe(k,G_{3}\le)\to H^{\le 2}\lbe(k,G_{\lbe 1}\le)
\]
is strict exact. Further, if $G_{2}$ {\rm (}respectively, $G_{3}${\rm )} is smooth, then $\partial$ {\rm (}respectively, $i^{\lle (1)}${\rm )} is open.
\end{corollary}
\begin{proof} By Propositions \ref{topg} and \ref{fc} and \cite[Proposition 4.2]{ces}, in the above sequence every group is Hausdorff, locally compact and second countable and every map is continuous with closed image. Now Lemma \ref{lem:stc}(ii) yields the first assertion of the corollary. The second assertion is \cite[Proposition 4.3, (b) and (c)]{ces}\e.
\end{proof}

\begin{remark}\label{ltd} If $A$ is an abelian variety over a $p\e$-adic field $k$, then the continuity of the Tate pairing
\begin{equation}\label{tpa}
\langle-,-\rangle_{\le 1\le,\le A}\colon H^{1}(k,A)\times A^{\vee}\lbe(k)\to\Q/\Z
\end{equation}
was established in \cite[pp.~268-269]{ta} using the fact that, for every $m\in \N$, $mA^{\vee}\lbe(k)$ is an open subgroup of $A^{\vee}\lbe(k)$ (the latter holds since $A^{\vee}\lbe(k)/mA^{\vee}\lbe(k)$ injects into the finite discrete group $H^{1}(k,A^{\vee}_{\le m})$). In the local function field case, we have been unable to find in the literature a proof of the continuity of $\langle-,-\rangle_{\le 1\le,\le A}$ (this continuity is claimed in \cite[lines 15\e-16, p.~274]{mil} but no proof is given there). Note that, since $A^{\vee}\lbe(k)/p\le A^{\vee}\lbe(k)$ is often infinite if $k$ is a local function field of characteristic $p$ \cite[Remark A.2]{ces15}, Tate's argument breaks down in this case. I thank Ki-Seng Tan for sending me the following proof of the continuity of $\langle-,-\rangle_{\le 1\le,\le A}$ which is valid for any local field $k$. Since $H^{1}(k,A)$ is discrete by Proposition \ref{fc}\,, Remark \ref{dct}(a) shows that we only need to check the following statement: given $\xi\in H^{1}(k,A)$, there exists an open neighborhood $U$ of $0_{A^{\lbe\vee}\be(k)}$ such that $\langle \xi,U\e\rangle_{\le 1\le,\le A}=\{0\}$. By \cite[Proposition 5, p.~417]{Sh64}, $H^{1}(k,A)$ agrees with the \v{C}ech cohomology group $\check{H}^{1}(k,A)$, whence there exists a finite and Galois extension $k^{\le\prime}\be/k$ such that $\res_{A,\, k^{\le\prime}\be/k}(\xi)=0$, where $\res_{A,\, k^{\le\prime}\be/k}\colon \check{H}^{1}(k,A)\to \check{H}^{1}(k^{\le\prime},A)$ is the restriction map in \v{C}ech cohomology (see, e.g., \cite[p.~422]{Sh64}). Since $\res_{A,\, k^{\le\prime}\be/k}$ is adjoint (relative to $\langle-,-\rangle_{\le 1\le,\le A}\e$) to the norm map $N_{A^{\lbe\vee}\be,\,k^{\le\prime}\be/k}\colon A^{\vee}\lbe(k^{\le\prime}\le)\to A^{\vee}\lbe(k)$, we have $\langle \xi,N_{A^{\lbe\vee}\be,\,k^{\le\prime}\be/k}(A^{\vee}\lbe(k^{\le\prime}))\e\rangle_{\le 1\le,\le A}=0$. Now the norm $k$-morphism $\rkk(\lbe A^{\vee}_{\le\kp}\be)\to A^{\vee}$ is smooth and surjective since its $(\le k^{\le\prime}\be/k\le)\le$-\le base change is the sum map $(A^{\vee}_{\le k^{\le\prime}})^{[\le k^{\le\prime}:\le k\le]}\to A^{\vee}_{\le k^{\le\prime}}$. Consequently, $N_{A^{\vee},\,k^{\le\prime}\be/k}\colon A^{\vee}\lbe(k^{\le\prime}\le)\to A^{\vee}\lbe(k)$ is open \cite[Proposition 4.3(a)]{ces} and  $U=N_{A^{\vee},\,k^{\le\prime}\be/k}(\lbe A^{\vee}\lbe(k^{\le\prime}\le))$ is the required open neighborhood of $0_{\lle A^{\vee}\lbe(k)}$.	 
\end{remark}

\section{Commutative algebraic groups}

Let $k$ be a field. We will write $\ck$ for the category of {\it commutative} algebraic $k$-groups and $\ak$ for its full subcategory of affine $k$-groups. By \cite[${\rm VI}_{\rm A}$, Theorem 5.4.2 and Corollary 5.4.3]{sga3}, $\ck$ and $\ak$ are abelian categories.

\begin{proposition}\label{ru1} Let $L$ be an affine commutative algebraic $k$-group. Then the following holds.
\begin{enumerate}
\item[(i)] $L\in\sks$ is a reflexive sheaf.
\item[(ii)] $E^{\le 1}\lbe(L^{\be D}\le)=0$.
\item[(iii)] $L^{\be D}$ is a reflexive object of $D^{\e[-1,0\le]}\lbe(k_{\fl})$.
\item[(iv)] If $E^{\le 1}\lbe(L\le)=0$, then $L$ is a reflexive object of $D^{\e[-1,0\le]}\lbe(k_{\fl})$.
\end{enumerate}	
\end{proposition}
\begin{proof} For (i), see \cite[Proposition 2.37]{ros23}. Assertion (ii) is \cite[Lemma 1.14]{rus13}. Since (i) implies the reflexivity of $L^{\be D}$, assertion (iii) follows from (ii) and Proposition \ref{res}\e. Similarly, (iv) follows from Proposition \ref{res}\e.
\end{proof}

\smallskip

\begin{remark}\label{imp0} If $\car k>0$ and $L\in\ak$, then $E^{\le 1}\lbe(L\le)=0$ by \cite[Proposition 2.27]{ros23}. Thus, by Proposition \ref{ru1}(iv), the object $L[\le 0\le]\in D^{\e[-1,0\le]}\lbe(k_{\fl})$ is reflexive for {\it every} affine and commutative algebraic $k$-group $L$. On the other hand, if $\car k=0$, then $E^{\le 1}\lbe(\gak)\neq 0$ by work of Gabber (see \cite[Remark 2.26]{ros23}, \cite[Example 2.3.17]{ri24} and \cite[\S6]{rr25}). It follows from the above and \cite[VIII, Proposition 3.3.1]{sga7} that $E^{\le 1}\lbe\lbe(L\le)=0$ if, and only if, $L$ contains no nontrivial unipotent subgroups.
\end{remark}
\smallskip

\begin{definition}\label{dex2} For every commutative algebraic $k$-group $G$, we will write $\s E_{\le G}$ for the set of all exact sequences in $\ck$ of the form
\begin{equation}\label{bex2}
\mathcal E(L,\iota, G,\pi, A)\colon 0\to L\overset{\!\iota}{\to} G\overset{\!\pi}{\to} A\to 0,
\end{equation}
where $L$ is affine and $A$ is an abelian $k$-variety. By \cite[Theorem 2.3(i)]{bri17b}, $\s E_{\le G}$ is a non\e-\le empty set. 
\end{definition}

\smallskip

If $\mathcal E_{\le r}=\mathcal E(L_{\le r},\iota_{\le r}, G,\pi_{\le r}, A_{\le r})$ are elements of $\s E_{\le G}$, where $r=1$ or $2$, we write $\mathcal E_{1} \leq \mathcal E_{2}$ if there exist morphisms $\lambda\colon L_{1}\to L_{2},\gamma\colon G\to G$ and $\alpha\colon A_{1}\to A_{2}$ in $\ck$ such that $(\lambda,\gamma,\alpha)\colon \mathcal E_{\le 1}\to \mathcal E_{\le 2}$ is a morphism of extensions in $\ck$, i.e., the diagram 
\[
\xymatrix{0\ar[r]&L_{1}\ar[d]^(.45){\lambda}\ar[r]^{\iota_{1}}&G\ar[d]^(.45){\gamma}\ar[r]^{\pi_{1}}&A_{1}\ar[d]^(.45){\alpha}
\ar[r]&0\\
0\ar[r]&L_{2}\ar[r]^{\iota_{2}}&G\ar[r]^{\pi_{2}}&A_{2}\ar[r]&0,
}
\]
commutes. It is not difficult to check that $\leq$ is a reflexive and transitive relation, whence $(\s E_{\le G},\leq\le)$ is a pre\e-\le ordered set. Further, $\s E_{\le G}$ is a {\it directed set}, as the following variant of \cite[proof of Theorem 2.3]{bri17b} shows. Let $\mathcal E_{r}=\mathcal E(L_{\le r},\iota_{r}, G,\pi_{r}, A_{\le r})\in\s E_{\le G}$, where $r=1$ or $2$, and consider the composition
\[
\varphi\colon L_{1}\!\times\! L_{2}\overset{\iota_{1}\be\times \iota_{2}}{\to}G\times G\overset{\!+}{\to} G.
\]
Set $L_{3}=\img\varphi$, $A_{3}=G/L_{3}$ and let $\iota_{3}\colon L_{3}\to G$ and $\pi_{3}\colon G\to A_{3}$ be, respectively, the inclusion and the canonical projection. For $r=1$ or $2$, let $\lambda_{\le r}\colon L_{\le r}\to L_{3}$ be the composition of canonical maps $L_{r}\into L_{1}\times L_{2}\onto L_{3}$. Clearly, $\iota_{r}=\iota_{3}\circ\lambda_{\le r}$, whence $\img \iota_{r}\subseteq L_{3}$. Now let $\alpha_{\le r}\colon A_{\le r}\onto A_{3}$ be defined by the commutativity of the diagram
\[
\xymatrix{A_{\le r}\ar[r]^{\alpha_{\le r}}&A_{3}\,,\\
G/\le\img \iota_{\le r}\ar[u]^{\overline{\pi}_{r}}_{\simeq}\ar@{->>}[ur]&
}
\]
where the oblique map is induced by the identity morphism of $G$. Then the following diagram in $\ck$ commutes
\[
\xymatrix{0\ar[r]&L_{\le r}\ar@{^{(}->}[d]^(.45){\lambda_{\le r}}\ar[r]^{\iota_{r}}&G\ar@{=}[d]\ar[r]^{\pi_{r}}&A_{\le r}\ar@{->>}[d]^(.45){\alpha_{\lle r}}\ar[r]&0\\
0\ar[r]&L_{3}\ar[r]^{\iota_{3}}&G\ar[r]^{\pi_{3}}&A_{3}\ar[r]&0.
}
\]
Thus $\mathcal E_{3}\overset{{\rm def.}}{=}\mathcal E(L_{3},\iota_{3}, G,\pi_{3}, A_{3})\in\s E_{\le G}$ and 
$\mathcal E_{\le r}\leq \mathcal E_{3}$ for $r=1$ and $2$, i.e., $\s E_{\le G}$ is a directed set, as claimed.

\smallskip

If $G\in\ck$ is {\it connected}, then $\s E_{\le G}$ contains a {\it distinguished object}. Indeed, by \cite[${\rm VII}_{\rm A}$, Proposition 8.3]{sga3}, if $\car\,k=p>0$, there exists a minimal integer $n\geq 0$ such that the image $G_{\sm}$ of the $n\e$-fold relative Frobenius morphism ${\rm Fr}^{\le n}\colon G\to G^{\e(\e p^{\lle n}\lle)}$ is smooth over $k$. If $\car\, k=0$, we set $n=0$ and $G_{\sm}=G$. Thus $G_{\sm}$ is smooth and connected, whence there exist an abelian $k$-variety $\Alb(G_{\sm}\lbe)$ and a homomorphism $\alb_{\e G_{\sm}}\colon G_{\sm}\onto \Alb(G_{\sm}\lbe)$ such that every homomorphism from $G_{\sm}$ to an abelian variety factors uniquely through $\alb_{\le G_{\sm}}$. Further, $\krn\be(\alb_{\le G_{\sm}}\be)$ is affine and connected (see \cite[Proposition 4.1.4 and Theorem 4.3.4]{bri17a}). Now let $q$ be the composition $G\onto G_{\sm}\overset{\lle\alb_{\lle G_{\sm}}}{\lra}\Alb(G_{\sm}\lbe)$. Then $q$ is surjective and $L(G\e)\overset{{\rm def.}}{=}\krn q$ is an extension $0\to \krn {\rm Fr}^{\le n}\to L(G\e)\to \krn\alb_{\e G_{\sm}}\to 0$. We call the sequence 
\begin{equation}\label{seq:frob-alb}
0\to L(G\e)\to G\overset{\be q}{\to} \Alb(G_{\sm}\lbe)\to 0
\end{equation}
the {\it Frobenius\e--Albanese decomposition of $G$}.

\smallskip

If $G\in\ck$ is {\it smooth and connected}, i.e., $n=0$ above, then \eqref{seq:frob-alb} is the {\it Albanese decomposition of $G$}:
\begin{equation}\label{seq:alb-dec}
0\lra\krn\be(\alb_{\le G})\lra G\,\xrightarrow{\alb_{\le G}}\Alb(G\e)\lra 0.
\end{equation}
The sequence \eqref{seq:alb-dec} is an {\it initial} object of $\s E_{\le G}$. Indeed, for every $\mathcal E(L,\iota, G,\pi, A)\in\s E_{\le G}$, there exists a unique homomorphism $\alpha\colon \Alb(G\e)\to A$ such that the diagram
\begin{equation}\label{diag:alb-diag}
\xymatrix{0\ar[r]&\krn\be(\alb_{\le G})\ar@{^{(}->}[d]\ar[r]&G\ar@{=}[d]\ar[rr]^(.45){\alb_{\le G}}&&\Alb(G\e)\ar@{->>}[d]^(.45){\alpha}\ar[r]&0\\
0\ar[r]&L\ar[r]^{\iota}&G\ar[rr]^(.45){\pi}&&A\ar[r]&0
}
\end{equation}
commutes. Consequently, the top row of the above diagram, i.e., the sequence \eqref{seq:alb-dec}, is an initial object of $\s E_{\le G}$, as claimed. We also note that, if the bottom row of \eqref{diag:alb-diag} is a {\it Chevalley decomposition of $G$}, i.e., $L$ is {\it smooth and connected}, then $L/\le\krn\be(\alb_{\le G})\simeq \krn\alpha$ is smooth, connected, affine and proper, hence trivial by Lemma \ref{triv}\,. Thus, if $G$ is smooth, connected and admits a Chevalley decomposition, then it admits a {\it unique} such decomposition that agrees with the Albanese decomposition of $G$ \eqref{seq:alb-dec}.

\begin{example} Let $k$ be an imperfect field and let $X$ be a projective $k$-variety with a finite non\e-\le normal locus $\spec A$ and normalization morphism  $X^{\be N}\to X$. If $X^{\be N}$ is {\it geometrically normal}, then ${\rm Pic}_{\lbe X^{\lbe N}\!/k}^{\le \sm,\e 0}={\rm Pic}_{\lbe X^{\lbe N}\!/k}^{\le 0,\le\red}$ is an abelian $k$-variety and \cite[Remark 2.6]{bri15} yields the following (Chevalley) decomposition of $\picom$:
\[
0\to R_{\e B/k}(\G_{m\lle,\le B}\lbe)/R_{\e A/k}(\G_{m\lle,\le A})\to\picom\to {\rm Pic}_{\lbe X^{\lbe N}\!\lbe/\lle k}^{\le 0,\le\red}\to  0,
\]
where $B=\mathcal O(\lbe X^{\lbe N}\!\times_{X}\be\spec A)$.
\end{example}

\smallskip

Now, given a sequence $\mathcal E=\mathcal E(L,\iota, G,\pi, A)\colon 0\to L\to G\to A\to 0$ in $\s E_{\le G}$, we will write
\begin{equation}\label{del}
\partial_{\e\mathcal E}=\partial_{\e\mathcal E(L,\le\iota,\le G,\le\pi,\le A)}\colon L^{\be D}\to A^{\lbe\vee}
\end{equation}
for the composition of canonical maps in $\sks$
\begin{equation}\label{fac0}
L^{\be D}\overset{\!\delta_{\lle\mathcal E}}{\lra} E^{\le 1}\lbe(A)\overset{\!\beta_{\lbe A}}{\lra} A^{\lbe\vee},
\end{equation}
where $\delta_{\lle\mathcal E}$ is the first connecting morphism in the long exact sequence associated to $\shomk(-,\G_{m,k})$ and $\beta_{\lbe A}$ is the Barsotti\e--\e Weil isomorphism \eqref{bwf}. Using \eqref{er}, the map $\delta_{\lle\mathcal E}$ can be described as follows: for every $k$-scheme $T$, 
\[
\delta_{\lle\mathcal E}(\lle T\le)\colon L^{\be D}\lbe(\le T\le)=\Hom_{\, T_{\fl}^{\le\sim}}\lbe(\le L_{\le T},\G_{m,\e T}\lbe)\to E^{\le 1}\lbe(A)(\lle T\lle)=\mathrm{Ext}_{\e T_{\be\fl}^{\le\sim}}^{1}\be(A_{\le T},\G_{m,\e T})
\]
maps a character $\chi\in L^{\be D}\lbe(\le T\le)$ to the equivalence class in $\mathrm{Ext}_{\e T_{\be\fl}^{\le\sim}}^{1}\be(A_{\le T},\G_{m,\e T})$ of the pushout extension
\begin{equation}\label{push0}
0\to \G_{m,\e T}\to G_{\le T,\le \chi}\to A_{\le T}\to 0,
\end{equation}
where
\begin{equation}\label{push}
G_{\le T,\le \chi}\overset{{\rm def.}}{=}\G_{m,\e T}\coprod_{\le\chi,\e L_{\lle T},\e \iota_{\lle T}}\be G_{\le T}.
\end{equation}
It follows from the preceding definition that $\delta_{\mathcal E}$, and therefore also $\partial_{\e\mathcal E}$, depends only on the equivalence class of $\mathcal E$ in $\ext^{1}\be(A,L\le)$. 

\smallskip

\begin{proposition}\label{bdg} Assume that $\car k>0$. Let $G$ be a commutative algebraic $k$-group and let $\mathcal E\colon 0\to L\to G\to A\to 0$ be an object of $\s E_{\le G}$. Further, let $G_{\sab}$ be the maximal semiabelian subvariety of $G$. 
Then there exists a morphism of extensions in $\ck$
\[
\xymatrix{\mathcal E^{\prime}\colon 0\ar[r]&G_{\sab}\times_{G}L\ar[r]\ar@{^{(}->}[d]^{\jmath}& G_{\sab}\ar[r]\ar@{^{(}->}[d]&A
\ar@{=}[d]\ar[r]&0\\
\mathcal E\colon 0\ar[r]&L\ar[r]&G\ar[r]&A\ar[r]&0
}
\]
that induces a factorization of the map $\partial_{\e\mathcal E}\colon L^{D}\to A^{\vee}$ \eqref{del}\rm{:}
\[
\partial_{\e\mathcal E}\colon L^{\be D}\overset{\jmath^{\le D}}{\longtwoheadrightarrow} (G_{\sab}\times_{G}L)^{D}\,\overset{\!\partial_{\lle\mathcal E^{\prime}}}{\lra}A^{\vee}.
\]
Further, the sheaf $(G_{\sab}\times_{G}L)^{D}$ is representable.
\end{proposition}
\begin{proof} By \cite[Lemmas 5.6.1 and 5.6.2]{bri17a}, $G$ is an extension in $\ck$
\begin{equation}\label{ach}
0\to G_{\sab}\to G\to Q\to 0,
\end{equation}
where 
\begin{equation}\label{qgr}
Q\overset{{\rm def.}}{=} G/G_{\sab}
\end{equation}
is affine\,\footnote{\e This fails if $\car k=0$. See \cite[Remark 5.6.5]{bri17a}.} and contains no nonzero $k$-subtorus\,\e\footnote{\e If $G$ is smooth and connected, then $Q$ \eqref{qgr}
is unipotent and some authors call the sequence \eqref{ach} the {\it anti-Chevalley decomposition of $G$}.}\e. Let $T=\krn\be(\alb_{\le G_{\sab}})$ be the maximal $k$-subtorus of $G_{\sab}$, which is also the maximal $k$-subtorus of $L$. We have a canonical commutative diagram in $\ck$
\[
\xymatrix{0\ar[r]&T\ar[r]\ar@{=}[d]& G_{\sab}\times_{G}L\ar[r]\ar@{^{(}->}[d]&F
\ar@{^{(}->}[d]\ar[r]&0\\
0\ar[r]&T\ar[r]&G_{\sab}\ar[r]^(.4){\alb_{\le G_{\sab}}}&\Alb\e(G_{\sab})\ar[r]&0,
}
\]
where $F\overset{{\rm def.}}{=}\alb_{\le G_{\sab}}(G_{\sab}\times_{G}L\le)$
is proper and affine, i.e., finite \cite[Tag 01WG]{sp}.
In particular, the top row of the above diagram expresses $G_{\sab}\times_{G}L$ as an extension of a finite $k$-group by a $k$-torus, whence the sheaf $(G_{\sab}\times_{G}L)^{D}$ is representable by, e.g., \cite[Proposition 2.34]{ros23}. Next, there exists a canonical exact and commutative diagram in $\ck$
\begin{align}\label{td1}
\xymatrix{0\ar[r]&G_{\sab}\times_{G}L\ar[r]\ar@{^{(}->}[d]& G_{\sab}\ar[r]\ar@{^{(}->}[d]&A^{\prime}
\ar@{^{(}->}[d]\ar[r]&0\\
0\ar[r]&L\ar@{->>}[d]\ar[r]&G\ar[r]^{\pi}\ar@{->>}[d]&A\ar@{->>}[d]\ar[r]&0\\
0\ar[r]&L/(G_{\sab}\!\times_{G}\be L)\ar[r]&Q\ar[r]&A/A^{\prime}\ar[r]&0,
}
\end{align}
where $A^{\prime}=\pi(G_{\sab})$. By \cite[Lemma 2.54]{bga}, \cite[II, Corollary 5.4.3(ii)]{ega} and \cite[${\rm VI}_{\rm A}$, 0.3]{sga3}, $A/\lbe A^{\prime}$ is proper, smooth and connected, i.e., an abelian variety. Since $A/A^{\prime}$ is also affine (as a quotient of $Q$), we conclude that $A=A^{\prime}$. Thus \eqref{td1} is a diagram
\[
\xymatrix{0\ar[r]&G_{\sab}\times_{G}L\ar[r]\ar@{^{(}->}[d]^{\jmath}& G_{\sab}\ar[r]\ar@{^{(}->}[d]&A
\ar@{=}[d]\ar[r]&0\\
0\ar[r]&L\ar@{->>}[d]\ar[r]&G\ar[r]\ar@{->>}[d]&A\ar[r]&0\e.\\
&Q\ar@{=}[r]&Q&&
}
\]
Next, since $E^{\le 1}\lbe(K\le)=0$ for every affine and commutative algebraic $k$-group $K$ by \cite[Proposition 2.27]{ros23}, the above diagram yields the following exact and commutative diagram in $\sks$
\[
\xymatrix{&	Q^{ D}\ar@{^{(}->}[d]\ar@{=}[r]&Q^{D}\ar@{^{(}->}[d]&&&\\
0\ar[r]&	G^{D}\ar@{->>}[d]\ar[r]&L^{\be D}\ar@{->>}[d]^(.4){\jmath^{\lle D}}\ar[r]^{\partial_{\le\mathcal E}}&A^{\vee}\ar@{=}[d]\ar[r]&
E^{\le 1}\lbe(\lle G\e)\ar@{^{(}->}[d]\ar[r]&0\\
0\ar[r]&G_{\sab}^{D}\ar[r]&(G_{\sab}\be\times_{G}\be L)^{D}\ar[r]^(.68){\partial_{\lle\mathcal E^{\prime}}}&A^{\vee}\ar[r]&E^{\le 1}\lbe(\lle G_{\sab}\le)\ar[r]&0,
}
\]
whence the proposition follows.  
\end{proof}

\smallskip

Now let $K$ be an affine commutative algebraic $k$-group, let $A$ be an abelian $k$-variety and let
\begin{equation}\label{vma}
v\colon K^{\lbe D}\to A
\end{equation}
be a morphism in $\sks$. Since $A^{D}=0$ and $E^{\le 1}\lbe(\lbe K^{\lbe D}\le)=0$ by Proposition \ref{ru1}(ii), we may apply Lemma \ref{pai} to \eqref{vma} and obtain that the sheaf
\begin{equation}\label{gv}
\hskip 1cm G(v\le)\overset{{\rm def.}}{=}\extc^{\e 1}\lbe(\e C\lbe( v\le),\G_{m,k})
\end{equation}
is an extension in $\sks$
\begin{equation}\label{gv2}
0\to K^{\lbe DD}\to G( v\le)\to E^{\le 1}\lbe(\lbe A\lle)\to 0.
\end{equation}

\begin{proposition}\label{koo} Let $K$ be an affine commutative algebraic $k$-group, let $A$ be an abelian $k$-variety and let $v\colon K^{\lbe D}\to A$ be a morphism in $\sks$. Then the sheaf $G( v\le)$ \eqref{gv} is an extension of commutative algebraic $k$-groups
\begin{equation} \label{ef}
\mathcal E( v\le)\colon 0\to K\to G( v\le)\to A^{\vee}\to 0.
\end{equation}
In particular, $G( v\le)$ is a commutative algebraic $k$-group.
\end{proposition}
\begin{proof} By Proposition \ref{ru1}(i), the map $e_{K}\colon K\to K^{\lbe DD}$ \eqref{evm} is an isomorphism in $\sks$. The sequence \eqref{ef} is defined by the commutativity of the following diagram in $\ck$:
\begin{equation}\label{kdia}
\xymatrix{
0\ar[r]& K^{\lbe DD}\ar[d]_(.47){e_{K}^{-1}}^(.45){\simeq}\ar[r]&G( v\le)\ar@{=}[d]\ar[r]&E^{\le 1}\lbe(\lbe A\lle)\ar[d]_(.5){\beta_{\lbe A}}^(.5){\simeq}\ar[r]&0\\
0\ar[r]& K\ar[r]&G( v\le)\ar[r]& A^{\vee}\ar[r]&0,
}
\end{equation}
where the top row is the sequence \eqref{gv2}. For the representability of $G( v\le)$, see, e.g., \cite[Remark 3.1.3]{ri24}.
\end{proof}

\smallskip

Now another application of Lemma \ref{pai} to \eqref{vma} yields a natural isomorphism in $D^{\e[-1,0\le]}\lbe(k_{\fl})$ similar to \eqref{iso0}
\begin{equation}\label{gv3}
\theta_{\lle v}\colon D(C\lbe( v\le))\isoto G( v\le),
\end{equation}
where $G(v\le)$ is now equipped with the extension structure $\mathcal E( v\le)$ \eqref{ef} (rather than \eqref{gv2}). 

\smallskip

\begin{remarks}\label{gv3.5}\indent
\begin{enumerate}
\item[(a)] By construction, the projection $G(v)\to A^{\le\vee}$ in \eqref{ef}
is the composition
\[
\extc^{\e 1}\lbe(\e C\lbe( v\le),\G_{m,k})\to \extk^{\e 1}\lbe(\e A,\G_{m,k})\overset{\!\beta_{A}}{\lra} A^{\le\vee},
\]
where the first map is induced by the canonical injection $A[0]\into C(v)$ and $\beta_{\lbe A}$ is the Barsotti\e--\e Weil isomorphism \eqref{bwf}.

\smallskip

\item[(b)] By Remark \ref{pai2}(b), if $K=0$ in \eqref{vma} (so that $C(v)=A$, $G( v\le)=A^{\vee}$ and $D(C\lbe( v\le))=D(A)=E^{\le 1}\lbe(\lbe A\lle)$ via the identification \eqref{obv2}), the isomorphism $\theta_{\lle 0}\colon D(A)\to A^{\lbe\vee}$ \eqref{gv3} will be identified with the Barsotti\e--\e Weil isomorphism $\beta_{A}\colon E^{\le 1}\lbe(\lbe A\lle)\isoto A^{\lbe\vee}$.
\end{enumerate}
\end{remarks}

\section{The category of $1$-motives}\label{cat1}

We keep the notation introduced in the previous section. In particular, $\ck$ denotes the category of commutative algebraic $k$-groups.

\begin{definition}\label{1mot} Let $k$ be a field. The category $\s{M}_{1}$ of $1$-motives over $k$ is defined as follows. An {\it object} of $\s{M}_{1}$, i.e., a $1$-motive over $k$, is a triple
\[
M=(K^{\lbe D}\be,u\e,\mathcal E\le)=(K^{\lbe D}\be,u\e,\mathcal E(L,\iota, G,\pi, A)\le),
\]
where $G$ is a commutative algebraic $k$-group, $\mathcal E=\mathcal E(L,\iota, G,\pi, A)$ \eqref{bex2} is an object of $\s E_{\le G}$, $K$ and $L$ are affine commutative algebraic $k$-groups such that
\begin{equation}\label{fcon}
\hskip -.6cm E^{\le 1}\lbe(K\le)=E^{\le 1}\lbe(L\le)=0
\end{equation}
and $u\colon K^{\lbe  D}\to G$ is a morphism in $\sks$. 

\smallskip

If $\e M_{\le i}=(K_{\lle i}^{\lbe  D}, u_{\le i},\mathcal E_{\le i}\le)=(K_{\lle i}^{\lbe D}, u_{\le i},\mathcal E(L_{\le i},\iota_{\le i}, G_{\lle i},\pi_{\lle i}, A_{\le i}))$ are $1$-motives, where $i=1$ or $2$, a {\it morphism} $\e M_{1}\to\e M_{2}$ in $\s{M}_{1}$ is a pair
\begin{equation}\label{mor1m}
(\le g^{\lle D}\!, (\lambda,\gamma,\alpha)\e)\colon (K_{1}^{\lbe D}, u_{\le 1},\mathcal E_{\le 1}\le)\to (K_{2}^{\lbe D}, u_{\le 2},\mathcal E_{\le 2}\le),
\end{equation}
where $g\colon K_{\le 2}\to K_{\le 1}$ is a morphism in $\ck$ and $(\lambda,\gamma,\alpha)\colon \mathcal E_{\le 1}\to \mathcal E_{\le 2}$ is a morphism of extensions in $\ck$ such that the following diagram in $\sks$ commutes:
\begin{equation}\label{sd}
\xymatrix{K_{\le 1}^{\lbe D}\ar[d]_(.45){g^{D}}\ar[r]^{u_{\le 1}}&G_{\le 1}\ar[d]^(.45){\gamma}\\
K_{\le 2}^{\lbe  D}\ar[r]^{u_{\le 2}}&G_{\le 2}.
}
\end{equation}
\end{definition}

\begin{remark}\label{imp1} As noted in Remark \ref{imp0}\,, condition \eqref{fcon} is vacuous if $\car\e k>0$. On the other hand, if $\car k=0$, then \eqref{fcon} excludes from Definition \ref{1mot} the ``$1$-motives with additive components" discussed, for example, in \cite{lau}. 
\end{remark}

Given a $1$-motive $M=(K^{\lbe D}\be,u\e,\mathcal E(L,\iota, G,\pi, A)\le)$, the {\it complex $[\lle M\e]$ associated to $M$} is the mapping cone of $u\e$:
\begin{equation}\label{comp}
\hskip 1cm [\lle M\e]=C\lbe(u)=[\le K^{\lbe D}\!\overset{\!u}{\to} G\e\le]\e\in C^{\e[-1,0\le]}\lbe(k_{\fl}).
\end{equation}

\smallskip

We will also write $[\lle M\e]$ for the corresponding object of $D^{\e[-1,\lle 0\le]}\lbe(k_{\fl})$.

\smallskip

Clearly, if $M=(0,0,\mathcal E(L,\iota, G,\pi, A)\le)$, then $[\lle M\e]=G\lle[\lle 0\lle]\in D^{\e[-1,0\le]}\lbe(k_{\fl})$.

\smallskip

There exists a distinguished triangle in $D(k_{\fl})$
\begin{equation}\label{csc}
K^{\lbe D}\to G\to[\lle M\e]\to K^{\lbe D}[1].
\end{equation}
Further, by the commutativity of diagram \eqref{sd}, the assignment
\begin{equation}\label{fn1}
\s{M}_{1}\to D^{\e[-1,0\le]}\lbe(k_{\fl}), M\mapsto [\lle M\e],
\end{equation}
is a functor.

\begin{theorem}\label{d1m} Let $M$ be a $1$-motive. Then $[\lle M\e]$ is a reflexive object of $D^{\e[-1,\lle 0\le]}\lbe(k_{\fl})$, i.e., the evaluation morphism \eqref{dduf}
\[
e_{\le[\lle M\e]}\colon [\lle M\e]\to DD(M\le)
\]
is an isomorphism in $D^{\e[-1,\lle 0\le]}\lbe(k_{\fl})$.
\end{theorem}
\begin{proof} First assume that $M=(0,0,\mathcal E(L,\iota, G,\pi, A)\le)$, so that $[M\le]=G$, and consider the triangle $L\to G\to A\to L[1]$. We have $H^{\le 0}(D(L\le))=E^{\le 1}\lbe(L\le)=0$ by \eqref{e2} and \eqref{fcon}. Thus
$D(L\le)[-1]=L^{\be D}$ by \eqref{obv1} and therefore $H^{\e 0}(D(D(L\le)[-1])))=H^{\e 0}(D(L^{\be D}))=E^{\le 1}\lbe(L^{\be D}\le)=0$ by \eqref{e2} and Proposition \ref{ru1}(ii). Further, $L$ and $A$ are reflexive objects of $D^{\e[-1,\lle 0\le]}\lbe(k_{\fl})$ by Proposition \ref{ru1}(iv) and Corollary \ref{abr}\e, respectively. Now an application of Lemma \ref{compat} to the triangle $L\to G\to A\to L[1]$ yields the reflexivity of $G[\lle 0\lle]\in D^{\e[-1,\lle 0\le]}\lbe(k_{\fl})$.

Now let $M$ be arbitrary and consider the triangle $K^{\lbe D}\to G\to[\lle M\e]\to K^{\lbe D}[1]$ \eqref{csc}. We have
$H^{\le 0}(D(K^{\lbe D}\le))=E^{\le 1}\lbe(K^{\lbe D})=0$ by \eqref{e2} and Proposition \ref{ru1}(ii). Consequently, $D(K^{\lbe D}\le)[-1]=K^{\lbe DD}\simeq K$ by \eqref{obv1} and Proposition \ref{ru1}(i), whence $H^{\e 0}(D(D(K^{\lbe D}\le)[-1])))\simeq H^{\e 0}(D(K\le))=E^{\le 1}\lbe(K\le)=0$ by \eqref{fcon}. Since $K^{\lbe D}$ and $G$ are reflexive objects of $D^{\e[-1,\lle 0\le]}\lbe(k_{\fl})$ by Proposition \ref{ru1}(iii) and the first part of the proof, a second application of Lemma \ref{compat} yields the reflexivity of $[\lle M\e]$.
\end{proof}

Now, set
\begin{equation}\label{mdu2}
\hskip 1.8cm D(M\le)\overset{{\rm def.}}{=}D([M\le])\in D^{\e[-1,\lle 0\le]}\lbe(k_{\fl}).
\end{equation}
We will show that $D(M\le)$ is represented by an explicit object of $C^{\e[-1,0\le]}\lbe(k_{\fl})$, namely the complex $[\le M^{\lle\diamond}\le]$ associated to the $1$-motive $M^{\le\diamond}$ defined below.

\smallskip

Proposition \ref{koo} applied to the morphism 
\begin{equation}\label{umo}
v\overset{{\rm def.}}{=}\pi\le\circ\le u\colon K^{\lbe D}\to A
\end{equation}
yields a commutative algebraic $k$-group \eqref{gv}
\begin{equation}\label{gt2}
\hskip 1cm G^{\e\vee}\overset{{\rm def.}}{=} G(\lle v\le)=\extc^{\e 1}\lbe(\e C\lbe(\le v\le),\G_{m,\e k})
\end{equation}
which is equipped with the extension structure \eqref{ef}
\begin{equation}\label{fex2}
\hskip .4cm \mathcal E^{\le\vee}\overset{{\rm def.}}{=}\mathcal E(\le v\le)\colon \, 0\lra K\overset{\!\iota^{\lbe\vee}}{\lra} G^{\e\vee}\overset{\be\pi^{\lbe\vee}}{\lra} A^{\vee}\lra 0.
\end{equation}
Further, by \eqref{gv3}, there exists a natural isomorphism in $D^{\e[-1,0\le]}\lbe(k_{\fl})$
\begin{equation}\label{gv4}
\theta_{\lle v}\colon D(C\lbe( v\le))\isoto G^{\e\vee}.
\end{equation}
Now the canonical exact sequence in $C^{\e[-1,0\le]}\lbe(k_{\fl})$
\begin{equation}\label{secx}
0\lra L\xrightarrow{(0,\le\iota)} C(u)
\xrightarrow{({\rm Id},\le\pi)}C(v)\lra 0  
\end{equation}
induces a morphism in $\sks$
\begin{equation}\label{vmor}
u^{\lbe\vee}\colon L^{\be D}\to G^{\e\vee}
\end{equation}
that is defined as follows.

For every $k$-scheme $T$ and every $\chi\in L^{\be D}(\le T\le)$, 
$u^{\lbe\vee}\lbe(\le T\le)(\chi)$ is the equivalence class in $G^{\e\vee}\lbe(\le T\le)\simeq\exct^{1}\lbe(C(v)_{\le T},\G_{\le m,\e T})$ \eqref{iso2} of the pushout extension
\begin{equation}\label{pushc0}
0\to \G_{\le m,\e T}\to C(u)_{\le T,\le \chi}\to C(v)_{\le T}\to 0,
\end{equation}
where
\[
C(u)_{\le T,\le \chi}\overset{{\rm def.}}{=}\G_{m,\e T}\coprod_{\le\chi\le,\e L_{\lle T},\e (0,\le\iota_{\lle T}\lbe)}\be C(u)_{\le T}.
\]

\smallskip

\begin{definition} \label{ddef} Let
$\e M=(K^{\lbe D}\be,u\e,\mathcal E\e)=(K^{\be\lle D}\be,u\e,\mathcal E(L,\iota, G,\pi, A)\le)$ be a $1$-motive over $k$. The {\it dual} $M^{\le\diamond}$ of $M$ is the $1$-motive
\begin{equation}\label{mdu}
M^{\le\diamond}=\big(L^{\be D}\be,u^{\lbe\lbe\vee}\e,\mathcal E^{\le\vee}\,\big)=(L^{\be D}, u^{\lbe\lbe\vee}, \mathcal E(K,\iota^{\lbe\vee}, G^{\e\vee},\pi^{\lbe\vee}, A^{\vee}\le)),
\end{equation}
where $G^{\e\vee}$ is the commutative algebraic $k$-group 
\eqref{gt2}, $\mathcal E^{\le\vee}=\mathcal E(K,\iota^{\lbe\vee}, G^{\e\vee},\pi^{\lbe\vee}, A^{\vee}\le)$ is the object \eqref{fex2} of $\s E_{\le G^{\le\vee}}$ and $u^{\lbe\vee}\colon L^{\be D}\to G^{\e\vee}$ is the morphism \eqref{vmor}.
\end{definition}

\smallskip

Note that, since the conditions \eqref{fcon} are invariant under the interchange $M\leftrightarrow M^{\le\diamond}$, the object $M^{\le\diamond}$ \eqref{mdu} is, in effect, a $1$-motive over $k$. Further,
\begin{equation}\label{dcomp}
[\lle M^{\le\diamond}\e]=C(u^{\be\vee}\le)=[\le L^{\be D}\overset{\be u^{\lbe\vee}}{\to} G^{\e\vee}\e\le]\e\in D^{\e[-1,0\le]}\lbe(k_{\fl}).
\end{equation}

\smallskip

Now recall the morphism $\partial_{\le\mathcal E}\colon L^{\be D}\to A^{\vee}$ \eqref{del} associated with $\mathcal E$.

\begin{lemma}\label{pud} We have $\pi^{\vee}\circ u^{\lbe\vee}=\partial_{\e\mathcal E}$.
\end{lemma}
\begin{proof} Let $T$ be a $k$-scheme and let $\chi\in L^{\be D}(\le T\le)$. Recall that  $u^{\lbe\vee}\lbe\lbe(\le T\le)(\chi)$ is the class in $\exct^{1}\lbe(C(v)_{\le T},\G_{\le m,\e T})$ of the pushout extension $C(u)_{\le T,\le \chi}$ \eqref{pushc0}. Now, by the description of $\pi^{\lbe\vee}\colon G^{\e\vee}\to A^{\e\vee}$ given in Remark \ref{gv3.5}(a), $\pi^{\lbe\vee}(\le T\le)(u^{\lbe\vee}\lbe(\le T\le)(\chi))$ is the image under $\beta_{\lbe A}(\le T\le)$ of the equivalence class in $\mathrm{Ext}_{\e T_{\be\fl}^{\le\sim}}^{1}\be(A_{\le T},\G_{m,\e T})$ of the pushout extension $G_{\le T,\le \chi}$ \eqref{push}\e, which is exactly $\partial_{\mathcal E}(\le T\le )(\chi)$ by the definition of $\partial_{\le\mathcal E}$ \eqref{del}.
\end{proof}

\smallskip

\begin{examples}\label{k0e}\indent
\begin{enumerate}
\item[(a)] If $M=(0,0,\mathcal E(L,\iota, G,\pi, A)\le)=(0,0,\mathcal E\le)$, then $\mathcal E^{\le\vee}=\mathcal E(0,0, A^{\vee},{\rm Id}_{\le A^{\vee}}, A^{\vee}\le)$, so that $G^{\e\vee}=A^{\vee}$, and $u^{\lbe\vee}=\partial_{\le\mathcal E}\colon L^{\be D}\to A^{\vee}$ by Lemma \ref{pud}\e. Thus 
\begin{equation}\label{mvee}
M^{\le\diamond}=\big(L^{\be D}\be,u^{\lbe\lbe\vee}\e,\mathcal E^{\le\vee}\,\big)=(L^{\be D}\be,\partial_{\le\mathcal E}\e, \mathcal E(0,0, A^{\vee},{\rm Id}_{\le A^{\vee}}, A^{\vee}\le)),
\end{equation}
whence
\[
\hskip 1.15cm [\lle M^{\le\diamond}\e]=C(\partial_{\le\mathcal E})=[\le L^{\be D}\overset{\! \partial_{\le\mathcal E}}{\lra} A^{\lbe\vee}\e\le].
\]
\item[(b)] If $M=(K^{\lbe D}\be,v\e, \mathcal E(0,0, A,{\rm Id}_{\le A}, A\le))$, then $M^{\le\diamond}=(0\e,0\e, \mathcal E^{\le\vee})$, where
\[
\mathcal E^{\le\vee}=\mathcal E(v)\colon \, 0\to K\to G^{\e\vee}\to A^{\vee}\to 0
\]
is the sequence \eqref{fex2}, with $G^{\e\vee}=G(v)$ \eqref{gt2}. In this case $[\lle M^{\le\diamond}\e]=G^{\e\vee}[0]$.
\end{enumerate}
\end{examples}

\medskip

Next, the exact sequence of complexes \eqref{secx} induces a distinguished triangle in $D\lbe(k_{\fl})$
\begin{equation}\label{lmmp}
L\xrightarrow{(0\lle ,\le\iota)}[\lle M\e]
\xrightarrow{({\rm Id},\le\pi)}C(v)\xrightarrow{h} L[1], 
\end{equation}
where the connecting morphism $h$ is defined as follows. 

\smallskip

By \eqref{comp} and \cite[Definition 11.1.3]{ks06}, we have 
\begin{equation}\label{coix}
C(0,\le\iota)=\bigl[L\e\oplus\e K^{D}\xrightarrow{\be \iota+u}G\e\bigr].
\end{equation}
Now the exactness of \eqref{secx} shows that the map
\begin{equation}\label{qma}
q=(\le{\rm pr}_{2}, \pi)\colon C(0,\le\iota)\to C(v)
\end{equation}
is a quasi\le-\le isomorphism. Then
\begin{equation}\label{hm}
h=r\circ q^{-1}
\end{equation}
in $D(k_{\fl})$, where
\begin{equation}\label{rma}
r=(\le{\rm pr}_{1}, 0\le)\colon C(0,\le\iota)\to L[1].
\end{equation}

\smallskip

Next, since $H^{\le 0}(D(L\le))=E^{\le 1}\lbe(L\le)=0$ by \eqref{e2} and \eqref{fcon}, Lemma \ref{ref1} applied to the triangle \eqref{lmmp} together with the identification $D(L\le)=L^{\be D}[1]$ \eqref{obv1} yields a distinguished triangle in $D\lbe(k_{\fl})$
\begin{equation}\label{dit}
L^{\be D}\xrightarrow{D\lle(h\lbe)} D\lle(C(v))
\xrightarrow{D\lle({\rm Id},\le\pi)}D\lle(M\le)
\xrightarrow{D\lle(0,\le\iota)}L^{\be D}[1].
\end{equation}

\smallskip

\begin{lemma} \label{col} We have
\[
\theta_{v}\circ D(h)=u^{\lbe\vee}
\]
in $D\lbe(k_{\fl})$, where $\theta_{v}$ and $u^{\lbe\vee}$ are the maps \eqref{gv4} and \eqref{vmor}, respectively.
\end{lemma}
\begin{proof} Since the inclusion $k_{\fl}^{\sim}\into D(k_{\fl})$ is fully faithful, it suffices to check
that $H^{\le 0}(\theta_v)\circ H^{\le 0}(D(h))=u^{\vee}$.
Let $T$ be a $k$-scheme and let $\chi\in L^{\be D}(\le T\le)$. First we compute $H^{\le 0}(D(h))(\le T\le)(\chi)\in H^{\le 0}(D(C(v)))(\le T\le)=\Hom_{\e D^{\e[-1,\lle 0\le]}\lbe(T_{\fl})}(C(v)_{T},\G_{m,\le T}[1])$.

\smallskip

Since $D(h)=\tau_{\leq 0}\e R(h)$ (see the proof of Lemma \ref{ref1}), we have
\begin{equation}\label{dhx}
H^{\le 0}(D(h))(\le T\le)(\chi)=\chi[1]\circ h_{\le T}\overset{\eqref{hm}}{=}\chi[1]\circ r_{\le T}\circ q_{\e T}^{-1}\colon C(v)_{T}\longrightarrow\G_{m,\le T}[1].
\end{equation}
Now recall the pushout $G_{\le T,\le \chi}=\G_{m,\e T}\coprod_{\le\chi,\e L_{\lle T},\e \iota_{\lle T}}\be G_{\le T}$ \eqref{push}, let $i_{\le\G_{m,\e T}}\colon  \G_{m,\le  T}\to G_{\le T,\le \chi}$ and $i_{\le G_{\lle T}}\colon  G_{\le T}\to G_{\le T,\le \chi}$ be the canonical morphisms and consider the 
composition
\[
u_{\le T,\le \chi}\colon K^{\lbe D}_{\le T}\overset{u_{\le T}}{\lra} G_{\le T}\,\overset{\!\!\be i_{\le G_{\lle T}}}{\lra} G_{\le T,\le \chi}.
\]
Then 
\[
C(0,i_{\le\G_{m,\le T}})=\bigl[\G_{m,\le T}\oplus K_{\lle T}^{\lbe D}
\xrightarrow{\be i_{\e\G_{m,\lle T}}+\le u_{\le T\lbe,\chi}} G_{\le T,\le \chi}\e\bigr],
\]
\medskip
$C(0,\iota_{\le T})=\bigl[L_{\le T}\e\oplus\e K_{\le T}^{D}\xrightarrow{\be \iota_{\le T}+u_{\le T}}G_{\le T}\e\bigr]$ by \eqref{coix} and we have natural maps
\[
\begin{array}{rcl}
s_{\e T,\le \chi}&=&(\e\chi\oplus\id_{K_{T}^{\lbe D}},i_{\le G_{\lle T}})\colon C(0,\iota_{\le T})\to C(0,i_{\le\G_{m,\le T}})\\
r_{\e T,\le\chi}&=&(\le{\rm pr}_{1}, 0\le)\colon C(0,i_{\le\G_{m,\le T}})\to\G_{m,T}[1]\\
q_{\e T,\le\chi}&=&({\rm pr}_{2}, \pi_{\le T,\le\chi})\colon C(0,i_{\le\G_{m,\le T}})\to C(v)_{T},
\end{array}
\]
where $\pi_{\le T,\le\chi}\colon G_{\le T,\le \chi}\to A_{\le T}$ is the projection in \eqref{push0}. Now
\begin{equation}\label{lots}
q_{\e T}=q_{\e T,\le\chi}\circ s_{\e T,\le \chi}\quad\textrm{and}\quad \chi[1]\circ r_{\le T}=r_{\e T,\le\chi}\circ s_{\e T,\le \chi},    
\end{equation}
whence \eqref{dhx} and \eqref{lots} yield the identity
\[
H^{\le 0}(D(h))(\le T\le)(\chi)=r_{\le T,\le\chi}\circ q_{\le T,\le\chi}^{-1}.
\]
The latter composition is the connecting morphism associated with the extension \eqref{pushc0}. Now, by the description of the map $H^{\le 0}(\theta_{v})(\le T\le)$ \eqref{iso1} in Remark \ref{pai2}(a), $H^{\le 0}(\theta_{v})(\le T\le)(H^{\le 0}(D(h))(\le T\le)(\chi))$ is the class of the extension \eqref{pushc0} in $G^{\le\vee}(\le T\le)\simeq \exct^{1}(C(v)_{T},\G_{m,\le T})$. As noted in the proof of Lemma \ref{pud}\,, the latter is exactly $u^{\lbe\vee}\lbe\lbe(\le T\le)(\chi)$.
\end{proof}

\begin{proposition}\label{dvd} Let $M=(K^{\lbe D}\be,u\e,\mathcal E(L,\iota, G,\pi, A)\le)$ be a $1$-motive. Then there exists an isomorphism in $D^{\e[-1,\lle 0\le]}\lbe(k_{\fl})$
\begin{equation}\label{teta}
\theta_{M}\colon D(M\le)\isoto [\lle M^{\le\diamond}\e]
\end{equation}
that fits into an isomorphism of distinguished triangles in $D\lbe(k_{\fl})$
\begin{equation}\label{cdl}
\xymatrix{L^{\be D}\ar@{=}[d]\ar[r]^(.4){D(h)}& D(C(v))\ar[d]_(.45){\simeq}^{\theta_{\lle v}}\ar[r]& D(M\le)\ar[d]_(.45){\simeq}^{\theta_{M}}\ar[r]& L^{\be D}[1]\ar@{=}[d]\\
L^{\be D}\ar[r]^{u^{\lbe\vee}}& G^{\e\vee}\ar[r]& [\lle M^{\le\diamond}\e]\ar[r]&L^{\be D}[1],
}
\end{equation}
where the top row is the triangle \eqref{dit} and $\theta_{\lle v}$ is the natural isomorphism \eqref{gv4}. If $( L\otimes K\le)^{D}(k)=0$, then $\theta_{M}$ is uniquely determined by $M$.
\end{proposition}
\begin{proof} First assume that $M=(K^{\lbe D}\be,v\e, \mathcal E(0,0, A,{\rm Id}_{\le A}, A\le))$, so that $L=0$ and $[\lle M\e]=C(v)$. Then $\theta_{M}=\theta_{\lle v}$ fullfils the conditions of the proposition. For arbitrary $M$, the existence of $\theta_{M}$ follows from \cite[Proposition 10.1.15 and Theorem 11.2.6]{ks06} since the first square in \eqref{cdl} commutes (in $D(k_{\fl})$) by Lemma \ref{col}\e. Now, by \cite[Proposition 10.1.17]{ks06}, the map $\theta_{M}$ is unique if
\[
\Hom_{\le D\lle(k_{\fl})}(D(C(v)), L^{\be D}\e)\simeq\shomk(G^{\e\vee},L^{\be D}\e)(k)=0.
\]
Since $\shomk(A^{\vee}, L^{\be D}\le)\simeq \shomk(L\otimes A^{\vee}, \G_{m,\le k})\simeq \shomk(L, (A^{\vee})^{D}\le)=0$, 
\eqref{fex2} yields an injection $\shomk(G^{\e\vee},L^{\be D}\le)\into \shomk(K,L^{\be D}\le)$. Finally
\[
\shomk(K,L^{\be D}\le)\simeq \shomk(L\otimes K, \G_{m,\le k})=( L\otimes K\le)^{D},
\]
whence the proposition follows.
\end{proof}

\begin{corollary}\label{gky} Let $G$ be a commutative algebraic $k$-group and let
\[
\mathcal E\colon 0\to L\xrightarrow{\iota} G\xrightarrow{\pi} A\to 0
\]
be an object of $\s E_{\le G}$ such that $L$ contains no nontrivial unipotent subgroups if $\car\, k=0$. Then there exists a unique isomorphism in $D^{\e[-1,\lle 0\le]}\lbe(k_{\fl})$
\begin{equation}\label{tee}
\theta_{\lle\mathcal E}\colon D(G\le)\isoto C(\partial_{\le\mathcal E})
\end{equation}
such that the following diagram is an isomorphism of distinguished triangles in $D(k_{\fl})$
\begin{equation}\label{cdl2}
\xymatrix{L^{\be D}\ar@{=}[d]\ar[r]^(.4){D(h)}& D(A)\ar[d]_(.45){\simeq}^{\beta_{A}}\ar[r]^{D(\pi)}& D(G\le)\ar[d]_(.45){\simeq}^{\theta_{\le\mathcal E}}\ar[r]^{D(\iota)}& L^{\be D}[1]\ar@{=}[d]\\
L^{\be D}\ar[r]^{\partial_{\le\mathcal E}}& A^{\lbe\vee}\ar[r]& C(\partial_{\le\mathcal E})\ar[r]&L^{\be D}[1],
}
\end{equation}
where $h\colon A\to L[1]$ is the connecting morphism in $D(k_{\fl})$ induced by $\mathcal E$.
\end{corollary}
\begin{proof} Apply the proposition to the $1$-motive 
$M=(0,0,\mathcal E\le)$ and recall the identification $\theta_{\le 0}=\beta_{A}$ made in Remark \ref{gv3.5}(b).
\end{proof}

\smallskip

\begin{remark}\label{demot} Let $S$ be a scheme and let $\s{M}_{1}^{\e\rm Del}$ be the category of Deligne $1$-motives over $S$, i.e., $1$-motives $M=(K^{\lbe D}\be,u\e,\mathcal E(L,\iota, G,\pi, A)\le)$ as above where $L$ and $K$ are $S\le$-\le tori. Then the functor $\s{M}_{1}^{\rm Del}\to D^{\e[-1,0\le]}\lbe(S_{\fl}), M\mapsto [\lle M\e],$ is fully faithfull \cite[Proposition 2.3.1]{ray94}. Further, the assignment $\s{M}_{1}^{\e\rm Del}\to \s{M}_{1}^{\e\rm Del}, M\mapsto M^{\le\diamond}$, is a contravariant functor \cite[10.2.12-14]{del74} and, for each Deligne $1$-motive $M$, there exists a unique morphism of Deligne $1$-motives 
\begin{equation}\label{mul}
e_{M}\colon M\to M^{\le\diamond\le\diamond}
\end{equation}
such that $\sigma_{\lbe M^{\lle\diamond}\!,\le M}^{\le *}\mathcal P_{\lbe M}=
(\id_{M^{\lle\diamond}}\times e_{M}\lle)^{\lle *} \mathcal P_{\lbe M^{\lle\diamond}}$, where $\mathcal P_{\lbe M}$ is the Poincar\'e biextension class of $M$ (cf. \eqref{bdi2}). Presumably, \eqref{mul} is the same morphism that is claimed to be an isomorphism in \cite[Theorem 1.11(3)]{jos13}. A (possibly different) biduality morphism $M\to M^{\le\diamond\le\diamond}$ is constructed, and claimed to be an isomorphism, in \cite[proof of Proposition 4.5]{haas19}. The argument there invokes classical
duality on the pure pieces and then uses naturality. For the
abelian\le-\le scheme piece, this requires verifying that the internal\le-\le Hom evaluation morphism on $A[0]$, under the Barsotti--Weil identifications, agrees with the classical biduality morphism $\kappa_{A}$. In other words, it is necessary to verify the commutativity of the following diagram in $\sfs\le$:
\begin{equation}\label{key2a}
\xymatrix{A\ar[d]_(.45){\kappa_{\lbe A}}^{\simeq}\ar[rr]^(.4){e_{\le(\lbe A\lle)}}&& E^{\le 1}\lbe(E^{\le 1}\lbe(A))\\
A^{\vee\vee}\ar[rr]^(.43){\beta_{\!\lbe A^{\lbe\lbe\vee}}^{\le -1}}_(.41){\simeq}&& E^{\le 1}\lbe(A^{\vee})\e,\ar[u]_{E^{\le 1}\lbe(\le\beta_{\lbe A})}^{\simeq}
}
\end{equation}
where $e_{\le (\lbe A\le)}\colon A\to E^{\le 1}\lbe(\lbe E^{\le 1}\lbe(A\le))$ is a certain map in $\sfs$ that can be identified with $e_{\lbe A[0\lle]}$ under the identification $DD(A\le)=E^{\le 1}\lbe(E^{\le 1}\lbe(A))$ induced by \eqref{obv2}. The proof of \cite[Proposition 4.5]{haas19} does {\it not} provide this verification, independently of whether the biduality morphism defined there agrees with \eqref{mul} or not. Now, diagram \eqref{key2a} is, essentially, $H^{\le 0}$ of the case $M=(0,0,A\le)$ of the following diagram in $D^{\e[-1,\lle 0\le]}\lbe(S_{\fl})$:
\begin{equation}\label{key2b}
\xymatrix{[\lle M\e]\ar[d]_{[\le e_{M}\le]}^{\eqref{mul}}\ar[rr]^(.45){e_{\le[\lle M\le]}}_(.45){\simeq}&& DD(M\le)\\
[\lle M^{\le\diamond\le \diamond}\e]\ar[rr]^(.45){\theta_{\be M^{\lle\diamond}}^{\le -1}}_(.45){\simeq}&&D(M^{\le\diamond})\ar[u]_{D(\theta_{M}\lbe)}^{\simeq},
}
\end{equation}
where the top horizontal arrow is an isomorphism by Theorem \ref{d1m} and the isomorphisms $\theta_{M^{\lle\diamond}}$ and $\theta_{M}$ are instances of \eqref{teta}. Clearly, the commutativity of diagram \eqref{key2b} implies that $e_{M}$ \eqref{mul} is an isomorphism, i.e., it yields the (often claimed, but apparently still unverified) {\it reflexivity of Deligne $1$-motives}. It can be shown that the commutativity of diagram \eqref{key2a} follows from (and, in fact, is equivalent to) Proposition \ref{xxx}\e. This commutativity is necessary (but not sufficient) for proving the commutativity of diagram \eqref{key2b}. The details of this proof will appear elsewhere, because the bijectivity of $e_{M}$ \eqref{mul} is not relevant to this paper. 
\end{remark}

\medskip

Given a $1$-motive $M=(K^{\lbe D}\be,u\e,\mathcal E(L,\iota, G,\pi, A)\le)$ and an integer $r$, we will write
\begin{equation}\label{mm}
H^{\le r}\be(k,M\e)\overset{{\rm def.}}{=} H^{\le r}\be(k,[\lle M\e]\e)
\end{equation}
for the $r$-th (hyper)cohomology group of the object $[\lle M\le]\in D^{\e[-1,\lle 0\le]}\lbe(k_{\fl})$. For every integer $r$, the natural map in $D\lbe(k_{\fl})$
\begin{equation}\label{dev}
E_{ M}\colon [\lle M\lle]\otimes^{\mathbf{L}}\! D(M\le)\to \G_{m,\e k}[1]
\end{equation}
induces a canonical cup product pairing
\begin{equation}\label{bpair}
\langle-,-\rangle_{ r,\e M}\e\colon  H^{\le r}\lbe(k,M\e)\times H^{\le 1-r}\lbe(k,D(M\e))\to\br\e k
\end{equation}
that is functorial in $[\lle M\le]$. See \cite[Tags 068G and 01FP]{sp}.

Now, for any fixed choice of the isomorphism $\theta_{M}\colon D(M\le)\isoto [\lle M^{\le\diamond}\e]$ \eqref{teta}\e, we obtain a corresponding isomorphism of abelian groups
\begin{equation}\label{hiso}
H^{\le 1-r}\lbe(k,\theta_{M})\colon H^{\le 1-r}\lbe(k,D(M\e))\isoto H^{\le 1-r}\lbe(k,M^{\le\diamond})
\end{equation}
and a pairing
\begin{equation}\label{bpair2}
\langle-,-\rangle_{ r,\e\le \theta_{ M}}\e\colon  H^{\le r}\lbe(k,M\e)\times H^{\le 1-r}\lbe(k,M^{\le\diamond})\to\br\e k
\end{equation}
given by
\[
\langle x,y\rangle_{ r,\e\le \theta_{M}}=\langle x,H^{\le 1-r}\lbe(k,\theta_{M})^{-1}(y)\rangle_{ r,\e M}
\]
for all $x\in H^{\le r}\lbe(k,M\e)$ and $y\in H^{\le 1-r}\lbe(k,M^{\le\diamond})$.

\smallskip

If $M=(0,0,\mathcal E(L,\iota, G,\pi, A)\le)=(0,0,\mathcal E\le)$ as in Example \ref{k0e}(a), the isomorphism $\theta_{M}=\theta_{\lle\mathcal E}\colon D(G\le)\isoto C(\partial_{\le\mathcal E})$ \eqref{tee} is uniquely determined by $\mathcal E$. In this case the isomorphism of abelian groups \eqref{hiso} 
\begin{equation}\label{isot}
H^{\le r}(k,\theta_{\lle\mathcal E})\colon H^{\le r}\lbe(k,D(G\le))\isoto H^{\le r}\lbe(k,C(\partial_{\le\mathcal E}))
\end{equation}
and the pairing \eqref{bpair2}
\begin{equation}\label{bpair3}
\langle-,-\rangle_{ r,\e\mathcal E}\e\colon  H^{\le r}\lbe(k,G\e)\times H^{\le 1-r}\lbe(k,C(\partial_{\le\mathcal E}))\to\br\e k
\end{equation}
are also uniquely determined by $\mathcal E$.

\section{The local duality theorems}

Henceforth, $k$ denotes a {\it non-archimedean} local field.

\smallskip

Recall that a {\it $k$-split unipotent $k$-group} is an algebraic $k$-group that admits a composition series whose successive quotients are each 
$k$-isomorphic to $\G_{ a,\le k}$. If $\car\le k=0$, then every $G\in\ck$ is smooth and every unipotent $k$-group is $k$-split \cite[${\rm VI}_{\rm B}$, Corollary 1.6.1, and XVII, Theorem 3.5, (i)$\be\implies\!$(ii)]{sga3}.
\smallskip

Let $L$ be an affine and commutative algebraic $k$-group.  By 
\cite[Theorem 3.2.3 and Lemma 3.5.1]{ros23}, $L^{\be D}\lbe(k)$ is finitely generated and $H^{\lle 2}\lbe(k, L^{\be D}\lle)$ is torsion. If $\car k=0$ or $r\neq 1$, $H^{\lle r}\lbe(k, L^{\be D}\lle)$ will be equipped with the {\it discrete} topology. If $k$ is a local function field, $H^{\lle 1}\lbe(k, L^{\be D})$ will be equipped with the topology introduced in \cite[\S3.6]{ros23}. This topology is defined as follows. By \cite[Lemma 2.1.7]{ros23}, $L$ is an extension
\begin{equation}\label{seq}
0\to H\to L\to U\to 0,
\end{equation}
where $H$ is an extension of a finite $k$-group by a $k$-torus and $U$ is a $k$-split unipotent $k$-group. By \cite[Proposition 2.3.1]{ros23}\e, the above sequence induces an exact sequence
$0\to U^{D}\to L^{\be D}\to H^{\lbe D}\to 0$ in $\sks$. Since $H^{\lle 1}\lbe(k, U^{\lbe D})=0$ by \cite[Proposition 2.5.3(iii)]{ros23}, the latter sequence yields an injection $H^{\lle 1}\lbe(k, L^{\be D})\hookrightarrow H^{\lle 1}\lbe(k, H^{\lbe D})$ and we identify $H^{\lle 1}\lbe(k, L^{\be D})$ with its image in $H^{\lle 1}\lbe(k, H^{\lbe D})$ under this injection. Now, by \cite[Proposition 2.3.5]{ros23}, the sheaf $H^{\lbe D}$ is representable by an affine and commutative algebraic $k$-group, whence $H^{\lle 1}\lbe(k, H^{\lbe D})$ is equipped with the topology discussed in section \ref{tops}\,. The subgroup $H^{\lle 1}\lbe(k, L^{\be D})\subset H^{\lle 1}\lbe(k, H^{\lbe D})$ is then equipped with the corresponding subspace topology.

\begin{lemma} \label{rle} Let $L$ be an affine and commutative algebraic $k$-group. Then the topology on $H^{\lle 1}\lbe(k, L^{\be D})$ defined above is independent of the choice of the sequence \eqref{seq}. Further, $H^{\lle 1}\lbe(k, L^{\be D})$ is locally profinite, second countable and of finite exponent.
\end{lemma}
\begin{proof} If $k$ is a local function field, see \cite[Lemma 3.11 and \S3.6 up to and including Lemma 3.27]{ros23}. If $\car k=0$, the first assertion is trivially true. Further, $H^{\lle 1}\lbe(k, L^{\be D})$ is finite by \cite[Theorem G.3 and comments after the proof\e]{ros23}, whence the remaining assertions are also valid in this case.
\end{proof}

The following statement summarizes the local duality theorems established in \cite{ros23} needed in this paper.

\begin{theorem} \label{ros23} Let $L$ be an affine and commutative algebraic $k$-group.
\begin{enumerate}
\item[(i)] The cup product pairing $L^{\be D}\lbe(k)\times H^{\le 2}\lbe(k, L) \to
\Q\le/\le\Z\e$ induces a continuous perfect pairing of Hausdorff and locally compact abelian topological groups
\[
L^{\be D}\be(k)^{\wedge}\times H^{\le 2}\lbe(k, L)\to\Q\le/\le\Z\e.
\]

\smallskip

\item[(ii)] The cup product pairing $L\lle(k)\times H^{\le 2}\lbe(k, L^{\be D})\to
\Q\le/\le\Z\e$ induces a continuous perfect pairing of Hausdorff and locally compact abelian topological groups
\[
L\lle(k)^{\wedge}\times H^{\le 2}\lbe(k,  L^{\be D})\to\Q\le/\le\Z\e.
\]

\smallskip

\item[(iii)] The cup product pairing
\[
H^{\le 1}\lbe(k, L^{\be D})\times H^{\le 1}\lbe(k,  L)\to\Q\le/\le\Z\e
\]
is a continuous perfect pairing of second countable and locally profinite  abelian topological groups of finite exponent.

\end{enumerate}
\end{theorem}
\begin{proof} If $k$ is a local function field, see \cite[Theorems 3.6, 3.25 and 3.28]{ros23} and also Proposition \ref{fc} for the local profiniteness of $H^{\le 1}\lbe(k,  L)$ in (iii). If $\car k=0$ (in which case both cohomology groups in (iii) are finite), see \cite[Theorem G.3 and comments after the proof\e]{ros23}\,.
\end{proof}

\smallskip

Henceforth, $M=(K^{\lbe D}\be,u\e,\mathcal E(L,\iota, G,\pi, A)\le)$ denotes a $1$-motive. In particular, $E^{\le 1}\lbe(K\le)=E^{\le 1}\lbe(L\le)=0$. As noted in Remark \ref{imp1}\e, the preceding conditions are vacuous if $\car k>0$. On the other hand, if $\car k=0$, then the above conditions hold if, and only if, the maximal unipotent subgroups of $K$ and $L$ are trivial.

\smallskip

\begin{lemma}\label{dle} Let $M=(K^{\lbe D}\be,u\e,\mathcal E(L,\iota, G,\pi, A)\le)$ be a $1$-motive. Then $H^{\le r}\lbe(k,M\le)\simeq H^{\le r+1}\lbe(k,K^{\lbe D})$ for every $r\geq 3$ and there exists a canonical exact sequence of abelian groups
\[
\begin{array}{rcl}
0&\to&H^{-1}\lbe(k,M\le)\to K^{\lbe D}\be\le(k)\to G(k)\xrightarrow{f^{\lle(0)}} H^{\le 0}\lbe(k,M\le)\xrightarrow{g^{\lle(0)}} H^{\le 1}\lbe(k,K^{\lbe D}\le)\\
&\xrightarrow{u^{\lle(1)}}& H^{\le 1}\lbe(k,G\e)\xrightarrow{f^{\lle(1)}} H^{\le 1}\lbe(k,M\le)\xrightarrow{g^{\lle(1)}} H^{\le 2}\lbe(k,K^{\lbe D})\to  H^{\le 2}\lbe(k,G\e)\to  H^{\le 2}\lbe(k,M\le)\to 0.
\end{array}
\]
\end{lemma}
\begin{proof} By \cite[Propositions 3.2 and 3.3]{ros23}, $H^{\le 3}\lbe(k,K^{\lbe D})=0$ and $H^{\le r}\lbe(k,G\e)=0$ for all $r\geq 3$. The displayed sequence, as well as the first assertion of the lemma, now follow by taking (hyper)cohomology of the sequence \eqref{csc}\e.
\end{proof}

For every $1$-motive $M$, the pairings \eqref{bpair} and the invariant map ${\rm inv}_{k}\colon \br\e k\isoto\Q/\Z$ of local class field theory \cite[XIII, \S 3, Proposition 6, p.~193]{ser2} induce canonical pairings of abelian groups for every $r\in\Z$
\begin{equation}\label{bpgp}
\hskip 1.5cm \langle-,-\rangle_{ r,\le M}\e\colon H^{\lle r}\lbe(k,M\e)\times H^{\le 1-r}\lbe(k,D(M\le))\to \Q/\Z\,.
\end{equation}
On the other hand, for any fixed choice of the isomorphism $\theta_{M}\colon D(M\le)\isoto [\lle M^{\le\diamond}\e]$ \eqref{teta}\le, the pairing \eqref{bpair2} induces a pairing of abelian groups
\begin{equation}\label{bpair2.1}
\langle-,-\rangle_{ r,\e\le \theta_{ M}}\e\colon  H^{\le r}\lbe(k,M\e)\times H^{\le 1-r}\lbe(k,M^{\le\diamond})\to\Q/\Z.
\end{equation}

\smallskip

Now let $G$ be a commutative algebraic $k$-group, let $\mathcal E=\mathcal E(L,\iota, G,\pi, A)$ be an object of $\s E_{\le\lle G}$ such that $L$ contains no nontrivial unipotent $k$-subgroups if $\car\,k=0$ and let $M=(0,0,\mathcal E(L,\iota, G,\pi, A)\le)=(0,0,\mathcal E\e)$ be the associated $1$-motive. Then \eqref{bpair3} and ${\rm inv}_{k}\colon \br\e k\isoto\Q/\Z$ induce a pairing
\begin{equation}\label{bpair3.1}
\langle-,-\rangle_{ r,\e\mathcal E}\e\colon  H^{\le r}\lbe(k,G\e)\times H^{\le 1-r}\lbe(k,C(\partial_{\le\mathcal E}))\to \Q/\Z
\end{equation}
that is uniquely determined by $\mathcal E$.

As discussed in section \ref{tops}\e, the abelian groups $H^{\le r}\lbe(k,G\le)$ in \eqref{bpair3.1} are equipped with natural topologies. Next we will define natural topologies on the abelian groups $H^{\le r}\lbe(k,C(\partial_{\le\mathcal E}))$ that intervene in \eqref{bpair3.1}.

\smallskip

Since $H^{\le 2}\lbe(k,A^{\vee}\le)=0$ by \cite[proof of Theorem I.3.2 and Theorem III.7.8]{adt}, Lemma \ref{dle} applied to the $1$-motive $(L^{\be D}\be,\partial_{\le\mathcal E}\e, \mathcal E(0,0, A^{\vee},{\rm Id}_{\le A^{\vee}}, A^{\vee}\le))$ yields an exact sequence of abelian groups
\begin{equation}\label{bseq}
\begin{array}{rcl}
0&\lra &H^{-1}\lbe(k,C(\partial_{\le\mathcal E}))\lra L^{\be D}\be\le(k)\xrightarrow{\partial_{\lbe\mathcal E}^{\le(0)}}A^{\vee}\be(k)\xrightarrow{f^{\lle(0)}} H^{\le 0}\lbe(k,C(\partial_{\le\mathcal E}))\\
&\xrightarrow{g^{\lle(0)}}& H^{\le 1}\lbe(k,L^{\be D}\le)\xrightarrow{\partial_{\lbe\mathcal E}^{\le(1)}}H^{\le 1}\lbe(k,A^{\vee}\le)\xrightarrow{f^{\le(1)}} H^{\le 1}\lbe(k,C(\partial_{\le\mathcal E}))\xrightarrow{g^{\lle(1)}}  H^{\le 2}\lbe(k,L^{\be D})\lra 0.
\end{array}
\end{equation}

\smallskip

\begin{lemma}\label{kont} Let $G$ and $\mathcal E=\mathcal E(L,\iota, G,\pi, A)$ be as above. Then, for every $\eta\in H^{\le 0}\lbe(k,C(\partial_{\le\mathcal E}))$, the homomorphism
\[
\langle\e -, \eta\,\rangle_{1,\e \mathcal E}\colon H^{\le 1}\lbe(k,G\e)\to\Q/\Z
\]
induced by the pairing \eqref{bpair3.1} is continuous at $0$.
\end{lemma}
\begin{proof} By the functoriality of the map $E_{\e(-)}$ \eqref{dev}, the following diagram in $D(k_{\fl})$ commutes:
\[
\xymatrix{L\otimes^{\mathbf{L}}\! D(G\le)\ar[d]_{\iota\e\otimes\e\id_{\lle D(G\lle)}}\ar[rr]^(.51){\id_{L}\e\otimes\le D\le(\le\iota\lle)}&& L\otimes^{\mathbf{L}}\! D(L\le)\ar[d]^{E_{\lle L}}\\
G\otimes^{\mathbf{L}}\! D(G\le)\ar[rr]^(.51){E_{\le G}}&&\G_{m,\e k}[1].
}
\]
It now follows from the commutativity of the right-hand square in diagram \eqref{cdl2} and the definition of the pairing \eqref{bpair3.1} that the following diagram of abelian groups commutes
\begin{equation}\label{cod}
\begin{tikzcd}
\langle-,-\rangle_{\e 1\le,\, L}\colon  H^{\le 1}\lbe(k,L\le) \arrow[d,shift left=3.5ex,"i^{\lle(1)}"] \arrow[r, phantom, "\hskip -.2cm\times"] &  H^{\le 1}\lbe(k,L^{\be D}) \arrow{r} & \Q/\Z \\
\langle-,-\rangle_{\e 1\le,\,\mathcal E}\e\colon H^{\le 1}\lbe(k,G\e) \arrow[r, phantom, "\times"]& H^{\le 0}\lbe(k,C(\partial_{\le\mathcal E}))  \arrow[u,"g^{\lle(0)}"'] \arrow{r} & \Q/\Z, \arrow[u, equals]
\end{tikzcd}
\end{equation}
where $g^{\lle(0)}$ is the map in \eqref{bseq}. By Theorem \ref{ros23}(iii), the top pairing $\langle-,-\rangle_{\e 1\le,\, L}$ is continuous, whence there exists an open neighborhood $U$ of $0_{H^{1}\lbe(k,\le L\le)}$ such that $\langle U,g^{\lle(0)}\lbe(\eta)\e\rangle_{\e 1\le,\, L}=\langle\le i^{(1)}\lbe(U), \eta\e\rangle_{1,\e \mathcal E}=\{0\}$. Since $i^{(1)}$ is an open map by Corollary \ref{got}\e, the lemma follows.
\end{proof}

\smallskip

\begin{proposition}\label{ko} The map $\partial_{\le\mathcal E}^{\le(1)}\colon H^{\lle 1}\lbe(k,L^{\be D})\to H^{\lle 1}\lbe(k,A^{\vee}\lle)$ in \eqref{bseq} is continuous and strict.
\end{proposition}
\begin{proof} Since $H^{\lle 1}\lbe(k,A^{\vee}\lle)$ is discrete by Proposition \ref{topg}\e, only the continuity of $\partial_{\lbe\mathcal E}^{\le(1)}$ needs checking. If $\car k=0$, then $H^{\lle 1}\lbe(k,L^{\be D})$ is finite and discrete \cite[Theorem G.3\e]{ros23} and $\partial_{\lbe\mathcal E}^{\le(1)}$ is clearly continuous. Assume now that $k$ is a local function field. By Proposition \ref{bdg}\e, $\partial_{\lbe\mathcal E}^{\le(1)}$ factors as
\[
H^{\lle 1}\lbe(k,L^{\be D})\xrightarrow{(\e\jmath^{\lle D}\lbe)^{\lle(1)}} H^{\lle 1}\lbe(k,(G_{\sab}\times_{G}L)^{D}\le)\xrightarrow{\partial_{\lbe\mathcal E^{\lbe\prime}}^{\le(1)}}  H^{\lle 1}\lbe(k,A^{\vee}\le).
\]
The first map is continuous by \cite[Proposition 3.29]{ros23}. The second map is also continuous by \cite[Proposition 4.2]{ces} since the sheaf $(G_{\sab}\times_{G}L)^{D}$ is representable. The proposition is now clear.
\end{proof}

\smallskip

Now we recall the maps $f^{\lle(0)}\colon A^{\vee}\lbe(k)\to H^{\le 0}\lbe(k,C(\partial_{\le\mathcal E}))$ and $g^{(0)}\colon H^{\le 0}\lbe(k,C(\partial_{\le\mathcal E}))\to H^{\le 1}\lbe(k,L^{\be D}\le)$ from the sequence \eqref{bseq}\e. The group
\begin{equation}\label{coi}
\coim f^{(0)}=A^{\vee}\lbe(k)/\e\krn f^{\lle(0)}=A^{\vee}\lbe(k)/\e\img \partial_{\lbe\mathcal E}^{\le(0)}
\end{equation}
(respectively, $\img g^{\lle(0)}\subset H^{1}(k,L^{\be D}\le)$) will be equipped with the quotient (respectively, subspace) topology of the topology of $A^{\vee}\lbe(k)$ (respectively, $H^{1}(k,L^{\be D}\le)$). 

\begin{corollary}\label{topr}\indent
\begin{enumerate} 
\item[(i)] $\coim f^{\lle(0)}$ \eqref{coi} is a compact and second countable abelian topological group.	
\item[(ii)] $\img g^{(0)}$ is a locally profinite  and second countable abelian topological group.
\end{enumerate}
\end{corollary}
\begin{proof} Assertion (i) follows from Proposition \ref{topg} and \cite[Lemma 6.2(a) and Theorem 6.7(b)]{str}\,. Now, since $\partial_{\lbe\mathcal E}^{\le(1)}\colon H^{\le 1}\lbe(k,L^{\be D}\le)\to H^{\le 1}\lbe(k,A^{\vee}\le)$ is continuous by Proposition \ref{ko} and $H^{\le 1}\lbe(k,A^{\vee})$ is Hausdorff, $\img g^{(0)}=\krn \partial_{\lbe\mathcal E}^{\le(1)}$ is a closed subgroup of $H^{\le 1}\lbe(k,L^{\be D}\le)$. Thus assertion (ii) follows at once from Lemma \ref{rle}\,.
\end{proof}

Now Lemma \ref{kont} provides a map
\begin{equation}\label{gmap}
\gamma\colon H^{\le 0}\lbe(k,C(\partial_{\le\mathcal E}))\to H^{\le 1}\lbe(k,G\e)^{*}, \eta\to\langle\e -, \eta\e\rangle_{\lbe 1,\e\mathcal E},
\end{equation}
that fits into the following diagram of abelian groups with exact rows, where the top row is part of \eqref{bseq} and the bottom row is strict exact by Corollary \ref{got}\e:
\begin{equation}\label{hell}
\xymatrix{A^{\vee}\lbe(k)\ar[r]^(.4){f^{(0)}}\ar[d]^{\simeq}&
H^{\le 0}\lbe(k,C(\partial_{\le\mathcal E}))\ar[r]^(.53){g^{(0)}}\ar[d]^(.45){\gamma}&H^{\le 1}\lbe(k,L^{\be D})\ar[d]^{\simeq}\ar[r]^{\partial_{\mathcal E}^{\le(1)}}& H^{\le 1}\lbe(k,A^{\vee})\ar[d]^{\simeq}\\
H^{\e 1}(k,A\le)^{*}\ar[r]^{(\pi^{(\lbe 1\lbe)}\lbe)^{*}}&
H^{\le 1}\lbe(k,G\e)^{*}\ar[r]^{(i^{(\lbe 1\lbe)}\lbe)^{*}}&H^{\le 1}\lbe(k,L\le)^{*}\ar[r]^(.55){(h^{(0)})^{*}}& A(k)^{*}.
}
\end{equation}
All vertical arrows except $\gamma$ are topological isomorphisms.

\begin{lemma} The diagram \eqref{hell} commutes.    
\end{lemma}
\begin{proof} The middle square of diagram \eqref{hell} commutes by the commutativity of diagram \eqref{cod}.
For the left\le-\le hand square, the functoriality of $E_{\e(-)}$ \eqref{dev} yields the following
commutative diagram in $D(k_{\fl})$:
\[
\xymatrix{G\otimes^{\le\mathbf L}\be D(A\le)\ar[rr]^{\id_{\lle G}\otimes\e D(\pi\lle)} \ar[d]_{\pi\le\otimes\e\id_{D(A\lle )}}&&G\otimes^{\le\mathbf L}\be D(G\le)\ar[d]^{E_{\le G}}\\
A\otimes^{\le\mathbf L}\be D(A\le)\ar[rr]_{E_{\lbe A}}&&\gmk[1].
}
\]
It now follows from the commutativity of the middle square in diagram \eqref{cdl2} that the following diagram of abelian groups commutes
\[
\begin{tikzcd}
\langle-,-\rangle_{\e 1\le,\, A}\colon   H^{\le 1}(k,A)\hskip -.7cm \arrow[r, phantom, "\hskip .2cm\times"] &  A^{\vee}\lbe(k)\arrow[d,shift left=.2ex,"f^{\lle(0)}"]\arrow{r} & \Q/\Z \\
\hskip 1cm\langle-,-\rangle_{\e 1,\e\mathcal E}\e\colon  H^{\le 1}\lbe(k,G\e)\arrow[u,shift right=8ex,"\pi^{\lle(1)}"'] \arrow[r, phantom, "\times"]& H^{\le 0}\lbe(k,C(\partial_{\le\mathcal E}))   \arrow{r} & \Q/\Z, \arrow[u, equals]
\end{tikzcd}
\]
where $\langle-,-\rangle_{\e 1\le,\, A}$ is the Tate pairing \eqref{tpa}. The commutativity of the left\le-\le hand square in diagram \eqref{hell} is now immediate. The commutativity of the right\le-\le hand square in diagram \eqref{hell} is established similarly, starting with the commutative diagram
\[
\xymatrix{A\otimes^{\le\mathbf L}\be D(L[1])\ar[rr]^{\id_{\lle A}\otimes\e D(h\lle)} \ar[d]_{h\le\otimes\e\id_{D(L[1]\lle )}}&& A\otimes^{\le\mathbf L}\be D(A\le)\ar[d]^{E_{\lbe A}}\\
L[1]\otimes^{\le\mathbf L}\be D(L[1])\ar[rr]_{E_{\le L[1]}}&&\gmk[1].
}
\]
and using the commutativity of the left\le-\le hand square in diagram \eqref{cdl2}.
\end{proof}

\smallskip

It follows from the commutativity of diagram \eqref{hell} and \cite[Lemma 3.3(ii)]{mac} that the map  $\gamma$ \eqref{gmap} is {\it surjective}. 

\smallskip

\begin{definition}\label{pea} We equip $H^{\le 0}\lbe(k,C(\partial_{\le\mathcal E}))$ with the topology induced from the topology of $H^{\le 1}\lbe(k,G\e)^{*}$ via the map $\gamma\colon H^{\le 0}\lbe(k,C(\partial_{\le\mathcal E}))\onto H^{\le 1}\lbe(k,G\e)^{*}$ in diagram \eqref{hell}. Thus, the open sets of $H^{\le 0}\lbe(k,C(\partial_{\le\mathcal E}))$ are exactly the sets of the form $\gamma^{-1}(U)$, where $U$ is an open subset of $H^{\le 1}\lbe(k,G\e)^{*}$ \cite[Definition 3.1]{jam}.
\end{definition}

Clearly, $\gamma\colon H^{\le 0}\lbe(k,C(\partial_{\le\mathcal E}))\onto H^{\le 1}\lbe(k,G\e)^{*}$ is continuous. Further, using \cite[Proposition 3.3]{jam}, it is not difficult to check that the topology of $H^{\le 0}\lbe(k,C(\partial_{\le\mathcal E}))$  defined above is a group topology.

Now, diagram \eqref{hell} induces the following exact and commutative diagram of 
abelian topological groups and continuous homomorphisms
\begin{equation}\label{hell2}
\xymatrix{0\ar[r]&\coim f^{(0)}\ar[r]^(.48){\widetilde{f^{(0)}}}\ar@{->>}[d]^(.45){\alpha}&
H^{\le 0}\lbe(k,C(\partial_{\le\mathcal E}))\ar[r]^(.54){\widetilde{g^{(0)}}}\ar@{->>}[d]^(.45){\gamma}&\img g^{(0)}\ar[d]^(.45){\beta}_(.45){\simeq}\ar[r]& 0\\
0\ar[r]&\coim((\pi^{\lle (\lbe 1\lbe)})^{*})\ar[r]^(.52){\widetilde{(\lbe\pi^{(\lbe 1\lbe)}\lbe)^{\lbe*}}}&
H^{\le 1}\lbe(k,G\e)^{*}\ar[r]^(.52){\widetilde{(\lbe\iota^{(\lbe 1\lbe)}\lbe)^{\lbe*}}}&\img((i^{\lle (\lbe 1\lbe)})^{*})\ar[r]&0,
}
\end{equation}
where the bottom row is a topological extension.

\begin{proposition}\label{kon} The maps $\alpha$, $\gamma$ and $\widetilde{g^{(0)}}$ in diagram \eqref{hell2} are continuous and strict.
\end{proposition}
\begin{proof} The assertion regarding $\alpha$ follows by applying Lemma \ref{prev} to the commutative diagram of abelian topological groups
\[
\xymatrix{A^{\vee}\lbe(k)\ar@{->>}[r]\ar[d]^(.45){\simeq}&\coim f^{(0)}\ar@{->>}[d]^(.45){\alpha}\\
H^{\e 1}(k,A\le)^{*}\ar@{->>}[r]&\coim((\pi^{\lle (\lbe 1\lbe)})^{*}),
}
\]
whose bottom horizontal map is open by \cite[Lemma 6.2(a)]{str}. The assertions regarding $\gamma$ and $\widetilde{g^{(0)}}$ follow at once from Definition \ref{pea} using the surjectivity of $\gamma$ and the commutativity of the right-hand square in diagram \eqref{hell2}. 
\end{proof}

\smallskip

\begin{remark} Since the map $\widetilde{f^{(0)}}$ in diagram \eqref{hell2} is continuous, Corollary \ref{topr} and \cite[Theorems 16.4(a) and 16.9(b)]{wil} show that $\krn \widetilde{g^{\le(0)}}=\img \widetilde{f^{(0)}}$ is a closed, compact and separable (but perhaps not second countable) subgroup of $H^{\le 0}\lbe(k,C(\partial_{\le\mathcal E}))$. Consequently, the (generally non\le-\le Hausdorff) topological group $H^{\le 0}\lbe(k,C(\partial_{\le\mathcal E}))$ is locally compact.
\end{remark}

\smallskip

\begin{proposition} \label{kon00} The pairing \eqref{bpair3.1}
\begin{equation}\label{bpg1}
\langle-,-\rangle_{ 1,\e\mathcal E}\e\colon H^{\lle 1}\lbe(k,G\e)\times H^{\le 0}\lbe(k,C(\partial_{\le\mathcal E}))\to \Q/\Z	
\end{equation}
is continuous. 
\end{proposition}
\begin{proof} Since $\gamma\colon H^{\le 0}\lbe(k,C(\partial_{\le\mathcal E}))\to H^{\le 1}\lbe(k,G\e)^{*}$ is continuous, Lemma \ref{kont} and Proposition \ref{pt} reduce the proof to checking continuity at $(0,0)$. By Theorem \ref{ros23}(iii), the pairing
$\langle-,-\rangle_{ 1\le,\e L}\colon  H^{\le 1}\lbe(k,L\le)\times H^{\le 1}\lbe(k,L^{\be D})\to \Q/\Z$ is continuous. Consequently, there exist open neighborhoods $U$ of $0_{H^{\le 1}\lbe(k,\le L\le)}$ and $V$ of $0_{H^{\le 1}\lbe(k,\le L^{\lbe D})}$ such that $\langle U,V\le\rangle_{ 1\le,\e L}=\{0\}$. Now, since $i^{(1)}\colon H^{1}(k,L)\to H^{1}(k,G\e)$ is open by Corollary \ref{got} and the map $g^{(0)}\colon H^{\le 0}\lbe(k,C(\partial_{\le\mathcal E}))\to H^{\le 1}(k,L^{\be D})$ is continuous, the sets $i^{(1)}(U)$ and $(\le g^{(0)}\lbe)^{-1}(V)$ are open neighborhoods of $0_{H^{\lle 1}\lbe(k,\le G\e)}$ and $0_{H^{\le 0}\lbe(k,\le C(\partial_{\le\mathcal E}))}$, respectively. Finally, the commutativity of diagram \eqref{cod} yields
\[
\langle\le i^{(1)}\lbe(U), (\le g^{(0)}\lbe)^{-1}(V)\rangle_{\le 1,\e\mathcal E}=\langle\le U, g^{(0)}(\le g^{(0)}\lbe)^{-1}(V)\e\rangle_{\le 1\le,\, L}\subseteq \langle\le U, V\e\rangle_{\le 1\le,\, L}=\{0\},
\]
i.e., the pairing \eqref{bpg1} is continuous at $(0,0)$.
\end{proof}

\smallskip

\begin{theorem} \label{big} Let $k$ be a non-archimedean local field, let $G$ be a commutative algebraic $k$-group and let  $\mathcal E\colon 0\to L\to G\to A\to 0$ be an object of $\s E_{\le G}$ such that $L$ contains no nontrivial unipotent $k$-subgroups if $\car\,k=0$. Then the pairing \eqref{bpg1} induces a continuous perfect pairing of Hausdorff and locally compact abelian topological groups
\begin{equation}\label{npa}
H^{\le 1}\lbe(k,G\e)\times H^{\le 0}\lbe(k,C(\partial_{\le\mathcal E}))_{\rm Haus}\to\Q/\lle\Z.
\end{equation}
\end{theorem}
\begin{proof} The pairing \eqref{npa} is continuous by Proposition \ref{kon00} and Remark \ref{dct}(b). Now consider the commutative diagram of abelian groups
\begin{equation}\label{mdi}
\xymatrix{L^{\be D}\be(k)^{\wedge}\ar[d]^{\simeq}\ar[r]^{(\lbe \partial_{\mathcal E}^{\le(0)}\lbe)^{\wedge}}& A^{\vee}\lbe(k)\ar[r]^(.4){f^{(0)}_{\rm Haus}}\ar[d]^{\simeq}&
H^{\le 0}\lbe(k,C(\partial_{\le\mathcal E}))_{\rm Haus}\ar[r]^(.57){ g^{(0)}_{\lle\rm Haus}}\ar[d]^{\gamma_{\e\rm Haus}}&H^{\le 1}\lbe(k,L^{\be D})\ar[d]^{\simeq}\ar[r]^{\partial_{\mathcal E}^{\le(1)}}& H^{\le 1}\lbe(k,A^{\vee})\ar[d]^{\simeq}\\
H^{\le 2}\lbe(k,L\le)^{*}\ar[r]&H^{\e 1}(k,A\le)^{*}\ar[r]^{(\pi^{(\lbe 1\lbe)}\lbe)^{*}}&
H^{\le 1}\lbe(k,G\e)^{*}\ar[r]^{(i^{(\lbe 1\lbe)}\lbe)^{*}}&H^{\le 1}\lbe(k,L\le)^{*}\ar[r]& A(k)^{*}.
}
\end{equation}
By Corollary \ref{got}\,, the bottom row is a strict exact sequence of abelian topological groups. Further, all maps on the top row are continuous and strict by Lemma \ref{lem:stc}\le(i) and Propositions \ref{ko} and \ref{kon}\e. On the other hand, the first and fourth (respectively, second and fifth) vertical maps are topological isomorphisms by Theorem \ref{ros23}\,, (i) and (iii) (respectively, Milne\,-Tate duality). The third vertical map $\gamma_{\le\rm Haus}$ is continuous and strict since the topology of $H^{\le 0}\lbe(k,C(\partial_{\le\mathcal E}))_{\rm Haus}$ is the quotient of the topology of $H^{\le 0}\lbe(k,C(\partial_{\le\mathcal E}))$ and $\gamma\colon H^{\le 0}\lbe(k,C(\partial_{\le\mathcal E}))\to H^{\le 1}\lbe(k,G\e)^{*}$ is continuous and strict by Proposition \ref{kon}\e. Now, by Lemma \ref{nos}\,, the top row is exact except perhaps at $A^{\vee}\lbe(k)$. Since $f^{(0)}_{\rm Haus}$ is continuous and
\[
\img((\lbe \partial_{\mathcal E}^{\le(0)})^{\wedge})=\overline{\img \partial_{\mathcal E}^{\le(0)}}=\overline{\krn f^{(0)}}
\]
by Lemma \ref{lem:eas}\,, we have
\[
f^{(0)}_{\rm Haus}(\e\img((\lbe \partial_{\mathcal E}^{\le(0)})^{\wedge}))=f^{(0)}_{\rm Haus}(\e\overline{\krn f^{(0)}}\e)\subseteq \overline{f^{(0)}_{\rm Haus}({\krn f^{(0)}})}=\overline{\{0\}}=\{0\}\e,
\]
since $H^{\le 0}\lbe(k,C(\partial_{\le\mathcal E}))_{\rm Haus}$ is Hausdorff. Thus $\img((\lbe \partial_{\mathcal E}^{\le(0)})^{\wedge})\subseteq \krn f^{(0)}_{\rm Haus}$. The reverse inclusion follows from the commutativity of the first two squares of diagram \eqref{mdi}, the exactness of the bottom row at $H^{\lle 1}\lbe(k,A\le)^{*}$, the surjectivity of the first vertical map and the injectivity of the second vertical map. Thus \eqref{mdi} is an exact and commutative diagram in the category of abelian topological groups. It now follows from the five lemma that the map $\gamma_{\le\rm Haus}\colon  H^{\le 0}\lbe(k,C(\partial_{\le\mathcal E}))_{\rm Haus}\to H^{\le 1}\lbe(k,G\e)^{*}$ is an algebraic and topological isomorphism. Finally, the Pontryagin duality theorem completes the proof. 
\end{proof}

\begin{remark} The proof of the theorem shows that $H^{\le 0}\lbe(k,C(\partial_{\le\mathcal E}))_{\rm Haus}$ is a topological extension
\[
0\to A^{\vee}\lbe(k)/\,\overline{\img \partial_{\mathcal E}^{\le(0)}}\to H^{\le 0}\lbe(k,C(\partial_{\le\mathcal E}))_{\rm Haus}\to \krn\be[H^{\e 1}(k,L^{\be D})\overset{\partial_{\mathcal E}^{\le(1)}}{\to} H^{1}(k,A^{\vee})]\to 0.
\]
The left (respectively, right)-hand group above is profinite and second countable by Proposition \ref{topg} (respectively, locally profinite, second countable and of finite exponent by Lemma \ref{rle}). In particular, $H^{\le 0}\lbe(k,C(\partial_{\le\mathcal E}))_{\rm Haus}$ is second countable. If $G$ is smooth, then $H^{\e 1}(k,G\e)$ is discrete and torsion by Proposition \ref{fc} and $H^{\le 0}\lbe(k,C(\partial_{\le\mathcal E}))_{\rm Haus}\simeq H^{\le 1}\lbe(k,G\e)^{*}$ is profinite and second countable.
\end{remark}

\smallskip

We now turn to the proofs of Theorems \ref{th-main}(i) and \ref{main2}\e.

\medskip

Let $G$ be as in Theorem \ref{big}\e. By the exactness of the sequence \eqref{bseq}, the abelian group $H^{\le 1}\lbe(k,C(\partial_{\le\mathcal E}))$ is torsion. We equip it with the {\it discrete} topology.
If $M$ is an arbitrary $1$-motive, the groups $H^{-1}\lbe(k,M\le)$ and $H^{\le 2}\lbe(k,M\le)$ \eqref{mm} will also be equipped with the discrete topology.

\begin{proposition} \label{ghp} The pairing \eqref{bpair3.1}
\begin{equation}\label{bpg0}
\langle-,-\rangle_{ 0\le,\e\mathcal E}\e\colon G\lbe(k)\times H^{\le 1}\lbe(k,C(\partial_{\le\mathcal E}))\to \Q/\Z	
\end{equation}
is continuous. 
\end{proposition}
\begin{proof} By the discreteness of 
$H^{\le 1}\lbe(k,C(\partial_{\le\mathcal E}))$, it suffices to check that, for every $\eta\in H^{\le 1}\lbe(k,C(\partial_{\le\mathcal E}))$, there exists an open neighborhood $U$ of $0_{\le G(k)}$ such that $\langle\e U,\eta\e\rangle_{ 0,\e\mathcal E}=0$ (cf. Remark \ref{ltd}). The isomorphism $\theta_{\mathcal E}\colon D(G\le)\isoto C(\partial_{\le\mathcal E})$ \eqref{tee} and the evaluation pairing $G\otimes^{\mathbf{L}}\! D(G\le)\to \G_{m,\e k}[1]$ induce a morphism $G\otimes^{\mathbf L}C(\partial_{\le\mathcal E})\to\gmk[1]$ in $D(k_{\fl})$. The latter map induces, in turn, a cup product pairing
\[
(-,-)_{\le G, \e\mathcal E}\colon H^{\le 0}(G,G\e)\times H^{\le 1}(G,C(\partial_{\le\mathcal E}))\to \br(G\le).
\]
Now, there exists a canonical bijection
\[
\Hom_{\e k-{\rm Sch}}(G,G\e)\isoto \Hom_{\e G}(G, G\times_{k}G\le)=H^{\le 0}(G,G\e)
\]
that maps a $k$-morphism $f\colon G\to G$ to the $G$-morphism $(\id_{\le G},f\le)_{G}\colon G\to G\times_{k}G$. In particular, the image of $f=\id_{\le G}$ is the diagonal morphism $\Delta_{\le G/k}\colon G\to G\times_{k}G$. We now consider the torsion class $q^{\lle *}\eta\in H^{\le 1}(G,C(\partial_{\le\mathcal E}))$, where $q\colon G\to\spec k$ is the structural morphism of $G$. Set
\[
b_{\le\eta}=(\Delta_{\le G/k},\le q^{\lle *}\eta)_{\le G, \e\mathcal E}\in \br\lbe(G\le)_{\rm tors}.
\]
By functoriality, we have
$\langle\e x,\eta\e\rangle_{ 0,\e\mathcal E}
={\rm inv}_{k}\le(x^{\lle *}\lbe(b_{\le\eta}))$ for every $x\in G(k)$. In particular, $0^{\le*}\lbe\lbe(b_{\le\eta})=0$ in $\br\e k$. Now choose a connected open affine neighborhood $V\subseteq G$ of $0_{\le G(k)}$. By Gabber's Theorem \cite[Theorem 7]{hoo} and the identity $0^{\le*}\lbe\lbe(b_{\le\eta})=0$, the class $b_{\le\eta}|_{V}\in \br(V\le)_{\rm tors}$ is represented by an Azumaya algebra $\mathscr A$ over $V$ such that $0^{\le*}\be \mathscr A\simeq M_{d}\lle(k)$, where $d$ is the (constant) degree of $\mathscr A$. Thus the scheme 
\[
P=\isomv(\mathscr A, M_{d}\lle({\mathcal O}_{\le V})\le)
\]
is a ${{\rm PGL}}_{\le d,\le V}\le$-\le torsor over $V$ such that 
\[
P(k)=\coprod_{x\le\in\le V\lbe(k)}(x^{*}\be P\e)(k)=\coprod_{x\le\in\le V\lbe(k)}{\rm Isom}_{\le k-{\rm alg}}(x^{\le*}\be \mathscr A,
M_{d}\lle(k)),
\]
where $(0^{\le*}\be P\le)(k)={\rm Isom}_{\le k-{\rm alg}}(0^{\le*}\be \mathscr A, M_{d}\lle(k))\neq\emptyset$. Observe that $(x^{*}\be P\e)(k)\neq\emptyset\Longleftrightarrow x^{\le*}\be \mathscr A$ is split $\Longleftrightarrow x^{\lle *}\lbe(b_{\le\eta})=0$ in $\br\e k$.

Now, set
\[
U\overset{{\rm def.}}{=}\{x\in V(k)\colon (x^{*}\be P\e)(k)\neq\emptyset\}=\img\be( P(k)\to V(k)).
\]
Since $P\to V$ is smooth and $(0^{\le*}\be P)(k)\neq\emptyset$, $U$ is an open neighborhood of $0_{\le G(k)}$ \cite[Proposition 2.9(a)]{ces}. Further, $x\in U\Longleftrightarrow x^{\lle *}\lbe(b_{\le\eta})=0$ in $\br\e k$, whence $\langle\e x,\eta\e\rangle_{ 0,\e\mathcal E}
={\rm inv}_{k}\le(x^{\lle *}\lbe(b_{\le\eta}))=0$ for all $x\in U$, i.e., $\langle\e U,\eta\e\rangle_{ 0,\e\mathcal E}=0$, as required.
\end{proof}

\smallskip

\begin{theorem}\label{big2} Let $k$ be a non-archimedean local field, let $G$ be a commutative algebraic $k$-group and let  $\mathcal E\colon 0\to L\to G\to A\to 0$ be an object of $\s E_{\le G}$ such that $L$ contains no nontrivial unipotent $k$-subgroups if $\car\,k=0$. Then the pairing \eqref{bpgp} induces a continuous perfect pairing of Hausdorff and locally compact abelian topological groups
\begin{equation}\label{ggv}
G\lbe(k)^{\wedge}\times H^{\le 1}\lbe(k,C(\partial_{\le\mathcal E}))\to\Q/\lle\Z\e.
\end{equation}
\end{theorem}
\begin{proof} By Proposition \ref{ghp} and Lemma \ref{lem:stc}\le(i), the pairing \eqref{bpgp} induces a continuous and strict morphism $\zeta\colon G\lbe(k)^{\wedge}\to H^{\le 1}\lbe(k,C(\partial_{\le\mathcal E}))^{*}$. We now consider the following commutative diagram of abelian groups, whose top (respectively, bottom) row is induced by \eqref{bex2} (respectively, \eqref{bseq}): 
\begin{equation}\label{kat}
\xymatrix{0\ar[r]&L\lbe(k)^{\wedge}\ar[r]\ar[d]^(.45){\eta}_(.45){\simeq}&
G\lbe(k)^{\wedge}\ar[r]^{\pi(k)^{\wedge}}\ar[d]^(.45){\zeta}& A\lbe(k)\ar[d]^{\simeq}\ar[r]^(.45){h^{(0)}}& H^{1}(k,L\le)\ar[d]^{\simeq}\\
0\ar[r]&H^{\lle 2}\lbe(k,L^{\be D})^{*}\ar[r]^(.5){(\le g^{(1)}\lbe)^{*}}&H^{\le 1}\lbe(k,C(\partial_{\le\mathcal E}))^{*}\ar[r]^(.52){(\le f^{\lle(1)}\lbe)^{*}}&
H^{\lle 1}\lbe(k,A^{\vee})^{*}\ar[r]^(.48){(\partial_{\le\mathcal E}^{\lle (1)})^{*}}&H^{\lle 1}\lbe(k,L^{\be D})^{*}.
}
\end{equation}
The bottom row is exact since $\partial_{\le\mathcal E}^{\lle (1)}$ is continuous by Proposition \ref{ko} and the groups $H^{\lle 1}\lbe(k,A^{\vee}), H^{\le 1}\lbe(k,C(\partial_{\le\mathcal E}))$ and $H^{\lle 2}\lbe(k,L^{\be D})$ are discrete. The first and fourth vertical maps are isomorphisms by Theorem \ref{ros23}\e, (ii) and (iii), and the third vertical map is an isomorphism by Milne\,--\e Tate duality. We claim that the top row, which is induced by the exact sequence in Corollary \ref{got}\,, is exact at $A(k)$. Indeed, since $H^{1}(k,L\le)$ is Hausdorff, $\img \pi(k)=\krn h^{(0)}=(h^{(0)})^{\lle-1}\lbe(0)\subseteq A(k)$ is closed and therefore profinite, whence $\img (\pi(k)^{\wedge})=(\img \pi(k))^{\wedge}=\img \pi(k)=\krn h^{(0)}$ by the right-exactness of the profinite completion functor. The top row of the diagram is also exact at $G\lbe(k)^{\wedge}$ (respectively, $L\lbe(k)^{\wedge}$) by \cite[Lemma 2.1]{gat09} (respectively, the commutativity of the first square of the diagram, the injectivity of the map $\eta$ and the injectivity of $(\le g^{(1)}\lbe)^{*}$). Thus \eqref{kat} is an exact and commutative diagram of abelian groups. Now the five lemma shows that the map $\zeta$ in \eqref{kat} is an isomorphism of abelian groups and thus a topological isomorphism since $\zeta$ is continuous and strict.
\end{proof}

\medskip

Finally, we prove Theorem \ref{main2}\e.

\medskip

Let $M=(K^{\lbe D}\be,u\e,\mathcal E(L,\iota, G,\pi, A)\le)$ be an arbitrary $1$-motive and set
\[
H^{\e -1}_{\lbe\wedge\be}\lbe(k,M\le)=\krn\!\!\lbe\left[\lbe K^{\lbe D}\be(k)^{\wedge}\,\overset{\! u(k)^{\wedge}}{\lra} G(k)^{\wedge}\le\right]\!.
\]
Since $K^{\lbe D}\lbe(k)$ is discrete, the map $u(k)$ is continuous and therefore $H^{\e -1}_{\lbe\wedge\be}\lbe(k,M\le)$ is a profinite group. Set
\[
\hskip 1.4cm M^{\le\prime}=(K^{\lbe D}, v,\mathcal E(\le 0,0, A,{\rm Id}_{A}, A\le)),
\]
so that $[M^{\le\prime}\e]=C(v)$, where $v=\pi\circ u\colon K^{\lbe D}\to A$ is the map \eqref{umo}. Then
\[
H^{\le -1}_{\wedge}\be\le(k,M^{\le\prime}\e)=\krn\!\!\lbe\left[\lbe K^{\lbe D}\lbe(k)^{\wedge}\,\overset{\! v(k)^{\wedge}}{\lra} A(k)\le\right]\!.
\]

\begin{lemma} \label{mm'} There exists an exact sequence of abelian groups
\[
0\to H^{\le -1}_{\wedge}\be\le(k,M\lle)\to H^{\le -1}_{\wedge}\be\le(k,M^{\le\prime}\e)\to L(k)^{\wedge}.
\]
\end{lemma}	
\begin{proof} We have
$H^{\le -1}_{\wedge}\be\le(k,M^{\le\prime}\e)=\krn\be(v(k)^{\wedge})=\krn\be(\pi(k)^{\wedge}\circ u(k)^{\wedge})$ and the exactness of the top row of diagram \eqref{kat} shows that $\krn\be(\pi(k)^{\wedge})=L(k)^{\wedge}$. The lemma now follows by applying Lemma \ref{ker-cok} to the  pair of morphisms
\[
K^{\lbe D}\lbe(k)^{\wedge}\overset{\be u(k)^{\wedge}}{\lra}G(k)^{\wedge}\overset{\!\pi(k)^{\wedge}}{\lra}A(k).
\]
\end{proof}

\begin{lemma}\label{-1dp} There exists a continuous perfect pairing of abelian topological groups
\[
H^{\e -1}_{\lbe\wedge\be}\be\le(k,C(\partial_{\le\mathcal E}))\times H^{\le 2}(k,G\e)\to\Q/\Z,
\]
where the left-hand group is profinite and the right-hand group is discrete and torsion.
\end{lemma}
\begin{proof} Since $H^{\le 2}(k,A\lle)=0$ by \cite[proof of Theorem I.3.2 and Theorem III.7.8]{adt}\e, the sequence $\mathcal E\colon 0\to L\to G\to A\to 0$ induces an exact sequence of discrete and torsion abelian topological groups
$H^{\le 1}(k,A\lle)\to H^{\le 2}(k,L\lle)\to H^{\e 2}(k,G\e)\to 0$ whose Pontryagin dual is the bottom row of the following exact and commutative diagram of abelian groups:
\[
\xymatrix{0\ar[r]& H^{\e -1}_{\lbe\wedge\be}\be\le(k,C(\partial_{\le\mathcal E}))\ar[d]^{\lambda}\ar[r]&L^{\be D}\lbe(k)^{\wedge}\ar[d]^{\simeq}\ar[r]^{\partial_{\le\mathcal E}\lbe(k)^{\wedge}}& A^{\vee}\lbe(k)\ar[d]^{\simeq}\\
0\ar[r]&H^{\e 2}(k,G\e)^{*}\ar[r]&	H^{\e 2}(k,L\le)^{*}\ar[r]&H^{\e 1}(k,A\le)^{*}.
}
\]
The middle and right-hand vertical maps are topological isomorphisms by Theorem \ref{ros23}(i) and Milne\,--\e Tate duality, respectively. It follows that the map labeled $\lambda$ above is an isomorphism of abelian groups. Since it is also continuous and strict by Lemma \ref{lem:stc}(i) and the continuity of the middle vertical map, $\lambda$ is, in fact, a topological isomorphism. This completes the proof.
\end{proof}

Now recall the $1$-motive $M^{\le\diamond}=(L^{\be D}, u^{\lbe\lbe\vee}, \mathcal E(K,\iota^{\lbe\vee}, G^{\e\vee},\pi^{\lbe\vee}, A^{\vee}\le))$ \eqref{mdu} dual to $M$, for which
\[
\hskip 1.5cm (M^{\le\diamond})^{\e\prime}=(L^{\be D}, \partial_{\le\mathcal E}, \mathcal E(0,0, A^{\vee},{\rm Id}_{A^{\vee}}, A^{\vee}\le))
\]
by Lemma \ref{pud}\e.

\begin{theorem}\label{-1bis} Let $k$ be a non-archimedean local field and let $M$ be an arbitrary $1$-motive. Then there exists a continuous perfect pairing of abelian topological groups
\[
H^{\e -1}_{\lbe\wedge\be}\be\le(k,M^{\le\diamond}\lle)\times H^{\le 2}\lbe(k,M\le)\to\Q/\Z\e,
\]
where the left-hand group is profinite and the right-hand group is discrete and torsion.
\end{theorem}
\begin{proof} There exists an exact and commutative diagram of abelian groups
\[
\xymatrix{0\ar[r]& H^{\e -1}_{\lbe\wedge\be}\be\lle(k,M^{\lle \diamond}\lle)\ar[d]\ar[r]&  H^{\e -1}_{\lbe\wedge\be}\be\le(k,C(\partial_{\le\mathcal E}))\ar[d]^{\simeq}\ar[r]&K\lbe(k)^{\wedge}\ar[d]^{\simeq}\\
0\ar[r]&H^{\le 2}\lbe(k,M\le)^{*}\ar[r]&	H^{\le 2}\lbe(k,G\e)^{*}\ar[r]&H^{\le 2}\lbe(k,K^{\be\le D}\lle)^{*},
}
\]
where the top row is the sequence of Lemma \ref{mm'} for $M^{\le\diamond}$ and the bottom row is the dual of the exact sequence $H^{\e 2}(k,K^{\lbe\lle D}\lle)\to H^{\e 2}(k,G\e)\to H^{\e 2}(k,M\le)\to 0$ in Lemma \ref{dle}\e. The middle (respectively, right-hand) vertical map in the above diagram is a topological isomorphism by Lemma \ref{-1dp} (respectively, Theorem \ref{ros23}(ii)). It follows, as in the proof of Lemma \ref{-1dp}\,, that the left-hand vertical map is also a topological isomorphism.
\end{proof}

\end{document}